\PassOptionsToPackage{dvipsnames,svgnames}{xcolor}
\documentclass{amsart}

\usepackage{breqn}
\usepackage[utf8]{inputenc}
\usepackage{subfigure}
\usepackage{tikz}
\usetikzlibrary{shapes.multipart}
\usepackage{varwidth}
\usepackage{gensymb} 
\usepackage{placeins}
\usepackage{todonotes}
\usepackage[font=small]{caption}
\usepackage[doi=false,isbn=false,backend=biber,style=alphabetic]{biblatex}
\usepackage[normalem]{ulem}

\usepackage{microtype}

\usepackage{hyperref}
\usepackage{cleveref}

\graphicspath{{fig/}}

\numberwithin{equation}{section}

\def\eps{{\epsilon}}

\def\C{{\mathbb C}}

\def\N{{\mathbb N}}

\def\R{{\mathbb R}}

\def\U{{\mathbb U}}
\def\Z{{\mathbb Z}}
\def\CP{{\mathbb C\mathbb P}}

\def\ov{\overline}

\newcommand{\cal}[1]{\mathcal{#1}}
\newcommand{\tend}{\longrightarrow}
\newcommand{\re}{\operatorname{Re}}
\newcommand{\im}{\operatorname{Im}}
\newcommand{\res}{\operatorname{res}}
\renewcommand{\Re}{\re}
\renewcommand{\Im}{\im}

\newcommand{\setof}[2]{\left\{ #1 \,\middle|\, #2  \right\}}

\theoremstyle{plain}
\newtheorem{lemma}{Lemma}[section]
\newtheorem{proposition}[lemma]{Proposition}
\newtheorem*{proposition*}{Proposition}

\newtheorem{theorem}[lemma]{Theorem}
\newtheorem*{theorem*}{Theorem}
\newtheorem{corollary}[lemma]{Corollary}
\newtheorem{convention}[lemma]{Convention}
\newtheorem*{convention*}{Convention}

\theoremstyle{definition}
\newtheorem{example}[lemma]{Example}
\newtheorem{remark}[lemma]{Remark}
\newtheorem{definition}[lemma]{Definition}

\theoremstyle{remark}

\title{Dynamics of generic $1$-parameter perturbations of a vector field with a singular point of codimension $k$}

\begin{document}

\author[A. Ch\'eritat]{Arnaud Ch\'eritat}
\address{Arnaud Ch\'eritat, Institut de Math\'ematiques de Toulouse, Université Paul Sabatier
118, route de Narbonne,
F-31062 Toulouse Cedex 9, France.}
\email{arnaud.cheritat@math.u-toulouse.fr}

\author[C. Rousseau]{Christiane Rousseau}
\address{Christiane Rousseau, D\'epartement de
math\'ematiques et de statistique\\Universit\'e de Montr\'eal\\C.P. 6128,
Succursale Centre-ville, Montr\'eal (Qc), H3C 3J7, Canada.}
\email{christiane.rousseau@umontreal.ca}

\thanks{The second author was supported by NSERC in Canada during part of this research. The two authors thank BIRS where this research was first initiated.}

\date{\today}

\begin{abstract} 
The paper describes the bifurcation diagram of germs of generic $1$-parameter families of complex vector fields $\dot z = \omega_\eps(z)$ on $\C$, unfolding a singular point of multiplicity $k+1$: $\omega_0= z^{k+1} +o(z^{k+1})$. As a preparatory step, the bifurcation diagram of the family of vector fields $\dot z = z^{k+1}-\eps$ over $\CP^1$ is given,
through a description of its associated translation surface as a star shaped domain (a straightening coordinate, a.k.a.\ time coordinate, of a vector field defines a translation surface structure on the complement of the singularities).
In this article, we define a notion of generic families of vector fields $\dot z = \omega_\eps(z)$, and classify them up to conjugacy by a holomorphic change of coordinate and parameter.
A description of the modulus space and several (almost) unique normal forms are provided.
Then, we describe an analogue of the star shaped domain for these generic vector fields, that we call the geometric model and which allows to describe the local bifurcation diagram for generic families of vector fields in fine details and compare it with that of the simple vector field $\dot z = z^{k+1}-\eps$ restricted to $B(0,r)$.
\end{abstract}

\maketitle

\pagestyle{myheadings}\markboth{A. Ch\'eritat and C. Rousseau}{Generic $1$-parameter perturbations of a singular point  of codimension $k$}

\let\oldtocsection=\tocsection
\let\oldtocsubsection=\tocsubsection
\let\oldtocsubsubsection=\tocsubsubsection

\renewcommand{\tocsection}[2]{\hspace{0em}\oldtocsection{#1}{#2}}
\renewcommand{\tocsubsection}[2]{\hspace{1em}\oldtocsubsection{#1}{#2}}
\renewcommand{\tocsubsubsection}[2]{\hspace{2em}\oldtocsubsubsection{#1}{#2}}

\tableofcontents

\section{Introduction}\label{sec:intro}

Consider a germ of vector field $\dot z = \omega(z)$ defined in a neighborhood of $0$ in $\C$ and such that
\begin{equation}
\dot z = \omega(z)=az^{k+1} + O(z^{k+2})\label{vector_field0}
\end{equation}
for some $k\geq 1$ and $a\in\C^*$.
The singular point $z=0$ of the vector field has multiplicity $k+1$.
It is also called a \emph{parabolic point} of multiplicity $k$ (or of codimension $k$, since
this is a condition of codimension $k$ on the space of vector fields with a singularity at $0$).
The system is not structurally stable and the multiple singular point typically splits into several singular points when perturbing the system.\footnote{In the case $k=0$, i.e.\ $\omega(0)=0$ and $\omega'(0)\neq 0$, the vector field and its perturbations are uniformly locally linearizable, so the classification is simpler.} 
To study all possible behaviors (phase portraits) it is natural to embed the vector field in a generic $k$-parameter unfolding. This has been done by Kostov \cite{K} who provided a simple normal form
\begin{equation} 
\dot z = \frac{z^{k+1} + \eps_{k-1} z^{k-1} + \dots + \eps_1z+\eps_0}{1+a(\eps)z^k}.\label{Kostov}
\end{equation}
It is nearly polynomial, save for the term with $a(\epsilon)$, which cannot be removed.
The Kostov normal form is obtained by a change of coordinate and multi-parameter. 
In \cite{RT} it is shown that the multi-parameter is almost unique, the only degree of freedom coming from rotations in $z$ of order dividing $k$. However, in practice, it is quite common to encounter $1$-dimensional perturbations $\omega_\eps$ of \eqref{vector_field0}. Generically, such a perturbation satisfies $\frac{\partial \omega_\eps}{\partial \eps}\neq0$. What are the possible phase portraits and bifurcations occurring in such perturbations? This is the question we address in this paper.
Note that we are interested in the local behavior near $0$.

One of our hopes is that the fine description done here will be a useful tool for mathematicians to understand and refine the study of generic one-parameter perturbation of parabolic points in one-dimensional complex dynamics, i.e.\ when the time is discrete.

\medskip

In \cref{sec:local_vect} we show that a change of coordinate independent of the parameter allows bringing a germ of generic $1$-parameter perturbation of \eqref{vector_field0} to the form
\begin{equation}\label{eq:pos:0}
\dot z = (z^{k+1} -\eps) u(z,\eps),
\end{equation}
with $u(0,0) \neq 0$. In this form, the
singular points are given by the roots of $z^{k+1}-\eps$. 
A further \emph{linear} change of variable and of parameter allows to assume that $u(0,0)=1$.

Coming back to the initial coordinate we immediately see that, in a generic perturbation, the singular points are approximately located at the vertices of a regular $(k+1)$-gon.
If $u$ were constant, the vector field would be conjugated by a linear change of variable to $\dot z=z^{k+1}-\eps'$ for some $\epsilon'$.
It is hence natural to study, as a first step, the family of vector fields $\dot z=z^{k+1}-\eps$, its phase portraits and bifurcation diagram, as a model for the behavior of the family $\omega_\epsilon$.
We give the bifurcation diagram of this simplified family over $\CP^1$ in \Cref{sec:ideal:global}.
A straightening coordinate (a.k.a.\ time coordinate) of a vector field defines a translation surface structure on the complement of the singularities.
We describe in \Cref{sec:ideal:global} the translation surface associated to $\dot z=z^{k+1}-\eps$, as a star shaped domain, and use it to determine the bifurcation diagram.

As a second step, we explore in \Cref{sec:local_vect} two equivalence problems solved analytically by Rib\'on in \cite{Ri} in a more general setting:
\begin{enumerate}
\item {\bf Equivalence Problem 1:} When are two generic perturbations $\omega_\eps$ of \eqref{vector_field0} conjugate under an analytic change of coordinate?
\item {\bf Equivalence Problem 2:} When are two generic perturbations $\omega_\eps$ of \eqref{vector_field0} conjugate under an analytic change of coordinate and parameter?
\end{enumerate} 
In both cases, the change of coordinate is allowed to depend on $\eps$.
As noted by Rib\'on a necessary and sufficient condition is that the eigenvalues be the same. 
 
In comparison, Rib\'on's paper addresses the very general case of an unfolding of a germ of vector field with zero $1$-jet at the origin, including the zero vector field. His unfoldings can depend on several parameters and are not necessarily generic.  We complete his  work in the $1$-parameter generic case by identifying  the moduli space and showing a realization theorem.
The very general and abstract statement of Rib\'on's result makes it difficult to derive the simple statement in the generic case without reading a significant part of his paper.

We start by studying a variation of the equivalence problem.
We focus on generic $1$-parameter perturbations, which we call \emph{unitary}, such that $\partial \omega_\eps / \partial \eps = -1$ at $(z,\eps)=(0,0)$ and $a=1$ in $\omega_0(z) = az^{k+1} + o(z^{k+1})$, i.e.\ %
\begin{equation}\label{eq:uni}
\omega_\eps(z) = z^{k+1}-\eps + \cal O(z^{k+2},\eps z,\eps^2)
.
\end{equation}
We look at local changes of variables and parameters, i.e.\ germs of analytic diffeomorphisms near $0$:
$\Psi: B(0,r)\times B(0,\rho) \rightarrow \C^2$,
$(z,\eps)\mapsto (\psi_\eps(z),h(\eps))$
such that $\Psi(0,0)=(0,0)$ 
and such that $\Psi$ is \emph{unitary}, i.e.\ its differential at $(0,0)$ is the identity, or  \emph{unipotent} i.e.\ of the form
$
D_{(0,0)} \Psi = \left[\begin{array}{cc} 1 & \bullet \\ 0 & 1 \end{array}\right]
$
where the dot stands for any complex number.
We say that it is \emph{parameter fixing} if $h(\eps)=\eps$.
The questions we address are:
\begin{enumerate}
\item {\bf Equivalence Problem 1':} When are two unitary generic perturbations $\omega_\eps$ of \eqref{vector_field0} conjugate under a parameter fixing unitary/unipotent change of coordinate?
\item {\bf Equivalence Problem 2':} When are two unitary generic perturbations $\omega_\eps$ of \eqref{vector_field0} conjugate under a unitary or unipotent change of coordinate and parameter?
\end{enumerate}
The classification under unitary or unipotent changes turns out to be identical (see \Cref{lem:0fix}).
We address the two equivalence problems above in \Cref{sub:chvar_roots,sub:eig}.
In \Cref{sub:ppal}, we deduce the answer to the initial equivalence problems from the answers to the problems above.
Moreover, we give in \Cref{sub:chvar_roots} additional normal forms and, for all normal forms, including that of Rib\'on, we prove their uniqueness up to a rotation of order $k$ of the coordinate. These in turn provide \emph{canonical parameters}, which become analytic invariants. The details are as follows.
 
For equivalence problems 1' and 2' we introduce an invariant in the form of an \emph{eigenvalue function}. 
It is a function $\lambda$ of $\delta=\eps^{1/k+1}$ which contains all the information about the eigenvalues at each singular point.
More precisely, we first put the family in the form \cref{eq:pos:0} with a unitary change of variable independent of the parameter.
Then the $k+1$ eigenvalues at the singular points of $\omega_\eps$ are given by the $\lambda(\delta_j)$, where $\delta_1,\dots, \delta_{k+1}$ are the solutions of $\delta^{k+1}=\eps$.
The function $\lambda$ vanishes at $\delta=0$ with order precisely $k$ and $\lambda(\delta) \sim (k+1)\delta^k$ as $\delta\to 0$. Moreover, any holomorphic germ equivalent at $0$ to $(k+1)\delta^k$ arises as the eigenvalue function of some family of the form \cref{eq:uni}.
If the family $\omega_\eps$ is not assumed unitary, there are $k+1$ eigenvalue functions associated to it, any two being equal up to right composition with a rotation of order dividing $k+1$.

For Equivalence Problem 1', we show that two generic families $\omega_\eps$ are equivalent if and only if their eigenvalue functions are equal and for Equivalence problem 1, if and only if their eigenvalue functions are equivalent  up to right-composition by a rotation of order  dividing $k+1$.
Equivalence Problems 2 and 2' are later reduced to Equivalence Problem 1 and 1' modulo a change to a \emph{canonical parameter}.
In addition, we show that $\omega_\eps$ is conjugate to the form
\[\dot z = (z^{k+1}-\eps)\sigma(z)\]
by a change of coordinate preserving the parameter, for some analytic function $\sigma$  independent of $\eps$ with $\sigma(0)\neq 0$.

For Equivalence Problems 2 and 2', it is possible to describe the action of a change of parameter on $\lambda$. In particular, we can bring the eigenvalue function to an almost unique normal form:
\begin{equation}\label{form_sigma}
\lambda(\delta)= (k+1)\delta^k\sigma(\delta)= (k+1)\delta^k\left(1+\sum_{j=1}^k a_j(\delta^{k+1})\delta^j\right)
,
\end{equation}
i.e. $\sigma(0)=1$ and the power series expansion of $\sigma$ contains no terms of degree $m(k+1)$ for $m\geq1$.
There is only one degree of freedom, and only if non-unipotent changes of variables and parameters are allowed (Equivalence Problem 2): it comes from substituting $\nu\delta$ to $\delta$ in \eqref{form_sigma} with $\nu$ a $k$-th root of unity.
A choice of parameter for which the eigenvalue function has this form is called \emph{canonical}: it is unique (resp.\ unique up to $\eps\mapsto \nu \eps$, $\nu$ a  $k$-th root of unity if non-unipotent changes of variables and parameters are allowed).
The modulus space is exactly the space of germs of analytic functions of the form \eqref{form_sigma}, (up to this equivalence relation if non-unipotent changes of variables and parameters are allowed).
A second consequence is  that a generic perturbation $\omega_\eps$ of the form \cref{eq:uni} with an eigenvalue function having normal form \eqref{form_sigma} is conjugate by a unitary change of coordinate and parameter to the unique normal form

\begin{equation} \dot z = (z^{k+1}-\eps)\sigma(z),\end{equation}
where the function $\sigma$ does not depend on $\epsilon$.
If non-unipotent changes of variables and parameters are allowed, the only degree of freedom comes from $(z,\eps)\mapsto(\nu z, \nu\eps)$ for $\nu$ a $k$-th root of unity.

In \Cref{sec:geom}, we describe an isomorphic model, which we call the \emph{geometric model} of the translation surface associated to $\omega_\eps$ on $B(0,r)$, inspired by the star shaped model for the simple vector field $\dot z = z^{k+1}-\eps$ of \Cref{sec:ideal:global}.

In \Cref{sec:bif2}, we define the bifurcation locus for the vector field family on $B(0,r)$ as the set of $\eps$ corresponding to vector fields that are not structurally stable. A vector field $\omega_{\eps_0}$ on B(0,r) is structurally stable if all vector fields $\omega_{\eps}$ for $\eps$ in a neighborhood of $\eps_0$ are conjugated to $ \omega_{\eps_0}$  under a homeomorphism, up to a reparametrization of time (we also say that $\omega_{\eps}$ is \emph{topologically orbitally equivalent} to $\omega_{\eps_0}$).
The regions of structural stability are of two 
types:
\begin{itemize}
\item $2k$ sectorial regions in parameter space 
where any singular point is linked to any other one by a finite chain of trajectories inside $B(0,r)$,
\item separated by thin regions where there exist trajectories crossing the disk and separating the singular points into subgroups with no trajectory between disjoint subgroups.
\end{itemize}
The vector field is tangent to the boundary $\partial B(0,r)$ at $2k$ points. The bifurcation diagram consists in $4k^2$ real analytic arcs starting at $\eps=0$, tangent to $2k$ equidistant directions at this point, one for each pair of tangency points: on such an arc there exists  an oriented trajectory tangent to $\partial B(0,r)$ at the two points of the pair.

\section{Study of the polynomial vector field \texorpdfstring{$\dot z= z^{k+1}-\eps$ on $\CP^1$}{z = z\^{}\{k+1\} - epsilon on CP1}}\label{sec:ideal:global}

This particular case will serve as an inspiration for the more general case.

Let us stress that we are interested in \emph{real-time} trajectories, with in mind an application to discrete dynamical systems.
Also, we start with a global study, i.e.\ on $\CP^1$, as opposed to a local study near $0$ that we will deduce later in \Cref{sub:ideal:case}.
This section is deeply inspired by \cite{DES} both for the terminology and the spirit.

\begin{figure}[htbp]
  \begin{center}
    \subfigure[]{
      \includegraphics[width=5cm]{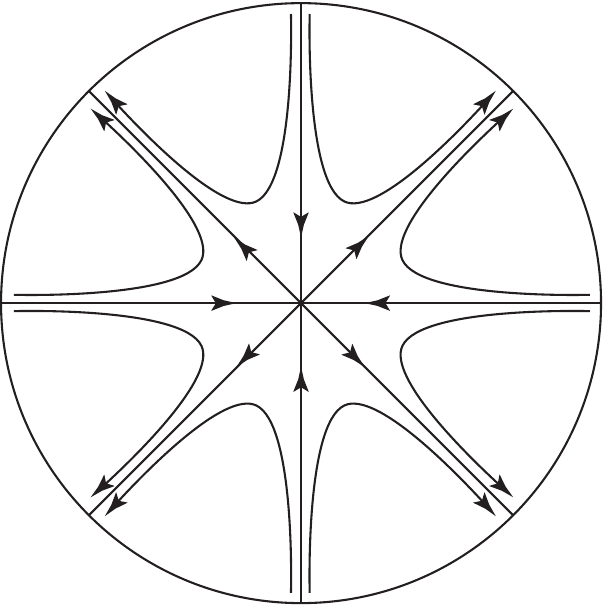}
    }
    \qquad
    \quad
    \subfigure[]{
      \includegraphics[width=5cm]{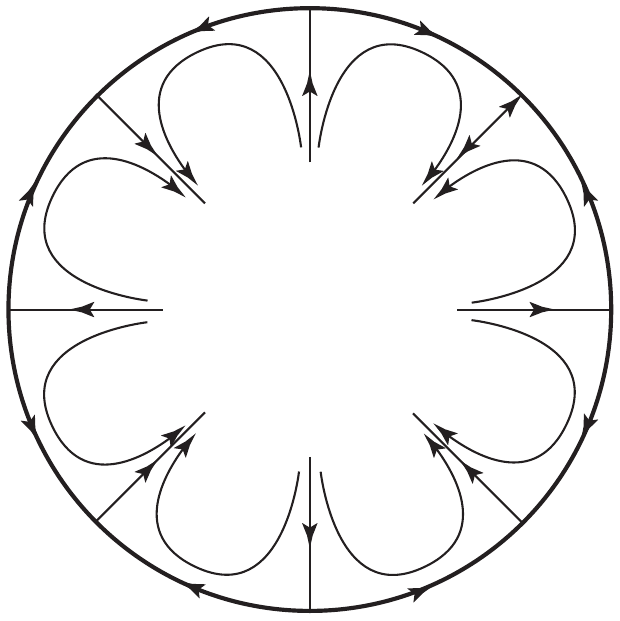}
    }
    \caption{\label{fig:X}
      Case $k=4$. In (a) The phase portrait near $\infty$ after the change $z\mapsto Z=1/z$ and mutiplication by $|Z|^{2(k-1)}$. In (b) the blow-up near $\infty$ after i) passing to the coordinates $(r,\theta)$ where $z=\frac{e^{i\theta}}{r} $, ii) multiplying by $r^k$, and iii) making the change $r\mapsto 1-r$.
    }
\end{center}
\end{figure}

The point at infinity is a pole of of order $k-1$ of the vector field: it has $k$ incoming and $k$ outgoing separatrices, see \Cref{fig:X}(a).
Blowing up the pole at infinity\footnote{In the blow up we replace the pole by a circle whose points are the directions at infinity.}, it is common to represent the neighborhood of infinity as the interior neighborhood of a circle as in \Cref{fig:X}(b).
In drawing the phase portraits we will indifferently draw them on the plane or on a disk where the boundary is invariant under the flow and represents the blow-up of infinity\footnote{A geometric visualization is obtained by projecting the plane on a sphere tangent to the plane at the origin, through the center of the sphere. The plane is projected on the lower hemisphere and the equator represents the circle of directions at infinity. The vector field can be extended to a line field along the equator. For the line field, the equator is invariant and has $2k$ singular points, which look like topological saddles but are reached in finite time. Slowing down the vector field by multiplying it by a real (non holomorphic) function $f(|z|)$ tending to $0$ at the right speed turns them into actual saddles}.

The polynomial vector field 
\begin{equation}\dot z= \frac{dz}{dt}= P_\eps(z)=z^{k+1}-\eps\label{vf}\end{equation} has a pole of order  $k-1$ at infinity (a regular point if $k=1$) and $2k$ \emph{separatrices}, alternately ingoing and outgoing,\footnote{It may happen that some outgoing separatrices are joined to some ingoing ones, i.e.\ they form a homoclinic loop.} i.e.\ trajectories that tend to $\infty$ in finite positive or negative time, with asymptotical directions $\arg{z}= \frac{2\pi j}{k}$ for $j=0, \dots, 2k-1$. The family is invariant under 
\begin{equation}
(z,\eps,t)\mapsto\left (e^{\frac{\pi i}{k}}z,-e^{\frac{\pi i}{k}}\eps, -t\right),\label{sym_rotation}
\end{equation}
and under
\begin{equation}
(z,\eps)\mapsto (\ov{z},\ov{\eps}).\label{sym_real}
\end{equation}
Note that in the particular case when $k$ is odd, \eqref{sym_rotation} yields that the family is invariant under
\begin{equation}
(z,\eps)\mapsto (-z,\eps, -t).\label{sym_imag}\end{equation}
For nonzero $\eps$, using the rescaling  
\begin{equation}\label{rescaling}
(z,\eps,t)\mapsto \left(|\eps|^{-\frac1{k+1}}z, \frac{\eps}{|\eps|},|\eps|^{\frac{k}{k+1}}t\right),
\end{equation}
we can reduce the study of the vector field to the case $|\epsilon|=1$. Hence,  we now suppose that
\begin{equation}
\eps=e^{i\theta}, \qquad\theta\in [0,2\pi].\label{def:eps}
\end{equation}
The singular points given by
$$z_\ell=e^{\frac{i(\theta+2\ell\pi)}{k+1}}, \qquad \ell=0, \dots, k,$$
are the vertices of a regular $(k+1)$-gon, and their eigenvalues are
$$P_\eps'(z_\ell)= (k+1) z_\ell^k=(k+1)e^{\frac{ik(\theta+2\ell\pi)}{k+1}}= (k+1)e^{\frac{ik\theta}{k+1}-\frac{2\ell\pi i}{k+1}}$$
using $\frac{2\pi k\ell}{k+1} \equiv -\frac{2\pi \ell}{k+1}, \bmod\, 2\pi$. In particular, the circular ordering of the eigenvalues is reversed as compared to that of $z_\ell$.
Let $$\kappa_\ell=\frac1{P_\eps'(z_\ell)}= \text{Res}\left(\frac1{P_\eps}, z_\ell\right),$$
which are ordered as the $z_\ell$.
\bigskip

\begin{figure}[htbp]
\begin{center}
\includegraphics[width=9cm]{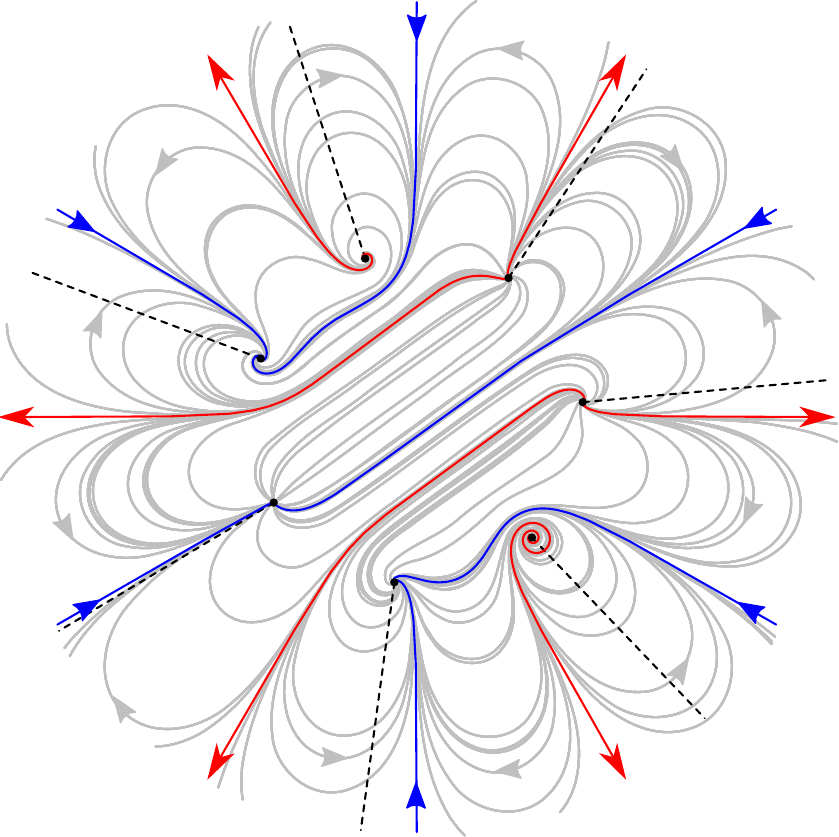}
\end{center}
\caption{Example of phase portrait for the vector field $\dot z=z^7-\epsilon$, with $\epsilon = e^{i2\pi/10}$. The seven black dots indicate the position of the singularities. Dark blue lines are incoming separatrices (field lines coming from infinity). Red lines are outgoing separatrices. In gray, we drew a random sample of other trajectories.}
\label{fig:f7}
\end{figure}

When $k=1$, infinity is a regular point of the vector field on $\CP^1$, but we will still speak of separatrices for the incoming and outgoing trajectories through $\infty$.
We will call \emph{homoclinic loop} a periodic loop on $\CP^1$ through $\infty$.

When $\eps=0$, the $2k$ separatrices of $\dot z=z^{k+1}$ coincide with some trajectories coming from or going to the parabolic singular point $0$, in infinite time: they form alternating and equally spaced straight half-lines from $0$ to infinity. They are in this case also called the repelling (coming from $0$) or attracting (going to $0$) axes; the positive real axis is a repelling axis. The situation is symmetric by a rotation of order $k$.

When $\eps\neq 0$, the $k+1$ singular points $z_\ell$ fit in this picture in a way that cannot be invariant by the same symmetry. There is a pleasant geometrical way of figuring out the argument of the eigenvalues (from which one can for instance deduce if the singular points are attracting)  according to the placement of these points with respect to the axes. Indeed since $P'_\eps(z) = P_0'(z) = (k+1)z^k$ it follows that an eigenvalue is positive and real if and only if the singularity $z_\ell$ of $P_\eps$ is on one of the half-lines that were repelling axes for $P_0$. It is real negative iff $z_\ell$ is on an attracting axis. And it is imaginary iff $z_\ell$ is equidistant from a neighborhing pair of repelling/attracting axes of $P_0$. See \Cref{fig:zones}.
We recall that the real part of the eigenvalue $\lambda$ of a singular point characterizes its attracting ($\Re\lambda<0$), neutral ($\Re\lambda=0$) or repelling ($\Re\lambda>0$) nature.
It turns out there is another simple characterization of the nature of the singular points of $\dot z = P_\eps(z)$:
indeed at a singularity $z$, $\lambda =P_\eps'(z) = (k+1)z^k = (k+1)\eps/z$, so its real part is $<0$ if and only if $\Re(z/\eps) <0$, if and only if $z$ is in the half-plane containing $-\eps$ and delimited by the line through $0$, orthogonal to $\eps$.
Note that $-\eps = P(0)$.
It may come as a suprise that, ordering the singularities according to their arguments, the attracting ones are consecutive, and the repelling ones too. This will follow naturally from the star-shaped translation surface model of the vector field we describe in \Cref{sec:geom}.

\begin{figure}[htbp]
\begin{center}
\begin{tikzpicture}
\node at(0,0){\includegraphics[scale=.4]{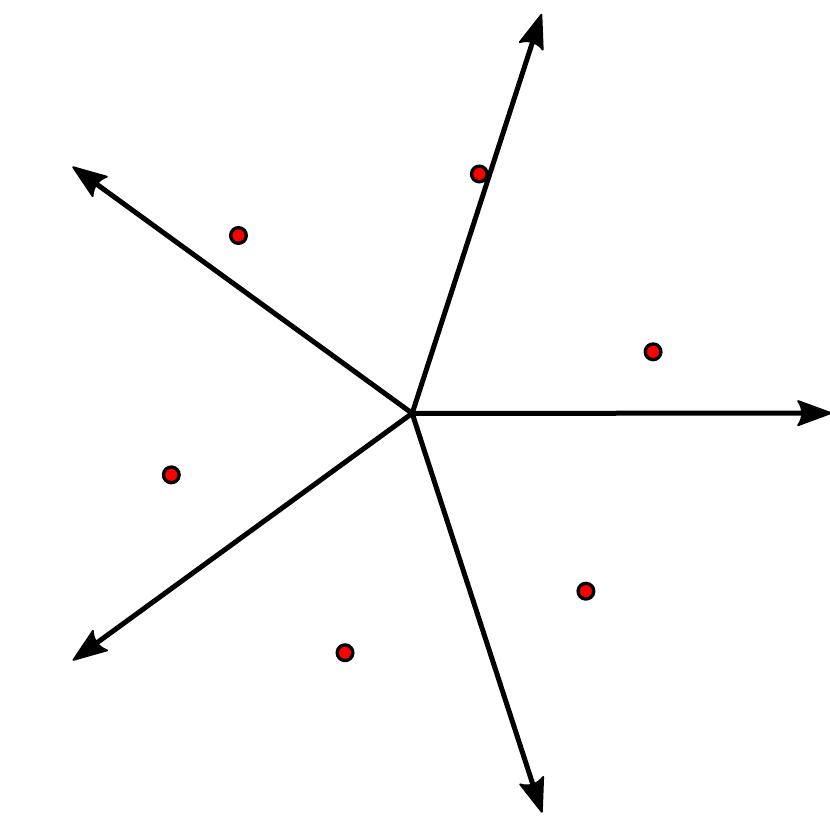}};
\node at(6.7,0){\includegraphics[scale=.4]{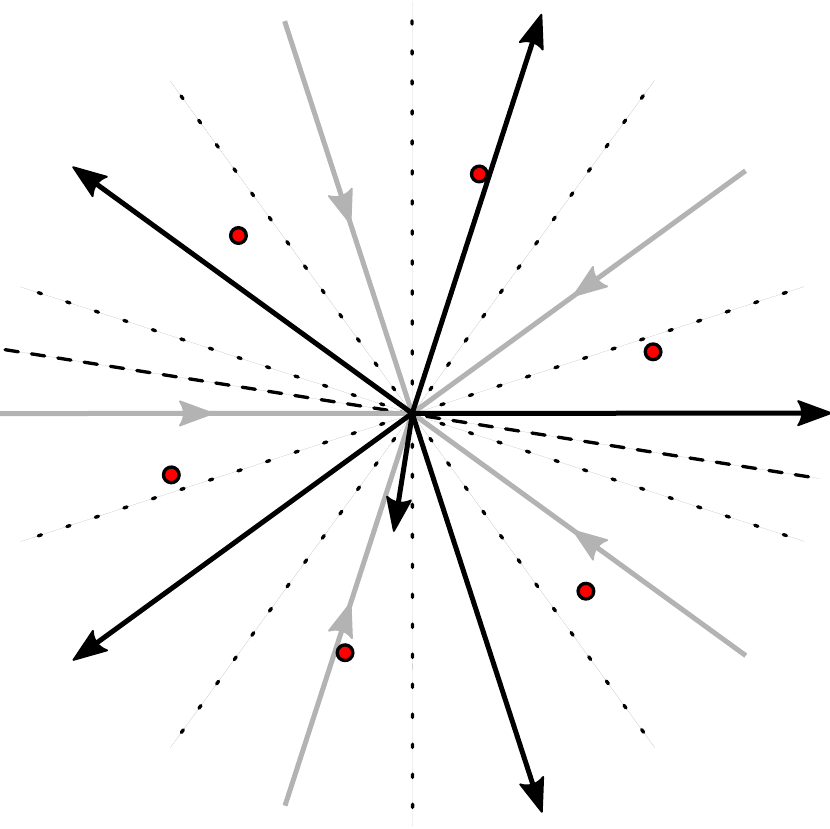}};
\end{tikzpicture}
\end{center}
\caption{A possible placement of the singular points of $\dot z = z^6-\eps$ for some $\eps$, compared to the position of the axes of $\dot z = z^6$. Left: repelling axes of $\dot z=z^6$ only. Right: added attracting axes of $\dot z=z^6$ in gray and intermediate axes in dotted style. A singular point that is closer to an attracting axis is attracting, and it is repelling if it is closer to a repelling axis. The attracting singular points are on the side containing $P(0)=-\eps$ (the short arrow) of the half-plane delimited by the dashed line.}
\label{fig:zones}
\end{figure}

Generically, for all but finitely many values of $\theta$ (recall $\theta=\arg \eps$), the separatrices of $\infty$ land at the singular points. Indeed, if a separatrix does not land, then it has to come back to infinity, thus forming a homoclinic loop $\gamma$: this loop separates  a group of singular points $\{z_\ell\}_{\ell\in I}$ from the other singular points.
It follows that $\sum_{\ell\in I} \kappa_\ell\in i\R$.
This comes from the residue theorem. Indeed, let $T$ be the travel time along $\gamma$ (note that separatrices reach infinity in finite time, because $\infty$ is either a pole (if $k>1$) or a regular point (if $k=1$) of the vector field). Then, by the residue theorem 
$$T=\int_\gamma dt= \int_{\gamma}\frac{dz}{P_\eps(z)} = \pm 2\pi i \sum_{\ell\in I} \kappa_\ell \in \R.$$
Note that the converse does not necessarily hold and it could happen that $\sum_{\ell\in I} \kappa_\ell\in i\R$ without having the corresponding group of singular points separated by a (homoclinic) loop: this will be clear once we will have determined exactly when homoclinic loops occur in this family, see \Cref{thm:bif_vf_simpl}.

\bigskip

Let us call \emph{singular-gon} the regular $(k+1)$-gon formed by the singularities of the vector field.
The \emph{period} of a singular point $z_\ell$ is $\mu_\ell$ defined by
\[ \mu_\ell=2\pi i\kappa_\ell=\frac{2\pi i}{P_\eps'(z_\ell)}
.\]
If $z_0, \dots, z_k$ is the order of the vertices of the singular-gon when turning in the positive direction, then $\mu _0, \dots \mu_k$ is the cyclic order of  the arguments of the periods when turning in the positive direction too.

We define the \emph{period-gon}, this time not using $\mu_\ell$ as vertices but as edge vectors, and with a minus sign. To be more precise, let us use $\frac12+\Z/(k+1)\Z$ as an indexing set for the vertices: there is a unique (convex) regular $(k+1)$-gon centered at the origin with vertices $v_{\ell+\frac12}$, such that
\[v_{\ell+\frac12}-v_{\ell-\frac12} = - \mu_\ell
.\]
Notice the minus sign in front of $\mu_\ell$.

The singular-gon rotates at speed $\frac1{k+1}$ when $\theta$ moves at speed $1$, while
the periods, and thus the period-gon, turn at negative speed $-\frac{k}{k+1}$. When $\theta$ makes a full turn, the singular-gon turns by $2\pi/(k+1)$, and since it is invariant by this rotation, it returns as a set to its initial position with the indexing of its vertices having shifted by $-1$. Meanwhile, the period-gon turns by $-2\pi k/(k+1) \equiv 2\pi/(k+1) (\bmod2\pi)$. So it is also unchanged as a set and undergoes exactly the same shift on indices as the singular-gon. This is coherent with the fact that the cyclic order of the singularities and the cyclic order of the argument of their respective periods are the same.

\begin{definition}\label{def:generic}
A polynomial  vector field  $\dot z = P(z)$ is \emph{generic} if its singular points are simple and there are no homoclinic loops.
\end{definition}

\begin{remark}
A complex 1-dimensional vector field of degree at least $2$ is generic if and only if it is structurally stable among complex $1$-dimensional vector fields of the same degree.\footnote{If a field is generic, then its topology is described by the combinatorial invariant, see \Cref{def:combi-inv}. Since every neighboring field has the same combinatorial invariant, it is topologically conjugate to it.
\\
If a field is non-generic, then there are neighboring fields that do not have the same topology, either because a multiple singular point is separated into several singular points, or because a homoclinic loop is broken.}
\end{remark}

Douady, Estrada and Sentenac \cite{DES} classified the monic generic polynomial vector fields of degree $k+1$.  They are completely characterized by an analytic invariant given by $k$ complex numbers with positive imaginary parts and a combinatorial invariant. We describe these now.

\begin{definition}\label{def:combi-inv}
Let $\dot z = P(z)$ be a \emph{monic} generic polynomial vector field in the sense of \Cref{def:generic}. Then the $2k$ separatrices of $\infty$ land at singular points. \begin{enumerate}
\item The \emph{Douady-Estrada-Sentenac combinatorial invariant} is the union of (see \Cref{DS_invariant}(a)):
\begin{itemize} \item the \emph{tree graph}, an embedded graph in the oriented plane (i.e. up to an orientation preserving homeomorphism) defined as follows: for the vertices, we take the set of singular points and, whenever there is trajectory joining two singular points, we choose one of them as an edge.
\item the information on how one separatrix of $\infty$ is attached to it (then the attachment of all other separatrices is determined).
\end{itemize}
\item Let the \emph{travel time} along a curve $\gamma$ be defined as $\int_\gamma dt= \int_\gamma\frac{dz}{P(z)}$. The \emph{Douady-Estrada-Sentenac analytic invariant} is the $k$-tuple of travel times $\chi_1, \dots, \chi_k$ along $k$ curves (the dotted curves in \Cref{DS_invariant}(b)) disjoint from the separatrices, each one going from $\infty$ to $\infty$ and cutting one edge of the Douady-Estrada-Sentenac combinatorial invariant: the direction of the trajectory is chosen so that the $\chi_j$ have positive imaginary part. (Since the curves are closed and the differential form is holomorphic at $\infty$, the travel time can be calculated by the residue theorem.) \end{enumerate}
\end{definition}

\begin{figure}[htbp]
\begin{center}
\subfigure[]{\includegraphics[width=5cm]{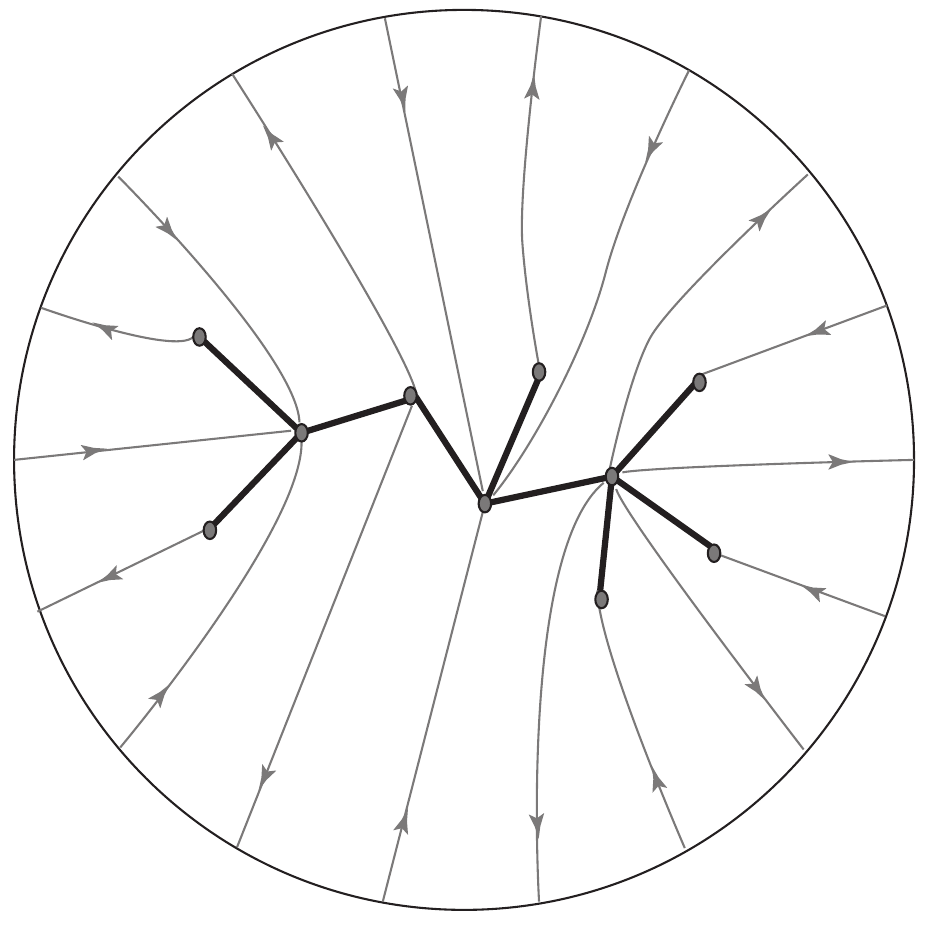}}\qquad\quad\subfigure[]{\includegraphics[width=5cm]{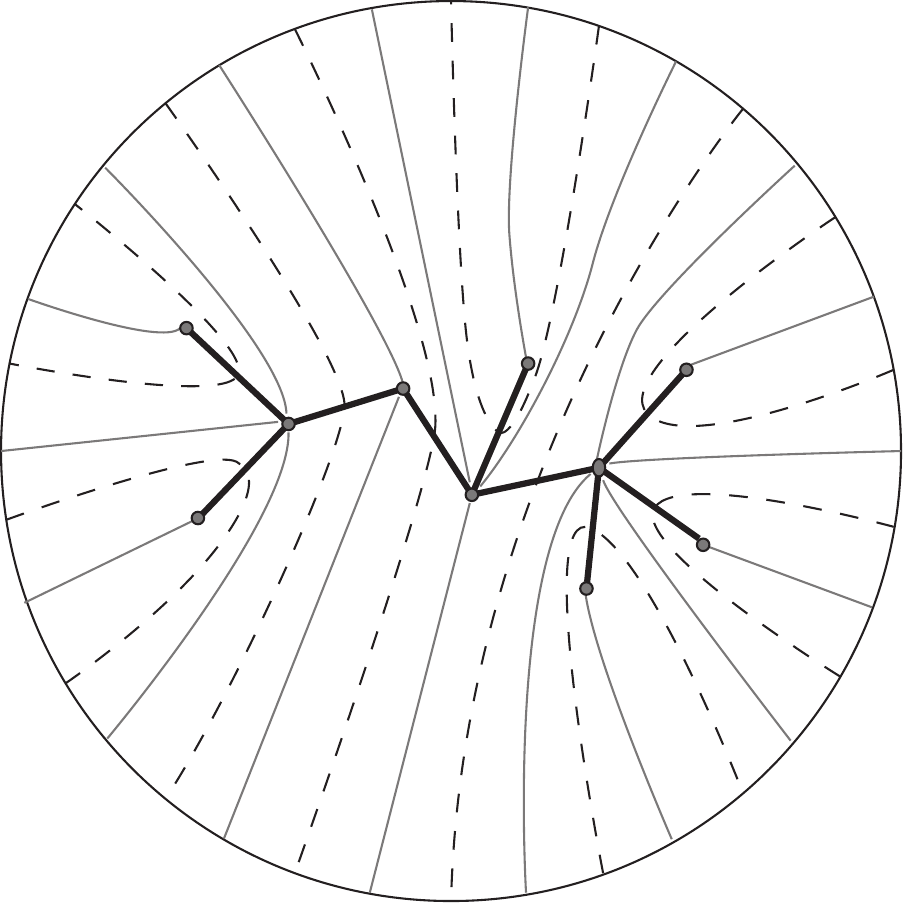}}
\caption{The Douady-Estrada-Sentenac invariant.}
\label{DS_invariant}
\end{center}
\end{figure}

In the following theorem will appear a special kind of linear ordering (i.e.\ a total order) on the set of singularities. A total order on a finite set $E$ of cardinal $m$ is equivalent to the data of a uniquely defined order preserving bijection $L : \{0,\ldots,m-1\}\to E$. Similarly, a circular order is defined by a bijection $L$ from the set $\U_m$ of $m$-th roots of unity to $E$, but this time it is unique only up to composition of $L$ by a rotation $\U_m\to\U_m$. For any $\eps\neq 0$, the set of $(k+1)$-th roots of $\eps$ is naturally circularly ordered.

\begin{definition}
Consider the set $\U_m$ of $m$-th roots of unity.
Consider the rotation $R_\theta:z\mapsto e^{i\theta} z$ followed by the projection $z \mapsto \re z$ to the real line.
Assume that this composition is injective on $\U_m$.
The order of $\R$ induces via this projection a linear ordering on $\U_m$, which we call a \emph{zig-zag ordering}.
We call \emph{zig-zag graph} the abstract graph whose set of vertices is $\U_m$ and which has an edge between two vertices if and only if they are consecutive in the linear ordering.\footnote{This kind of graph is an instance of a \emph{trunk}, i.e.\ a finite tree with no branch.}
See Figure~\ref{fig:zig-zag}.
This notion transfers to circularly ordered sets.
\end{definition}

Consider $\theta$ (non unique) as above. If $m\geq 3$, then no points in $R_\theta(\U_m)$ can be on the real line, for otherwise
there would be points in $R_\theta(\U_m)$ with the same real part. The real line cuts the unit circle into two halves.
The zig-zag ordering of $R_\theta(\U_m)$ alternates between these two halves, because the points in $R_\theta(\U_m)$ cuts the circle into arcs of equal angular span and each half-circle projects monotonically to the real line. Advancing two steps along the zig-zag order, one follows the points of the lower half-circle in the positive circular orientation, and the points it the upper half are followed in the negative circular orientation.
A zig-zag graph is thus also obtained as follows: choose an initial vertex and an initial direction on the circle; walk from the initial point in that direction until one meets another vertex; link these two vertices; walk in the opposite direction along the circle until one meets the next free point (free points are those not yet linked); link that free point to the most recent linked point; walk in the opposite direction and repeat until there are no more free points. \footnote{There are arithmetical characterization of zig-zag orderings. For instance on $\Z/m\Z$ these orderings take the form $n \mapsto n_0 + s \frac {1-(-1)^{n} (2n+1)}4 (\bmod m)$ with $s=1$ or $s=-1$ and $n_0\in\Z/m\Z$. There are also topological characterizations: for instance ``for any adjacent pair of points for the linear order, cut the circle at these points to form two arcs. Then the successors must be all on one side and the predecessors all on the other side''. And many others.}

\begin{figure}[htbp]
\begin{center}
\begin{tikzpicture}
\node at (0,0) {\includegraphics[width=10cm]{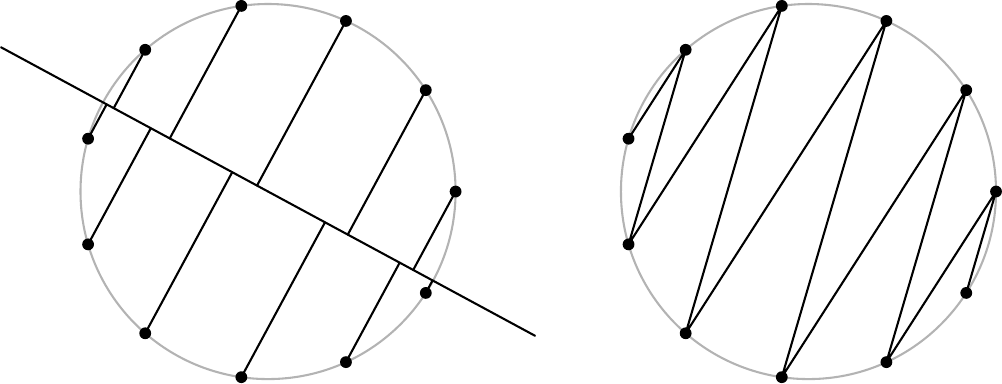}};
\end{tikzpicture}
\end{center}
\caption{The orthogonal projection of points of $\U_{11}$ to a straight line defines the zig-zag graph shown on the right.
}
\label{fig:zig-zag}
\end{figure}

\begin{theorem}\label{thm:bif_vf_simpl}
 We consider the polynomial vector field $\dot z = z^{k+1} -\eps$ with $\epsilon=e^{i\theta}$.
\begin{enumerate} 
\item In the absence of homoclinic connection, the Douady-Estrada-Sentenac combinatorial invariant (DES) is a zig-zag graph through the singular-gon: if $k>1$ it is given by the orthogonal projection of the singular-gon to the line of argument $\theta-\pi/2$ or, equivalently, the real parts of the periods of the corresponding singular points\footnote{Not to be confused with the period-gon.}. If $k=1$ there is only one possible graph.
\item If $k>1$ the following are equivalent:
\begin{itemize}
\item there are one or several homoclinic loop bifurcations,
\item the imaginary axis is a symmetry axis of the period-gon,\footnote{
Recall that the period-gon was defined as the polygon centered on $0$ and whose \emph{edges} are given by the negated periods $-\mu_\ell$.}
\item the vector field has an antisymmetry axis (a.k.a.\ \emph{axis of reversion}),
\item the singular-gon is symmetric with respect to one of the intermediate axes (i.e.\ a bissector of an attracting and an adjacent repelling axis, represented as dotted axes in \Cref{fig:zones}),
\item the singular-gon is symmetric with respect to the line through $0$ orthogonal to $\eps$,
\item the line through $0$ orthogonal to $\eps$ is an intermediate axis.
\end{itemize}
There are $2k$ such bifurcations occurring for 
\begin{equation}\label{theta_j}
\theta = \theta_j= \begin{cases} \frac{j\pi}{k}, & k\:{\rm odd,}\\
\frac{\pi}{2k} +\frac{j\pi}{k}, & k\:{\rm even,}\end{cases}
\end{equation}
for $j=0, \dots, 2k-1$.
\item If $k=1$ the homoclinic loop bifurcation occurs precisely at the value of $\epsilon$ such that the period-gon (which is a segment) is horizontal, which is equivalent to $\theta \equiv \pi \bmod 2\pi$.
\item (See \Cref{homoclinic1,homoclinic3}.) If $k>1$ then, across each bifurcation, the DES combinatorial invariant is modified as follows. Some of the segments are preserved and some are replaced. Preserved and replaced segments alternate.
For a given zig-zag invariant, there are exactly two ways to choose alternating segments, each choice corresponds to one end of the interval of structural stability in $\theta$-space with the given invariant.
Once a choice is made, modify the linear ordering of the elements by exchanging the two elements of each preserved segment, while keeping the order between the segments.
This reattaches the whole chain and gives the DES invariant on the other side of the bifurcation.
\end{enumerate}\end{theorem}
\begin{proof}
This follows from a description of the vector field in the rectifying coordinate (a.k.a.\ time coordinate)
\begin{equation}
t= \int_0^z\frac{du}{P_\eps(u)}.\label{coor_t}
\end{equation}
Such an integral is well defined, as a function of $z$, on $\C$ minus radial cuts from the singular points to $\infty$ (see \Cref{fig:f6}). 
In this coordinate, the vector field simply becomes $\dot t =1$, and the trajectories are horizontal lines. The vertices of the period-gon in $t$-space are $k$ images of $\infty$ in $z$-space. 
The singular points in $z$-space are sent to $\infty$ in $t$-space. The image of $\C$ minus the radial cuts is a \emph{star-shaped domain} obtained by  taking  the filled period-gon (i.e.\ the convex hull of the vertices of the period-gon) and gluing to it  $k+1$ half-infinite straight strips that we call \emph{branches}, which are perpendicular to the sides and of width given by the periods $\mu_\ell$.
To see all this, it is easier to start with the case $\eps\in\R$, i.e.\ $\theta=0$: see \Cref{fig:dessin-2} for a sector, which can then be spread by reflections along its two boundaries and whose image is spread to a star shaped domain as described above.

\begin{figure}[htbp]
\begin{center}
\subfigure{%
\begin{tikzpicture}
\node at (0,0) {\includegraphics[width=6cm]{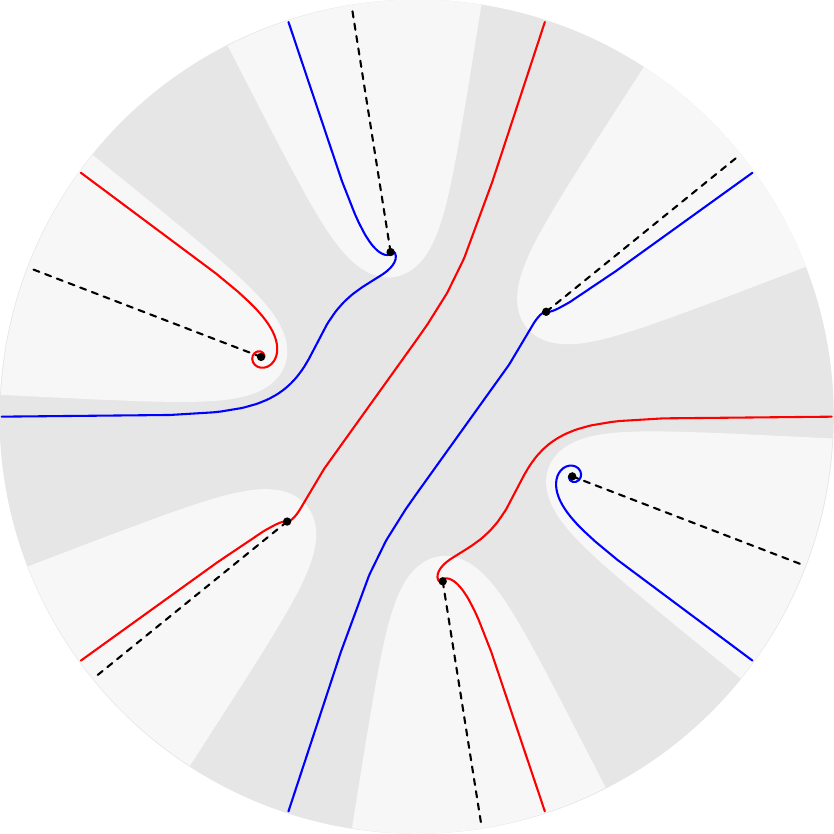}};
\node at (0.7,0.8) {$z_0$};
\node at (0.1,1.2) {$z_1$};
\node at (-1.4,0.35) {$z_2$};
\node at (-1,-0.55) {$z_3$};
\node at (-0.05,-1.2) {$z_4$};
\node at (1.4,-0.35) {$z_5$};
\end{tikzpicture}%
}%
\subfigure{%
\begin{tikzpicture}
\node at (0,0) {\includegraphics[width=6cm]{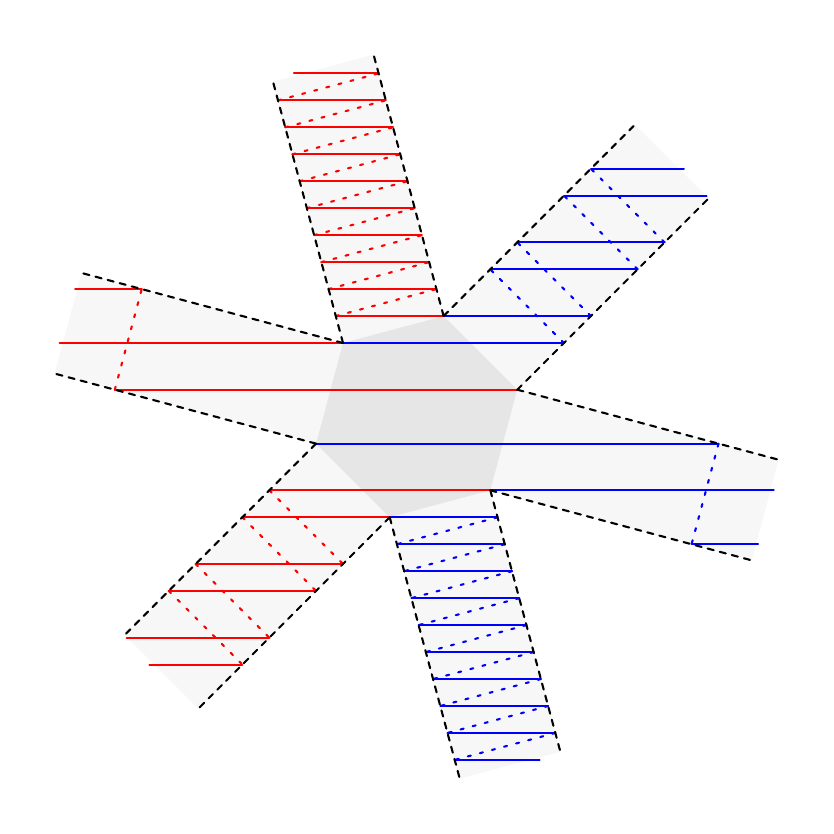}};
\node at (2.8,-0.8) {$z_0$};
\node at (2.1,2.1) {$z_1$};
\node at (-0.7,2.8) {$z_2$};
\node at (-2.8,0.7) {$z_3$};
\node at (-2.1,-2.1) {$z_4$};
\node at (0.8,-2.8) {$z_5$};
\end{tikzpicture}%
}%
\end{center}
\caption{Straightening coordinate $t$ of the vector field $\dot z=z^6-\exp\left(2\pi i\frac{13}{20}\right)$ (see \Cref{example:k5}). Left: the separatrices of the vector field in $z$-space.
Right: the image in $t$-space.
The transverse dashed lines (in blue and red) link parts of trajectories that are identified by the gluing of the sides of the strips (black dashed lines).
Between the left and right parts of the figure, the gray and light gray regions correspond.}
\label{fig:f6}
\end{figure}

\begin{figure}[htbp]
\begin{center}
\subfigure{%
\begin{tikzpicture}
\node at (0,0) {\includegraphics[width=6cm]{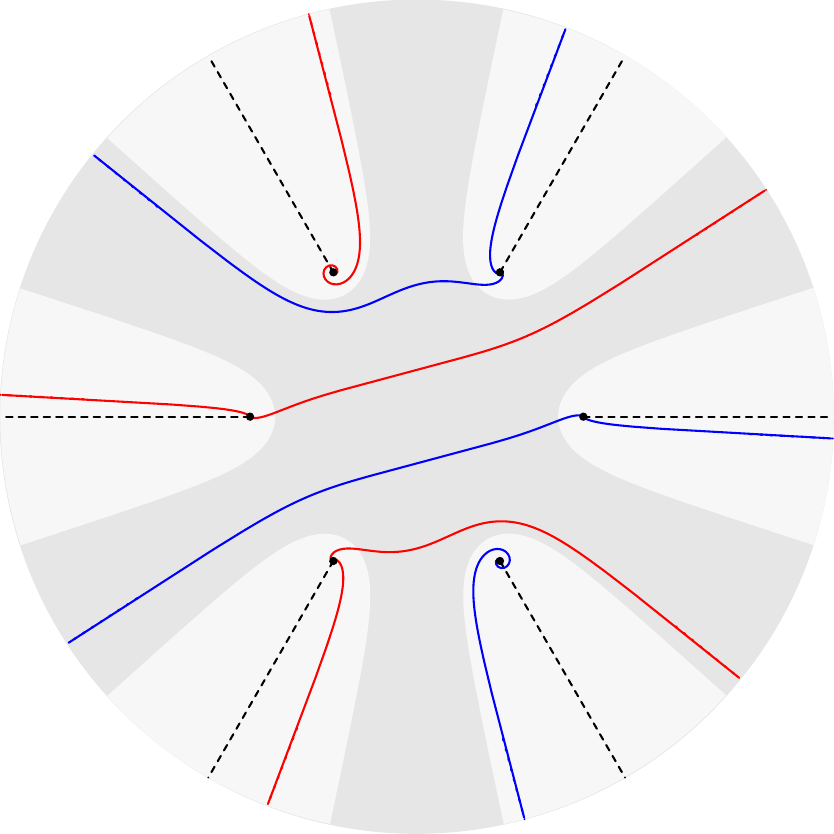}};
\end{tikzpicture}%
}%
\subfigure{%
\begin{tikzpicture}
\node at (0,0) {\includegraphics[width=6cm]{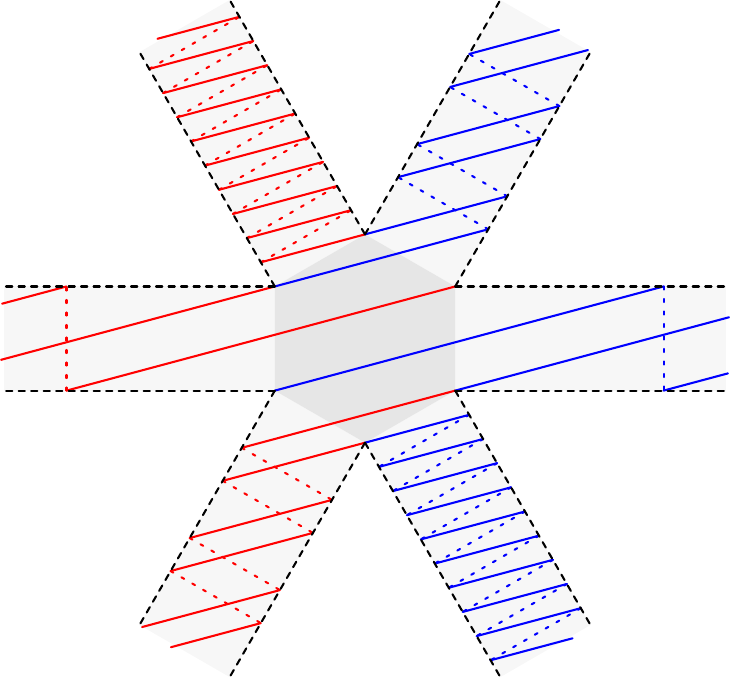}};
\end{tikzpicture}%
}%
\end{center}
\caption{We show the same images as on \Cref{fig:f6} but rotated, so that the isomorphism is easier to figure out. The left part is rotated by $2\pi\frac{-13}{120}=-39\degree$, the right part by $2\pi\frac{5}{120}=15\degree$.
The two parallel sides of a branch of the star on the right are mapped to the dashed line radiating from the singular point (one of the black dots) 
that points in the same direction. But the correspondence exchanges finite and infinite ends of each of these lines.
Between the left and right parts of the figure, the gray and light gray regions correspond.
}
\label{fig:f6b}
\end{figure}

\begin{figure}[htbp]
\centering
\begin{tikzpicture}
\node at (0,-2.25) {\includegraphics[width=8cm]{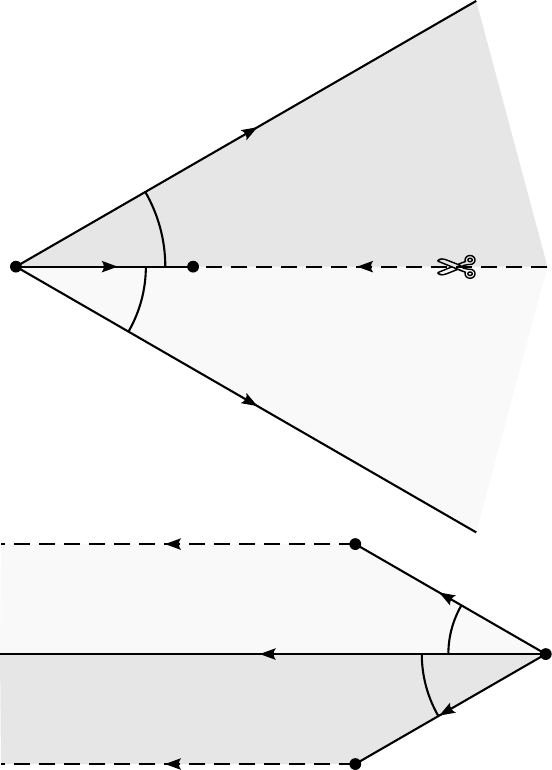}};
\node at (-4.1,-0.6) {$0$};
\node at (-1.2,-0.9) {$1$};
\node at (-0.6,1.8) {$A$};
\node at (1.3,-0.9) {$B$};
\node at (-2.5,-0.9) {$C$};
\node at (-0.6,-2.8) {$D$};
\node at (2,1) {$z$};
\node at (-0.7,0.1) {$\pi/(k+1)$};

\node at (2.7,-7.2) {$A$};
\node at (-1.5,-4.2) {$B$};
\node at (-1.5,-8.05) {$B$};
\node at (0,-5.7) {$C$};
\node at (2.7,-5.1) {$D$};
\node at (4.2,-6.15) {$0$};
\node at (-4.3,-6.15) {$1$};
\node at (1.15,-7.4) {$a$};
\node at (1.15,-4.2) {$\infty$};
\node at (1.15,-8.05) {$\infty$};
\node at (-3,-6.95) {$t$};
\node at (1.2,-6.7) {$\pi/(k+1)$};
\end{tikzpicture}
\caption{\small We prove here that for $\theta=0$, the dark gray sector of angle $\pi/(k+1)$ in $z$-space on the top is mapped one-to-one to the dark gray domain in $t$-space in the bottom.
Consider in $z$-space the oriented half-line marked $A$, from $0$ and of argument $\pi/(k+1)$, parametrized as $z=r \sigma$ with $\sigma =e^{\pm \pi i/(k+1)}$, $r>0$.
On it, $P_\eps(z) = -( r^{k+1} + \eps)$. In particular the integral converges along the line to some limit value $a\in\C$ of argument $\pi+\pi/(k+1)$.
Actually, $z=\infty$ is a zero of order $k-1$ of the differential $du/P_\eps(u)$, as the change of variable $w=1/z$ shows.
Coming back from $z=\infty$ up to $z=1$ along the half-line $B=(1,+\infty)\subset \R$ oriented from $\infty$ to $1$, so $z=r>1$, 
the integral traces a horizontal half-line from $a$ to $\infty$, with real part tending to $-\infty$, since $dt/dz = 1/(r^{k+1}-1)>0$, $r>1$ and $z$ decreases along the real line.
Finally, on part $C = (0,1)$, $P_\eps$ takes a negative value so $t$ varies from $0$ to $-\infty$ along the real line.
At $z=1$ we have $1/P_\eps(z) = \frac{1}{(k+1)(z-1)} + $ a holomorphic map, so $t=\frac{1}{k+1}\log(z-1)+$ a holomorphic map, so points close to $z=1$ are mapped close to $\infty$.
Using the argument principle ensures that the map $z\mapsto t$ is one-to-one and that its image is the indicated set.}
\label{fig:dessin-2}
\end{figure}

Then, since a change of coordinate and (complex!) time 
\begin{equation}
(z,t) \mapsto \left(ze^{-\frac{i\theta}{k+1}}, te^{-\frac{ik\theta}{k+1}}\right)
,
\label{symm:theta}
\end{equation} sends $\dot z= z^{k+1}-e^{i\theta}$ to $\dot z= z^{k+1}-1$, it follows that the descriptions for different $\theta$ just correspond to rotating the star-shaped domain.

This description allows an easy and visual interpretation of several phenomena.
A trajectory in the star shaped domain, once it exits the period-gon, remains trapped in the branch it has entered.
This is valid also when following a trajectory backwards.
As a consequence, a homoclinic loop occurs exactly when two vertices of the period-gon lie on a horizontal line.
This can only occur when the vertical axis is a symmetry of the period-gon, in which case homoclinic bifurcation(s) occur simultaneously, grouping the singular points in sets of one or two points.

The various equivalences of point (2) can be checked by determining explicitly the values of $\arg\eps$ for which they occur and realizing that they coincide.
But that is not very elegant.
As an alternative proof, we may start proving that the period-gon is symmetric w.r.t.\ the vertical axis if and only if the $\pi/2$-rotated vector field ($\dot t = i$ i.e.\ $\dot z = i P_\eps(z)$) is invariant under an anti-holomorphic involution. In $\CP^1$ it is necessarily a reflection through a line passing through $0$, and this line contains $i\eps$ so it is determined.
So the condition is equivalent to the vector field $i P_\eps$ being invariant under the reflection with respect to the line through $0$ and $i\eps$ (equivalently, $P_\eps$ is reversed by the reflection).
The intermediate axes of $i P_0$ are the attracting and repelling axes of $P_0$ and an involution preserving $i P_\eps$ must preserve $i P_0$ (by rescaling and taking a limit), so its set of repelling axes.
A rotation brings the vector field $\dot z =iP_\eps$ to the form $\dot z = Q(z) = z^{k+1}-\eta$ for some $\eta\in\C$.
Then we are left with proving that, given a line $L$ through $0$ and the reflection $s_L$ through this line, then among the following three conditions, any subset of two conditions imply the third one:
\begin{itemize}
\item $(s_L^*Q)(0) = Q(0)$,
\item $s_L$ preserves the set of roots of $Q$,
\item $s_L$ preserves the set of repelling axes of $Q$.
\end{itemize}
This is easy using algebra: 
$s_L^{*} Q$ expresses as 
$az^k-b$ and the three conditions are respectively equivalent to $b = \eta$, $b/a = \eta$ and $a=1$.

Consider now two singularities $\alpha$, $\beta$ of the vector field. Each of them corresponds to a side $A$, $B$ of the period-gon.
One sees that there are trajectories between $\alpha$ and $\beta$ if and only if there is a non-empty intersection of the interior of the orthogonal projections of $A$ and $B$ to the imaginary axis. In the absence of homoclinic loops, this corresponds to a zig-zag ordering, the same as the linear ordering of the $\Re(\kappa_k)$ if $k>1$.

Moreover, from the rotational movement of the period-gon, we deduce the monotonic behavior of the separatrices of $\infty$ very close to a bifucation: in the $t$-coordinate, the separatrices coming from infinity\footnote{Sometimes called attracting separatrices because they are asymptotic, at $\infty$ in $z$ coordinate, to the attracting axes of the parabolic point of the case $\eps=0$.} are the horizontal lines that stem rightwards from the vertices of the period-gon, while the horizontal lines stemming leftwards from the vertices correspond to separatrices going to infinity.
Between two consecutive values of $\theta$ such that the period-gon is symmetric w.r.t.\ the imaginary axis, the imaginary part of each vertex of the period-gon will depend monotonically on $\theta$, leading to the announced change in the Douady-Estrada-Sentenac combinatorial invariant.
\end{proof}

In the star shaped domain, the singularities are at the end of the branches. The attracting ones correspond to branches tending to infinity in a direction in $(-\pi/2,\pi/2)$, the neutral one to vertical branches and the repelling ones to a direction in $(\pi/2,3\pi/2)$.
In particular, this reproves that the attracting (resp.\ repelling) singularities have consecutive arguments.

\begin{corollary}
 Homoclinic loops are arranged as follows.
 \begin{enumerate}
\item For $k=2m$ even (i.e.\ an odd number of singularities), bifurcations of homoclinic loops occur precisely when one singular point $z_\ell$ has a eigenvalue in $i\R$.
These occur for $\arg(\eps)= \frac{\pi}{2k}+\frac{n\pi}{k} $ for $n=0, \dots, 2k-1$.
Then, there are $m$ simultaneous homoclinic loops separating $\C$ into $m+1$ regions containing respectively $z_\ell$ and the $m$ pairs of points $\{z_{\ell-j},z_{\ell+j}\}$, $j=1, \dots m$.
When $\eps$ rotates in the positive direction, the next group of simultaneous bifurcations to occur is the one in which  $z_{\ell-m}$ is isolated by a homoclinic loop. See \Cref{fig_k4}(c).
\item For $k=2m+1$ odd with $m\geq0$ (i.e. an even number of singularities), there are  one or  two kinds of possible arrangements:
  \begin{itemize}
  \item In the first type, two opposite singular points $z_\ell$ and $z_{\ell+m+1}$ (indices are modulo $k+1$) have eigenvalues in $i\R$ and  there are $m+1$ simultaneous homoclinic loops separating the singular points into the $m+2$ groups $\{z_\ell\}$, $\{z_{\ell-1},z_{\ell+1}\}$, \dots, $\{z_{\ell-m},z_{\ell+m}\}$, $\{z_{\ell+m+1}\}$. These occur when $\arg(\eps)= \frac{\pi}{k} +\frac{2n\pi}{k} $ for $n=0, \dots, k-1$. See \Cref{fig_k5}(e).
  \item In the second type,  which occurs only for $m\geq1$, there are $m$ homoclinic loops separating the points into $m+1$ pairs $\{z_\ell,z_{\ell+1}\}$, $\{z_{\ell-1},z_{\ell+2}\}$ \dots, $\{z_{\ell-m},z_{\ell+m+1}\}$. This case occurs  when $\arg(\eps)=\frac{2n\pi}{k} $ for $n=0, \dots, k-1$.  See \Cref{fig_k5}(b).
  \end{itemize}
When $\eps$ rotates uniformy in the positive direction, the two types of bifurcations alternate. The first type with an end group $\{z_\ell\}$ is followed by one of the second type with end group 
$\{z_\ell,z_{\ell+1}\}$, and then by one of the first type with end group $\{z_{\ell+1}\}$, etc.
\end{enumerate}
\end{corollary}

\begin{proof}
We treat the case $k=2m$.
Recall that the period-gon has edges given by the vectors $-\mu_\ell = -2\pi i/P'(z_\ell)$ and that $P'(z_\ell) = (k+1) e^{i\frac{k(\theta+2\pi\ell)}{k+1}}$.
By the previous theorem, there is a bifurcation if and only if the imaginary axis is a symmetry axis of the period-gon.
Since it has an odd number $k+1$ of sides, this happens if and only if there is a horizontal edge $\mu_\ell\in\R$, if and only if  $P'(z_\ell)\in i\R$, or equivalently if and only if $\kappa_\ell \in i\R$ (recall $\kappa_j=1/P'(z_j)$).
Solving for $\frac{k(\theta+2\pi\ell)}{k+1}\equiv \pi/2 \bmod \pi$ gives the values of $\theta$ in the statement.
For any $\theta$, the $\kappa_\ell$ are the vertices of a regular $(k+1)$-gon centered at the origin.
When $\eps$ rotates uniformly in the positive direction, the $\kappa_\ell$ rotate uniformly (slower, with factor $k/(k+1)$) in the negative direction. 
If some $\kappa_\ell$ belongs to $i\R$, when rotating $\theta$ in the positive direction, the next $\kappa_j$ to reach $i\R$ is $\kappa_{\ell-m}$.

The case $k=2m+1$ is treated similarly.
\end{proof}

When considering the phase portrait on $\CP^1$ for nonzero $\eps$ then, except when $\eps$ is on the $2k$ half-lines $e^{i\theta_j}\R^+$, where  $\theta_j$ is defined in \Cref{theta_j}, all separatrices  of $\infty$ end at singular points, and there exist a chain of trajectories between the singular points, forming a trunk. The $\eps$-space is then decomposed as the union of:
\begin{itemize} 
\item $2k$ open sectors $V_j= \{\eps\: :\: \arg\eps\in (\theta_j, \theta_{j+1})\}$ where the vector field is structurally stable;
\item  the bifurcation locus of real codimension 1, which is composed of the $2k$  half-lines $e^{i\theta_j}\R^+$ where homoclinic connections occur;
\item the limit point $\eps=0$: there, the homoclinic connections become heteroclinic connections passing through the singular point $z=0$. \end{itemize}

\begin{example}\label{example:k4}
Let us consider the family $\dot z= z^5-\eps$. 
\Cref{fig_k4} shows the phase portraits of $\dot z= z^5-\eps$ for five values of $\theta\in[0, \pi/4]$.
\end{example}

\begin{figure}[htbp]
\begin{center}
\subfigure[$\theta=0$]{\includegraphics[width=3.6cm]{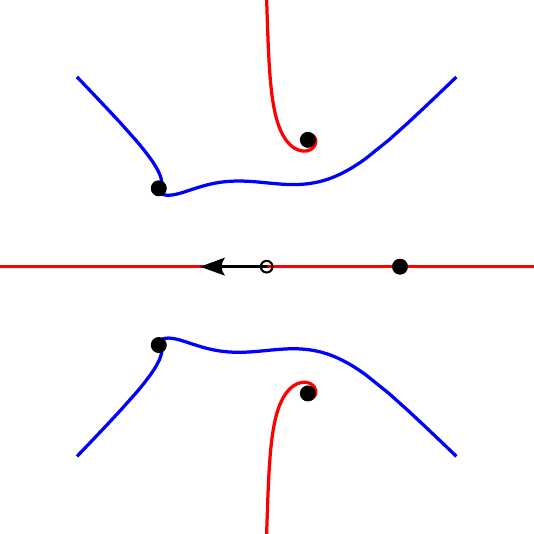}}\qquad
\subfigure[$\theta\in (0,\frac{\pi}8)$]{\includegraphics[width=3.6cm]{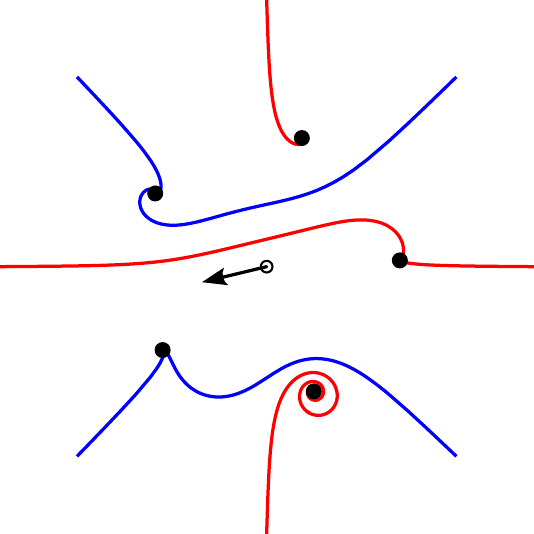}}\qquad
\subfigure[$\theta=\frac{\pi}8$]{\includegraphics[width=3.6cm]{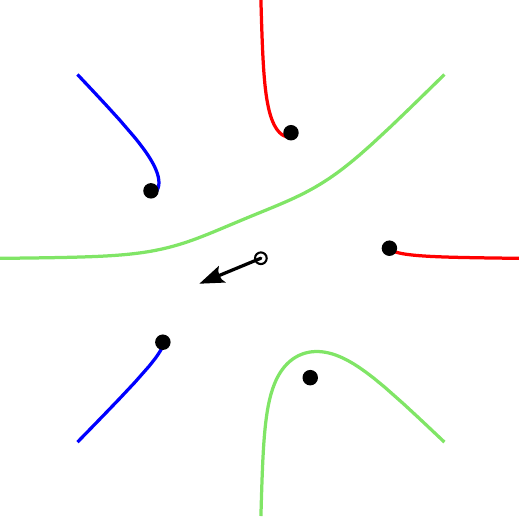}}\qquad
\subfigure[$\theta\in (\frac{\pi}8, \frac{\pi}4)$]{\includegraphics[width=3.6cm]{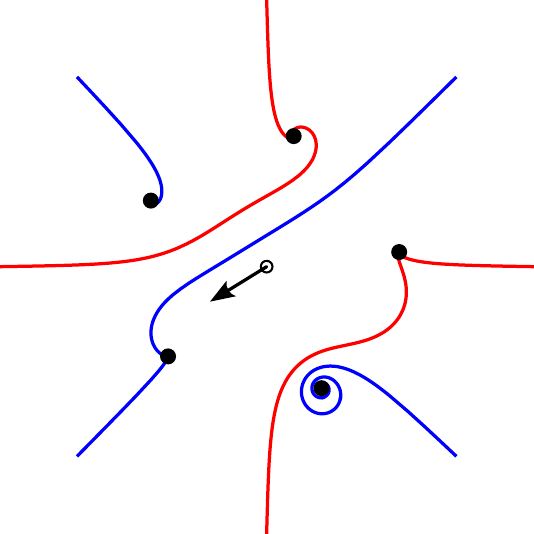}}\qquad
\subfigure[$\theta=\frac{\pi}4$]{\includegraphics[width=3.6cm]{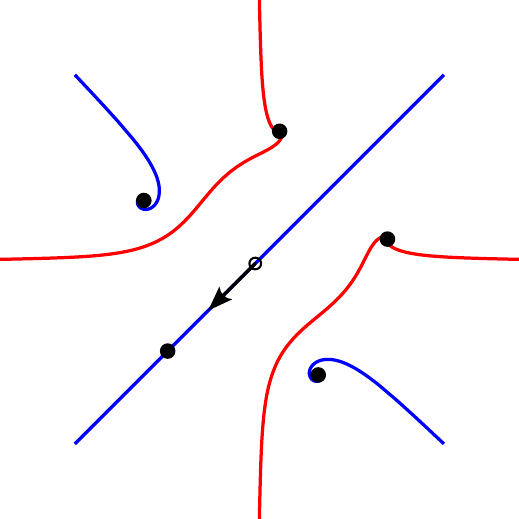}}
\caption{The phase portraits of $\dot z= z^5-\eps$ for $\theta\in[0, \pi/4]$. The arrow indicates the direction of the vector $P_\eps(0)=-\eps$.} \label{fig_k4}
\end{center}
\end{figure}

\begin{example}\label{example:k5} Let us consider the family $\dot z= z^6-\eps$.
\Cref{fig:f6} represents a generic situation. It shows a branch of $t=\int \frac{dz}{z^6-\eps}$ defined on the following simply connected set: the complex plane minus the 6 straight slits depicted in dashed lines, radiating from the roots $z_i$ of $z^6-\eps$ and going to $\infty$.
The branch is injective and its image is a star shaped domain, depicted on the right.
The two parallel sides of a branch of the star on the right are the images of  the two sides of a corresponding slit.
To recover the complex plane (minus singularities), one has to roll and glue the branches, by gluing points separated by a period. A neighborhood of infinity on the left is cut in 6 sectors that are mapped on the right to neighborhoods of the 6 corners.
The separatrices land at the singular points: indeed, no two corners of the star are at the same vertical height.
\\
Because of the symmetries, to highlight the bifurcations it is sufficient to consider $\theta\in\left[0,\frac{\pi}{5}\right]$. The phase portraits appear in \Cref{fig_k5}.

\begin{figure}[htbp]
\begin{center} 
\begin{tikzpicture}
\node at (0,0){\includegraphics[width=6cm]{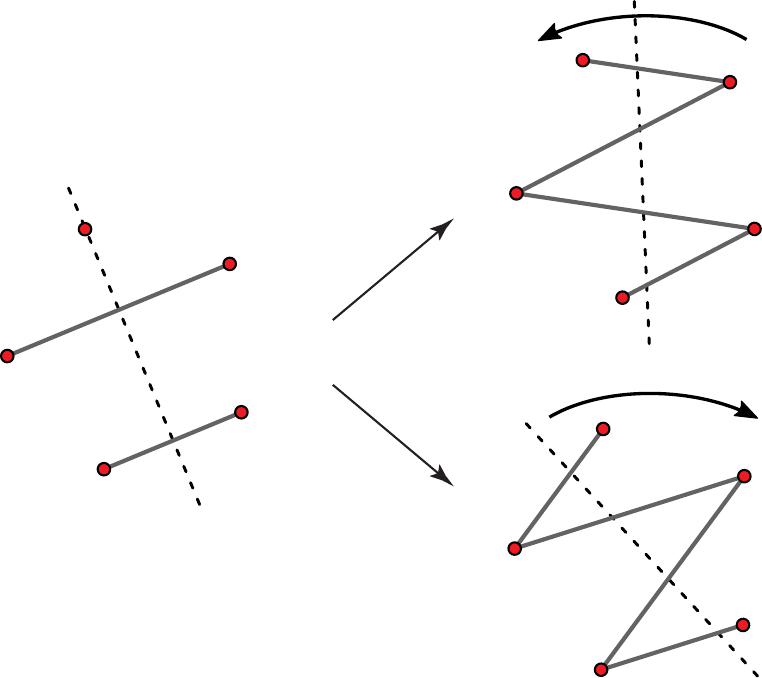}};
\node at (-0.1,1) {$A$};
\node at (-0.1,-1.1) {$B$};
\end{tikzpicture}
\caption{Type of bifurcation of Douady-Estrada-Sentenac invariants for $k+1=5$. The red dots denote the singular points of the vector field $z^{k+1}-\epsilon$ for various values of $\arg\epsilon$. The dashed line is orthogonal to $-\eps = P_\eps(0)$. The symmetry axis for the homoclinic loop can be any of the lines $e^{\frac{{2j+1}}{8}\pi i}\R$. The corresponding bifurcations of the separatrices of $\infty$ through homoclinic loops can be seen in \Cref{fig_k4} (rotated by $\frac{129\pi}{40}$). We have the passage  through arrow $A$ (resp.\ $B$) from \Cref{fig_k4}(c) to \Cref{fig_k4}(d)  (resp.\ \Cref{fig_k4}(b)).}
\label{homoclinic3}
\end{center}
\end{figure}

\begin{figure}[htbp]
\begin{center} 
\subfigure[]{\begin{tikzpicture}
\node at (0,0){\includegraphics[width=5.6cm]{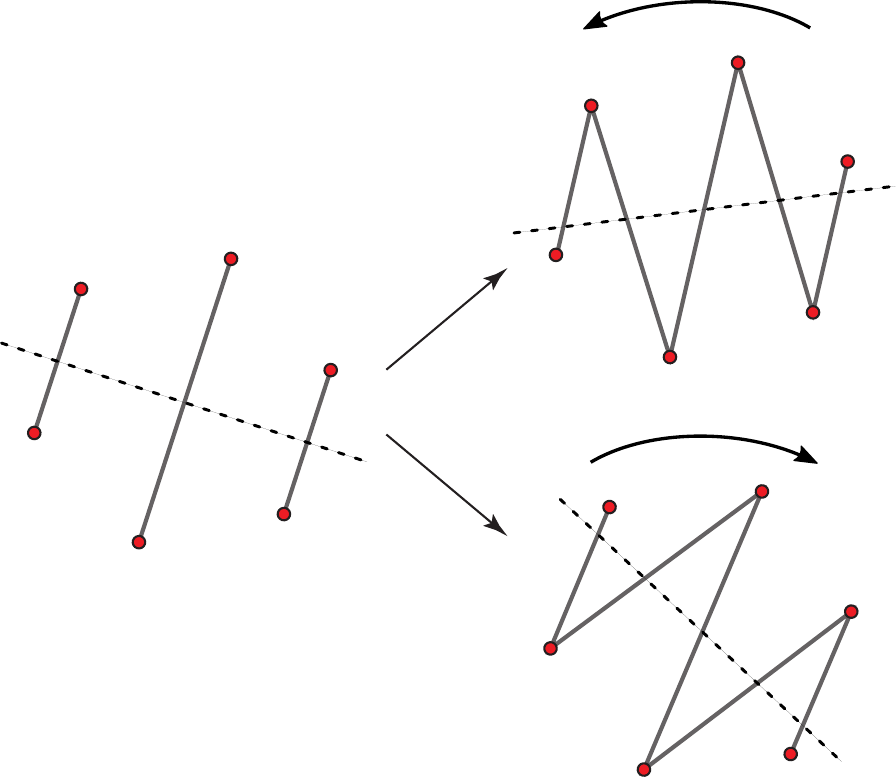}};
\node at (-0.2,0.8) {$A$};
\node at (-0.2,-1) {$B$};
\end{tikzpicture}}
\quad
\subfigure[]{\begin{tikzpicture}
\node at (0,0){\includegraphics[width=5.6cm]{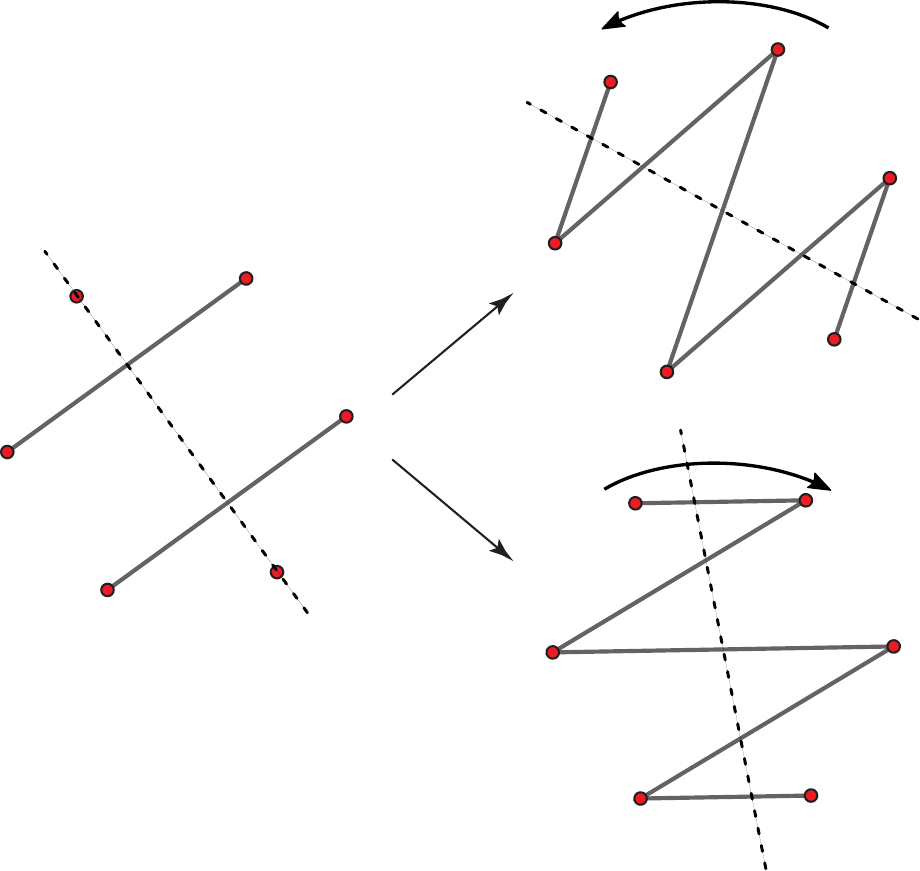}};
\node at (-0.2,0.9) {$A$};
\node at (-0.2,-0.9) {$B$};
\end{tikzpicture}}
\caption{\small The two types of bifurcations of the Douady-Estrada-Sentenac invariant for $k+1=6$. Same conventions as on \Cref{homoclinic3}. For each of the two sub-figures we have on the left the broken combinatorial invariant and on the right the two different combinatorial invariants before and after the homoclinic loops. The symmetry axis for the homoclinic loop can be any of the lines $e^{\frac{(2j+1)\pi }{10}}\R$.  The corresponding bifurcations of the separatrices of $\infty$ through homoclinic loops can be seen in \Cref{fig_k5} (with a rotation in $z$-space). In (a) above, we have  through arrow $A$ (resp.\ $B$) the passage from \Cref{fig_k5}(b) to \Cref{fig_k5}(a)  (resp.\ \Cref{fig_k5}(c)). In (b) above, we have the passage  through arrow $A$ (resp.\ $B$) from \Cref{fig_k5}(e) to \Cref{fig_k5}(d)  (resp.\ \Cref{fig_k5}(f)).}
\label{homoclinic1}
\end{center}
\end{figure}

\begin{figure}[htbp]
\begin{center}
\subfigure[$\theta=-\delta$ for small $\delta>0$]{\includegraphics[width=3.6cm]{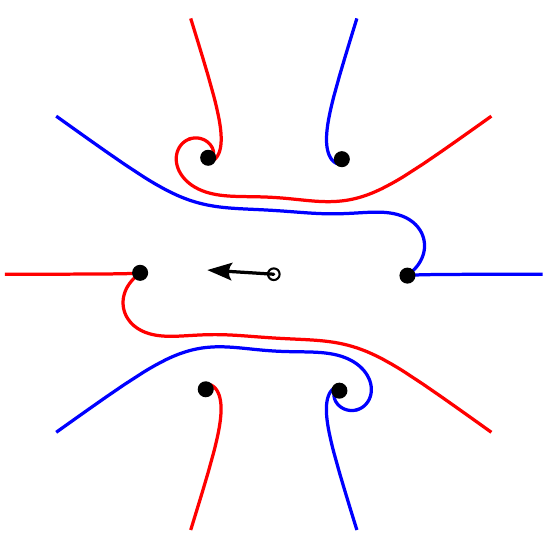}}\qquad
\subfigure[$\theta=0$]{\includegraphics[width=3.6cm]{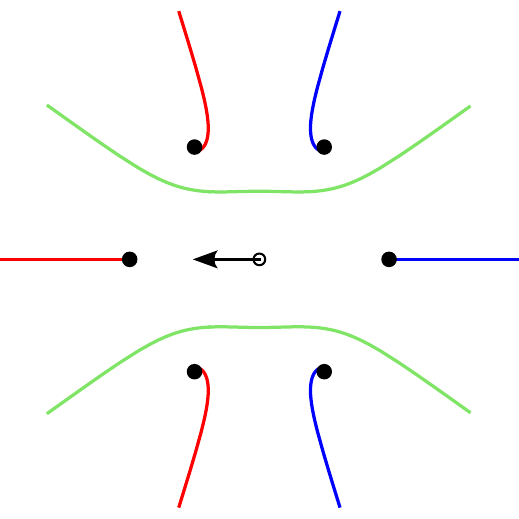}}\qquad
\subfigure[$\theta=\delta$ for small $\delta>0$]{\includegraphics[width=3.6cm]{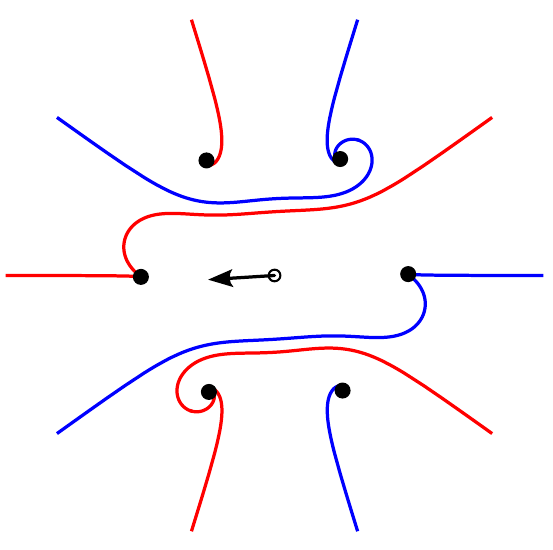}}\qquad
\subfigure[$\theta=\frac{\pi}5-\delta$ for small $\delta>0$]{\includegraphics[width=3.6cm]{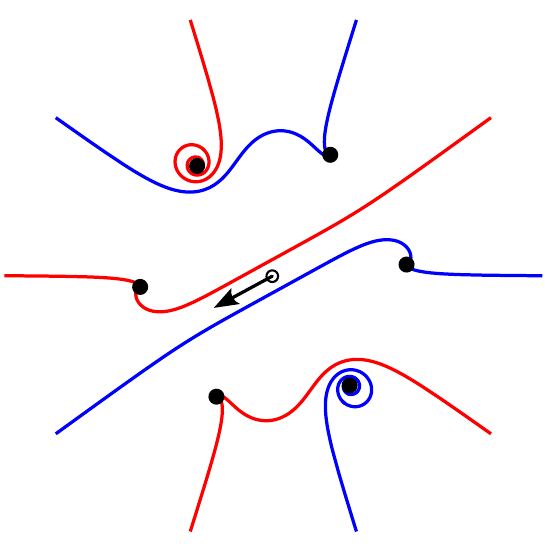}}\qquad
\subfigure[$\theta=\frac{\pi}5$]{\includegraphics[width=3.6cm]{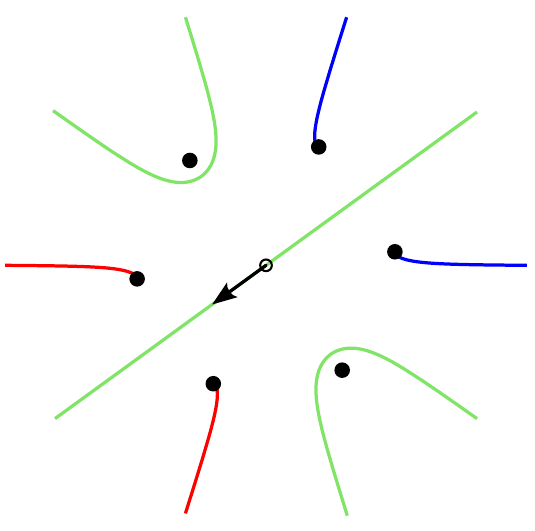}}\qquad
\subfigure[$\theta=\frac{\pi}5+\delta$ for small $\delta>0$]{\includegraphics[width=3.6cm]{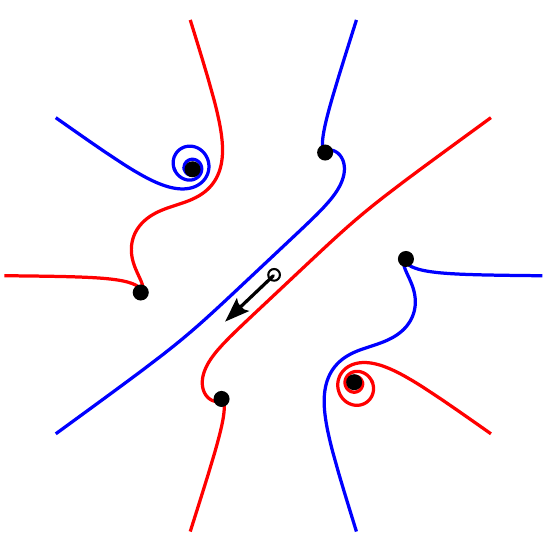}}
\subfigure[$\theta=-\delta$ for small $\delta>0$]{\includegraphics[width=3.6cm]{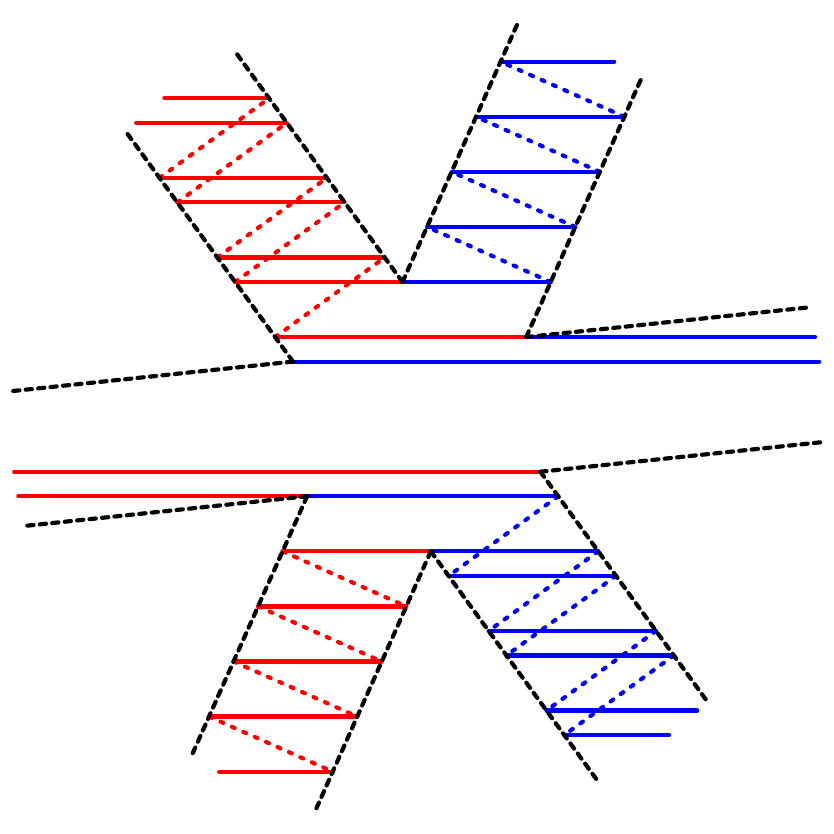}}\qquad
\subfigure[$\theta=0$]{\includegraphics[width=3.6cm]{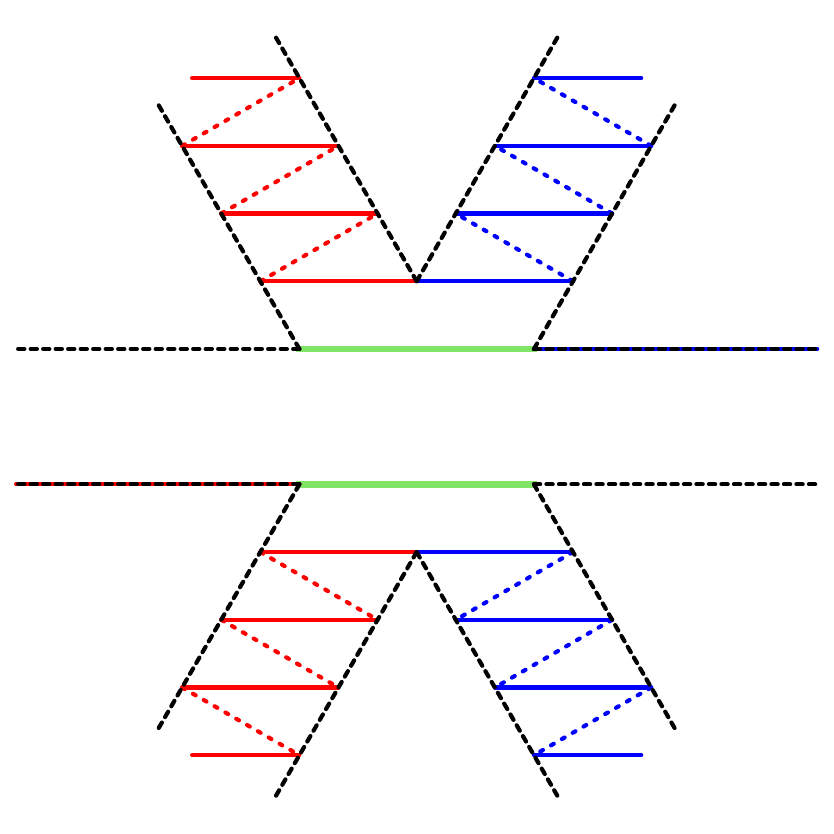}}\qquad
\subfigure[$\theta=\delta$ for small $\delta>0$]{\includegraphics[width=3.6cm]{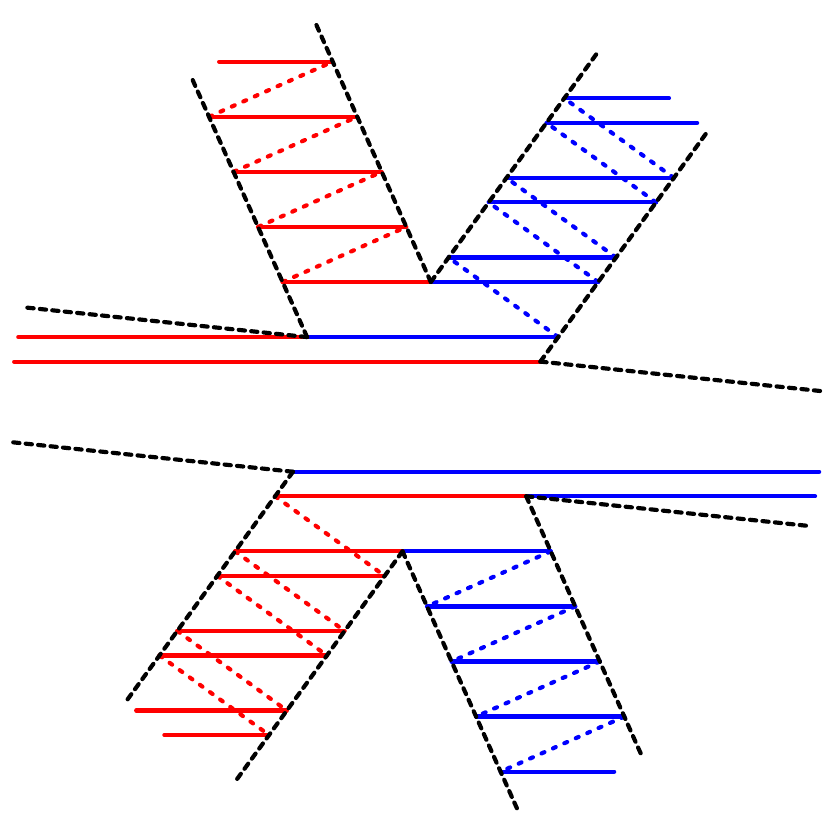}}\qquad
\subfigure[$\theta=\frac{\pi}5-\delta$ for small $\delta>0$]{\includegraphics[width=3.6cm]{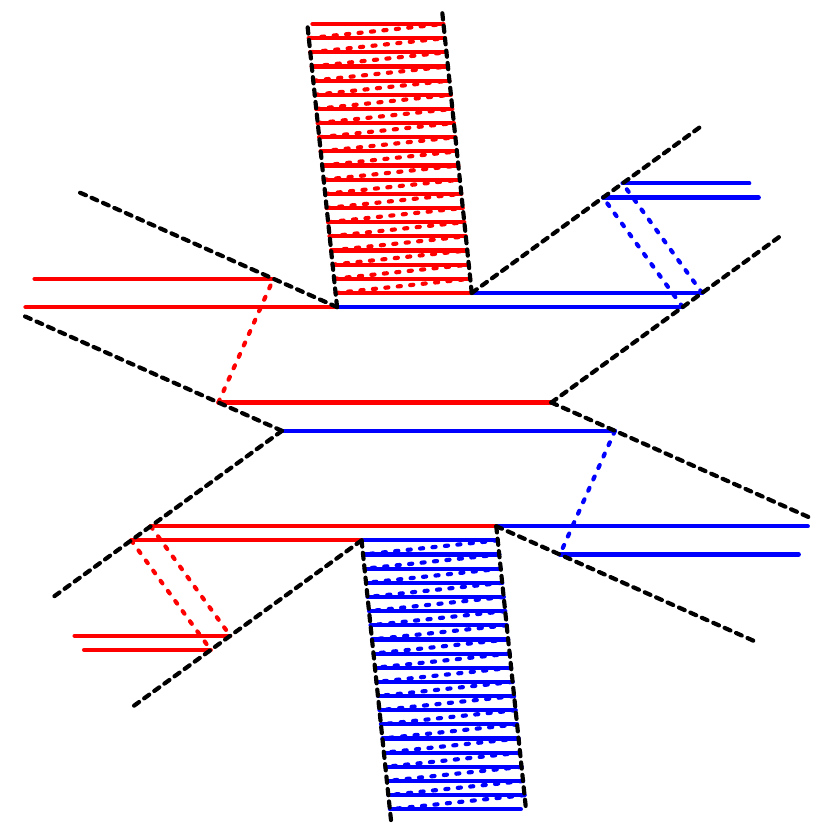}}\qquad
\subfigure[$\theta=\frac{\pi}5$]{\includegraphics[width=3.6cm]{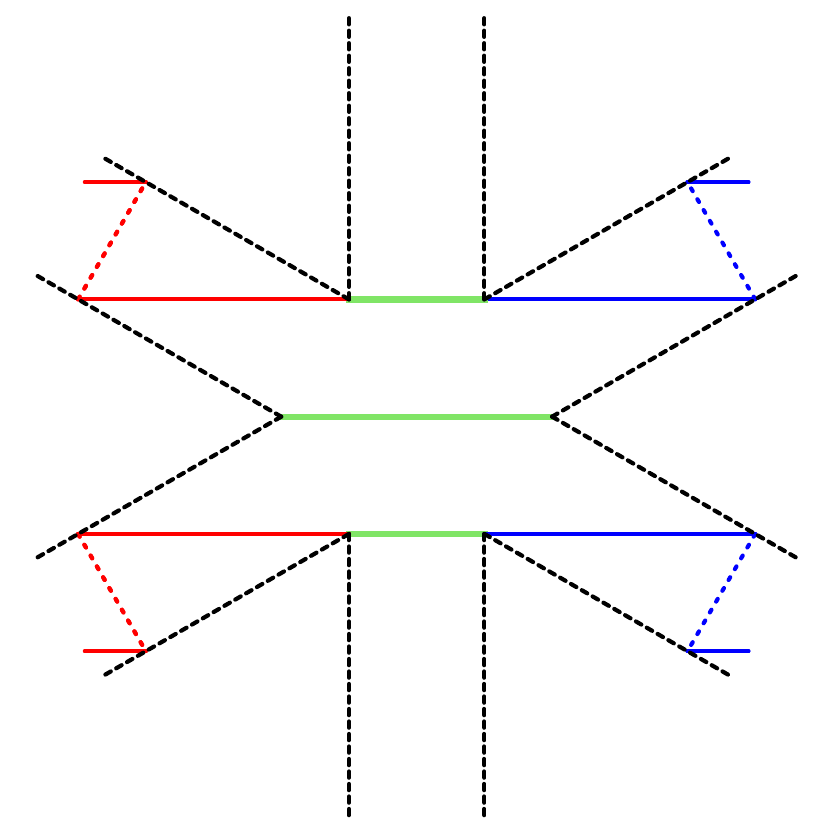}}\qquad
\subfigure[$\theta=\frac{\pi}5+\delta$ for small $\delta>0$]{\includegraphics[width=3.6cm]{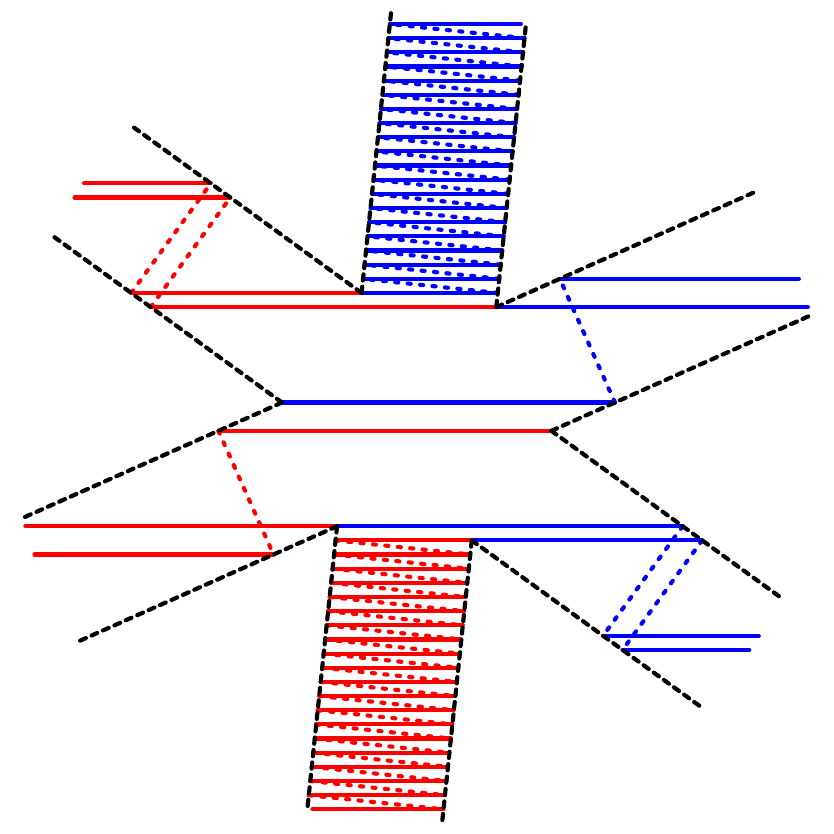}}\qquad
\caption{The phase portraits of $\dot z= z^6-\eps$ and the corresponding lines in $t$-space for $\epsilon=e^{2\pi i \theta}$ and $\theta \in[-\delta,\delta]$ or $\theta\in[-\delta,\frac{\pi}{5}+\delta]$, with $\delta>0$, small. The arrow at the center shows the direction of the vector $P_\eps(0)=-\eps$.} \label{fig_k5}
\end{center}
\end{figure}
\end{example}

\FloatBarrier

We saw in the proof of \Cref{thm:bif_vf_simpl} that when homoclinic loops occur then (if $k>1$ the converse also holds) the vector field $\omega$ for the equation $\dot z=P_\eps(z) = z^{k+1}-\eps$ has an axis $\Delta$ of antisymmetry (the system is reversible w.r.t.\ $\Delta$), i.e.\ $s^* \omega = -\omega$ with $s$ the reflection across $\Delta$. This axis is necessarily an intermediate axis orthogonal to $-\eps = P_\eps(0)$. 
If $-\eps$ is orthogonal to an intermediate axis $\Delta$, then $\Delta$ is an antisymmetry axis of $\omega$.
We orient $\Delta$ according to $-i\eps$: this is the direction that is mapped upwards (argument $\pi/2$) in straightening coordinate.
Let
\[\theta' = \arg (-i\eps ) = \theta - \frac{\pi}{2} \in\R/2\pi \Z
,\] be the argument of the oriented axis $\Delta$.

We recall that the vector field has $k$ separatrices coming from infinity along the $k$ directions for which $z^{k+1} / z \in \R^+$ and $k$ separatrices going to infinity along the $k$ directions for which $z^{k+1}/z \in \R^-$. 
These directions are also respectively the attracting and the repelling direction of the parabolic singularity at $0$ of $\dot z = P_0(z) = z^{k+1}$.
A homoclinic loop $\gamma$ can be seen as a link between an incoming separatrix and an outgoing one. 

\begin{proposition}\label{prop:hl:pairs:dirs}
For each incoming direction $\theta_a$ and each outgoing direction $\theta_r$, the vector field $\dot z = P_\eps(z)$ has a homoclinic loop linking these directions if and only if the argument of the oriented antisymmetry axis (reversion axis) is the midpoint $\theta_m$ of the segment from $\theta_a$ to $\theta_r$ in the anticlockwise direction.
\end{proposition}
As we saw above, for $\eps = e^{i\theta}$ this happens for a unique $\theta\in\R/2\pi\Z$: $\theta = \theta_m+\pi/2$.
\begin{proof}
If there is such a homoclinic loop $\gamma$, then we saw that there is an antisymmetry axis $\Delta$, and also in \Cref{thm:bif_vf_simpl} that the period-gon has a vertical axis of symmetry. Actually the straightening coordinate conjugates the symmetry across $\Delta$ to the symmetry across the imaginary axis.
It maps the half-plane $H_+$ delimited by $\Delta$ and containing $P_\eps(0)=-\eps$ to the right half of the star shaped domain cut by the imaginary axis, and the other half-plane $H_-$ to the other half of the star.
We also saw that it maps $\gamma$ to a horizontal segment, oriented from left to right.
It follows that $\gamma$ is symmetric w.r.t.\ $\Delta$, starts from the half-plane $H_-$, crosses $\Delta$ and ends in $H_+$.
The asymptotic directions of $\gamma$ are symmetric with respect to $\Delta$ and the argument $\theta'$ of $\Delta$ (oriented by $-i\eps$ as above) is the midpoint between $\theta_a$ and $\theta_r$ in anticlockwise order.

Conversely, if $\theta=  \theta_m+\pi/2 $, then $\omega$ has an oriented antisymmetry axis $\Delta$ of argument $\theta_m$, and the period-gon is symmetric w.r.t.\ $i\R$, $H_-$ is mapped to the left of $i\R$ and $H_+$ to the right.
From every incoming direction situated in $H_-$ stems a separatrix connected to the outgoing direction that is symmetric w.r.t\ $\Delta$: indeed in the star shaped domain, incoming separatrices are the horizontal lines from the vertices of the period-gon, while the outgoing ones are the horizontal lines to these vertices.
\end{proof}

\section{Generic one-parameter unfoldings of vector fields}\label{sec:local_vect}

\subsection{Definitions}

Consider a germ of holomorphic vector field in one complex dimension 
\[ \dot{z}=\omega_0(z)= Az^{k+1} +O(z^{k+2}),\ A\in\C^*
,\]
with a parabolic point of codimension $k\geq1$.

\emph{Repelling axes} of $\omega_0$ are defined as the set of complex numbers $z$ such that $Az^{k+1}$ and $z$ have the same argument. They are the set of points for which, at first order, the vector field points away from the origin. They form $k$ half-lines and should be thought of as living in the tangent space to $\C$ at the origin. 

\begin{definition}\label{def:gen_vf}
 A germ of one-parameter analytic family of vector fields $\omega_\eps$ unfolding a parabolic germ $\omega_0(z)$ is \emph{generic} if
\begin{equation}\label{eq:generic}
 \frac{\partial\omega_\eps(0)}{\partial \eps}\neq 0
\end{equation}
for $\eps=0$.\footnote{Hence in a neighborhood $\epsilon=0$.}
\end{definition}

The genericity condition implies that
\[
\omega_\eps(z) = A z^{k+1} - B \eps + O(z^{k+2},\eps z,\eps^2),
\]
where the terms of the $O$ can be for instance deduced from the Newton polygon:

\begin{center}
\begin{tikzpicture}
\node at (0,0) {\includegraphics{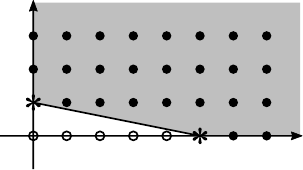}};
\node at (2.7,-1) {$z$};
\node at (-2.3,1.4) {$\eps$};
\end{tikzpicture}
\end{center}

\begin{definition}\label{def:unitary}
We say that a family $\omega_\eps$ as in \Cref{def:gen_vf} is \emph{unitary} if $A=1$ and $B=1$, i.e.\ if
\begin{equation}\label{eq:generic:unitary}
\omega_\eps(z) = z^{k+1} - \eps + O(z^{k+2},\eps z,\eps^2).
\end{equation}
\end{definition}

Without loss of generality, we assume in the rest of \Cref{sec:local_vect} (except in \Cref{sub:ppal}) that $\omega_\eps$ is unitary.
This can be achieved by a linear change of variable and a linear change of parameter.
In \Cref{sub:ppal} we explain the geometric meaning of such a choice.

The map $\omega_0$ has a multiple root of order $k+1$ and it follows that, near $0$, $\omega_\eps$ has $k+1$ roots counted with multiplicity. From the condition $\frac{\partial}{\partial \eps} \omega_\eps = 1 \neq 0$ for $(z,\eps)=(0,0)$, the implicit function theorem implies that one can locally parametrize the zero locus of $\omega_\eps(z)$ as $(z,f(z))$ for a holomorphic $f$ with $f(0)=0$. Plugging this into \Cref{eq:generic:unitary} implies that $f(z)\sim z^{k+1}$ as $z\to 0$.
As a consequence:

\begin{lemma}\label{lem:roots:asymp}
For all $r>0$ small enough there exists $\rho>0$ such that for $\eps\in B(0,\rho)$ and $\eps\neq 0$, there are $k+1$ singularities of $\omega_\eps$, and as $\eps\to 0$, their affixes are asymptotic to the solutions of $z^{k+1}-\eps = 0$.
\end{lemma}
\noindent We state a stronger result in \Cref{sub:chvar_roots}.

\medskip

We are interested in the classification of germs of generic families of vector fields unfolding a parabolic singularity. 

\begin{definition}\label{def:conj}
Two generic germs $\omega_\eps$ and $\widetilde{\omega}_{\tilde{\eps}}$ are \emph{conjugate} if there exist $\rho>0$, $r>0$ and:
\begin{itemize}
\item a change of parameter $\tilde \eps = h(\eps)$ where the analytic map $h:B(0,\rho)\to \C$ satisfies $h(0)=0$ and $h'(0)\neq 0$,
\item a germ of analytic diffeomorphism near $0$:
$\Psi: B(0,r)\times B(0,\rho) \rightarrow \C^2$,
$(z,\eps)\mapsto (\psi_\eps(z),h(\eps))$
such that $\Psi(0,0)=(0,0)$ 
and such that $\forall \epsilon \in B(0,\rho)$,
$\psi_\eps$ is a conjugacy over $B(0,r)$ from the vector field $\omega_\eps$ to (a restriction of) the vector field $\widetilde{\omega}_{\tilde{\eps}}$.
\end{itemize}
The map $\Psi$ is called a \emph{change of variable and parameter}.
\end{definition}

Note that the differential of $\Psi$ at any point is necessarily upper triangular.

\begin{definition}\label{def:u}
If $h(\eps) = \epsilon$, then $\Psi$ is said to \emph{preserve the parameter}.
We say that $\Psi$ is \emph{$0$-fixing} if $\psi_\eps(0)=0$ for all $\eps$, that $\Psi$ is \emph{unitary} if its differential at $(0,0)$ is the identity, and that it is \emph{unipotent} if it is of the form
\[
D_{(0,0)} \Psi = \left[\begin{array}{cc} 1 & \bullet \\ 0 & 1 \end{array}\right]
\]
where the dot stands for any complex number.
\end{definition}

If $\omega_\eps$ is unitary, then it remains unitary after a unipotent change of variable and parameter.
A change of variable and parameter that is unipotent and $0$-fixing is in particular unitary.

For unitary generic germs $\omega_\eps$ and $\tilde \omega_{\tilde \eps}$, being conjugate by a unipotent change of variable and parameter or by a unitary one are equivalent since we have the following result:
\begin{lemma}\label{lem:0fix}
If $\Psi$ is a unipotent conjugacy between two generic germs in the sense of \Cref{def:gen_vf} then there exists a unipotent change of variable $\Phi$ fixing the parameter and such that $\Phi\circ\Psi$ is a $0$-fixing conjugacy between the same germs. \end{lemma}
\noindent Being moreover unipotent, $\Phi\circ\Psi$ is thus unitary.
\begin{proof}[Proof of \Cref{lem:0fix}]
Since $\Psi(0,\eps)= (a(\eps), h(\eps))$ for some function $a$ and $h$ with $h'(0)=1$, we look for a self conjugacy $\Phi$ of $\tilde\omega_{\tilde\eps}$ of the form $\Phi (\tilde z,\tilde \eps)= (\phi_{\tilde \eps}(\tilde z), \tilde \eps)$ satisfying  $\phi_{\tilde \eps}(a(\eps))  = 0$, i.e.\ $\phi_{\tilde \eps}(a(h^{-1}(\tilde \eps))) = 0$, i.e.\ $\phi_{\tilde \eps}(g(\tilde \eps)) = 0$ (for all $\tilde \eps$ close to $0$) where $g = a \circ h^{-1}$.
Such a family of conjugacies $\phi_\eps$ is provided by \Cref{lem:1} below, as the time $t(\eps)$ of the flow of  $\tilde\omega_{\tilde\eps}$ for some well chosen $t(\eps)$.
\end{proof}

From now on, we choose to formulate everything in terms of existence of unipotent conjugacies, keeping in mind that it implies the existence of $0$-fixing unitary ones.

\begin{lemma}\label{lem:1}
For any holomorphic function $g(\eps)$ with $g(\eps) = 0$ there exists a germ of change of variable fixing the parameter $\Phi(z,\eps) = (\phi_\eps(z),\eps)$ conjugating $\omega_\eps$ to itself, so that $\phi_0'(0)=1$ (i.e.\ $\Phi$ is unipotent) and $\phi_{\eps} (g(\eps)) = 0$.
\end{lemma}
\begin{proof}
  We can choose $r>0$ sufficiently small, and then $\rho>0$ sufficiently small so, that for all $\eps$ such that $|\eps| < \rho$,
  $\forall z\in B(0,r(\eps))$ with $r(\eps) = |\eps/3|^{1/(k+1)}$,
  $ \left| \frac{\omega_\eps(z)}{-\eps} - 1 \right| < 1/2
  $.
  It follows that the antiderivative $F_\eps$ of $1/\omega_\eps$, taking value $0$ at $\eps$ is (well defined since we avoid singularities and) injective on $B(0,r(\eps))$. We recall that $F_\eps$ is a complex time coordinate, which straightens the vector field.
  Since $g(0)=0$ we have $g(\eps) = \eps f(\eps)$ for some holomorphic function $f$. 

  We let $\phi_{\eps}$ be the time-$t(\eps)$ map of the vector field $\omega_{\eps}$, where  $t(\eps) = - F_\eps(g(\eps))$.
  We have $g(\eps) = o(r(\eps))$ as $\eps\to 0$, so it follows that the straight line $[0,t(\eps)]$ is contained in $F_\eps(B(0,r(\eps)))$, whence $\phi_\eps(g(\eps))=0$.
  We have $t(\eps) = -F_\eps(g(\eps)) = -\int_0^{g(\eps)} dz/\omega_{\eps} (z) = \int_0^1 \eps f(\eps)dt / (\eps - \eps^{k+1} f(\eps)^{k+1} + \ldots) = \int_0^1 f(\eps)dt / (1-\eps^{k} f(\eps)^{k+1} + \ldots)$ is well-defined for $\eps$ small and tends to $f(0)$ as $\eps\to 0$, so it is holomorphic in a neighborhood of $0$.
  Since $\omega_0$ has a multiple zero at the origin, $\phi_0(z)\sim z$ as $z\to 0$.
\end{proof}

\begin{remark}\label{rem:uniq}
We also mention the following fact: a self-conjugacy of a vector field that fixes a non-singular point is the identity. It follows that a $0$-fixing self-conjugacy of a germ of unfolding in our sense (not necessarily unitary) is the identity. As a consequence a parameter-fixing and $0$-fixing conjugacy between germs of unfoldings in our sense is \emph{unique}.
\end{remark}

\subsection{Statements}\label{sub:statements}

We will give a complete classification for the following two equivalence relations stemming from \Cref{def:conj,def:u} and identify the corresponding moduli space:
\begin{itemize}
\item conjugacy under a unipotent change of variable preserving the parameter,
\item conjugacy under a unipotent change of variable and parameter.
\end{itemize}
The first one is interesting in itself, but also will be used in the second one.

We prove in \Cref{sub:eig} that the eigenvalues of the $k+1$ singularities of a unitary generic $\tilde \omega_\eps$ are given by the values at the $(k+1)$-th roots of $\eps$ of a holomorphic function that we call the \emph{natural eigenvalue function}.

\begin{theorem}[Classification without change of parameters]
Two unitary generic families $\omega_\epsilon$ and $\widetilde \omega_{\epsilon}$ are conjugate near $(0,0)$ by a unipotent change of variable $z\mapsto \tilde z = \tilde z (z,\epsilon)$ fixing the parameter if and only if they have the same germ of natural eigenvalue function.
\end{theorem}

As a consequence we get \Cref{prop:class} in \Cref{sub:canon}, which we copy here:
\begin{proposition}
Consider two unitary generic families of vector fields $\omega$ and $\bar \omega$ and denote their respective canonical eigenvalue functions by $\lambda$ and $\bar \lambda$. The families are conjugate by a unipotent change of variable and parameter if and only if there exists a holomorphic map $\xi$ with $\xi'(0)=1$ that commutes with the rotation by $2\pi/(k+1)$ and such that the following holds near $0$:
\begin{equation}
\lambda = \bar \lambda \circ \xi.
\end{equation}
\end{proposition}

And \Cref{thm:mod_fix}:
\begin{theorem*}
The modulus space of germs of unitary generic families, up to  unipotent conjugacy fixing the parameter, is naturally in bijection with the set of of holomorphic functions of the form
\[\lambda(\delta) = (k+1)\delta^k\sigma(\delta)\]
with $\sigma(0) = 1$. 
\end{theorem*}

Finally, we identify a canonical parameter by defining a canonical representative in the set of eigenvalue functions up to the equivalence relation induced by the above proposition: this is part of \Cref{thm:canon_param} in \Cref{sub:canon}.

\begin{theorem*}[Canonical parameter]
Let $\omega_\eps$ be a unitary generic family of vector fields.
There exists a unique analytic change of parameter, $\eps\mapsto \tilde{\eps}$, and there exists a non-unique change of variable that depends on the parameter, such that:
\begin{itemize}
\item the change of variable and parameter is unipotent,
\item the conjugate $\tilde \omega$ of $\omega$ has the form $\tilde\omega_{\tilde \eps}(\tilde z) = (\tilde z^{k+1}-\tilde \eps) (1+O(\tilde z,\tilde\eps))$;
\item the \emph{natural} eigenvalue function $\tilde \lambda(\tilde \delta) = (\tilde\omega_{\tilde \delta^{k+1}})'(\tilde \delta)$ takes the form $\tilde{\lambda}(\tilde{\delta})= (k+1)\tilde{\delta}^k(1+O(\tilde{\delta}))$ with $\tilde{\delta}^{k+1}=\tilde{\eps}$, where $\tilde{\lambda}(\tilde{\delta})$  contains no terms in $\tilde{\delta}^{k+m(k+1)}$ for $m\geq1$.
\end{itemize}
\end{theorem*}

This will enable us to define notions of \emph{canonical parameter} and of \emph{prepared} germs of families. 
In particular two prepared germs of families are unipotently conjugate if and only if there exists a unipotent conjugacy preserving the parameter.

\medskip

\subsection{About the position of the singularities: positioned coordinate}\label{sub:chvar_roots}

We saw in \Cref{lem:roots:asymp} that the singularities are approximately on the vertices of a regular $(k+1)$-gon centered on the origin.
In fact a stronger statement is known to hold: we can make a change of variable independent of the parameter and placing them exactly on those vertices, for all $\eps$ sufficiently small.
Let us justify this.

We first recall a general feature of equations near a multiple root.

\begin{proposition}\label{prop:zeroes}
  Consider an analytic family of equations $u_\epsilon(z)=0$ defined in a neighborhood of $(z,\epsilon)=0$ such that $u_0(z)$ has a root of order $d$ at the origin. Make the following genericity assumption: $\frac{\partial u}{\partial \epsilon}(0,0) \neq 0$. Then there exists a change of variable $z\mapsto \tilde z$ fixing $0$, independent of the parameter and a restriction of $u_\eps$ in a neighborhood of $(0,0)$, such that for all $\epsilon$ the roots of $u_\eps$ are exactly the solutions of $\tilde z^{d}-\epsilon=0$.
  If moreover the coefficient of $u_0(z)$ in $z^d$ is $1$ and $\frac{\partial u}{\partial \epsilon}(0,0)= -1$, then the change of variable can be taken to have derivative $1$ at $0$. 
\end{proposition}
\begin{proof} There are many ways to prove this, one goes as follows. By the implicit function theorem, the zero locus of $u$ can be locally parameterized as $\epsilon = f(z)$ for some holomorphic $f$. 
Hence
\begin{equation}
\left.u_\eps\right|_{\eps=f(z)}\equiv 0.\label{eqf}
\end{equation}
Since $u_\eps(z)= *z^d+*\eps(1+O(z,\eps))+O(z^{d+1})$ ($*$ denotes non-zero constants), it follows by substituting $f(z)=\sum_{n\geq 1} a_nz^n$ in \eqref{eqf}
that $f$ has a root of exact order $d$ at the origin.
Then $f(z)= z^dh(z)= \tilde{z}^d$ for $\tilde{z}=g(z)=zh^{1/d}(z)$ where $g$ is analytic, since $h(0)\neq0$.
The conditions of the last sentence of the statement tell that $u_\eps(z)= z^d-\eps+O(z\eps,\eps^2,z^{d+1})$, which leads to $h(0)=1$ and we can thus take $g'(0)=1$.
\end{proof}

In the situation of the proposition above, the function $u_\eps$ factors by $\tilde z^d-\eps$ and the quotient does not vanish near $(0,0)$.

\begin{corollary}[positioned coordinate]\label{cor:factor}
Let $\omega_\eps$ be a germ of unitary generic one-parameter analytic family of vector fields unfolding a parabolic germ. 
Then there exists a local change of variable $z\mapsto \tilde{z}= g(z)$ independent of the parameter, such that $g'(0)=1$, and bringing the family to the form
\begin{equation}
\dot{\tilde{z}} = \widetilde{\omega}_\eps(\tilde{z})=(\tilde{z}^{k+1}-\eps) v_\eps(\tilde{z}).
\label{form_omega}
\end{equation}
Then $v_\eps(\tilde{z})=1+O(\tilde{z},\eps)$.
\end{corollary}
\begin{proof}
Before justifying the corollary, let us note that for any change of variable $\tilde z = g(z) = \alpha z + O(z^2)$ we have (by $\tilde \omega_\eps(\tilde z) = \omega_\eps(z) \frac{\partial \tilde z}{\partial z}$)
\[ \tilde \omega_\eps(\tilde z) = \alpha^{-k} \tilde z^{k+1} - \alpha \eps + O(\tilde z^{k+2},\tilde z \eps, \eps^2)
.\]
The claim then follows from \Cref{prop:zeroes} applied to $u=\omega$.
\Cref{prop:zeroes} tells us we can take $g'(0)=1$ and then the computation above gives that $v_{0}(0) = 1$.
\end{proof}

We say that the vector field is in \emph{positioned form}. Note that the conjugacy provided by \Cref{cor:factor} is unipotent and $0$-fixing, hence unitary.

\subsection{About the eigenvalues}\label{sub:eig}

The eigenvalue of the vector field $\omega$ at a singularity $z$ is just the complex number $\lambda=\omega'(z)$. Its dynamical significance is that the flow $\Phi_t$ then fixes $z$ with multiplier $\Phi_t'(z)=\exp(\lambda t)$.

Eigenvalues are invariant by a change of variable. How do they depend on $\epsilon$, and what are the collections that actually arise?

\begin{definition}\label{def:nat}
Given a vector field in positioned form in the sense of the previous section, i.e.\ 
\begin{equation}\label{form_factor}
\dot z = \omega_\eps(z) = (z^{k+1}-\eps) v_\eps(z),
\end{equation}
with $v_\eps(0) = 1$, we define the \emph{natural eigenvalue function} $\lambda$ as follows: $\lambda(\delta)$ is the eigenvalue of $\omega_\eps$ at the singularity $z=\delta$ for $\eps=\delta^{k+1}$:
\[\lambda(\delta) = \omega_{\delta^{k+1}}'(\delta)
.\]
\end{definition}

\begin{remark}\label{rem:lam:1}
An elementary computation yields
\begin{equation}\label{eq:lam:0}
\lambda(\delta)= (k+1)\delta^kv_{\delta^{k+1}}(\delta)
.
\end{equation}
In particular:
\begin{equation}\label{eq:lam:1}
  \lambda(\delta)\sim(k+1)\delta^k
\end{equation}
as $\delta \to 0$.
\end{remark}

\begin{remark}\label{rem:lam:2}
The whole collection of eigenvalues of $\omega_\eps$ is given by the values of $\lambda$ at the $(k+1)$-th roots of $\eps$. Note here a remarkable fact: the collection is given by a \emph{single} function $\lambda$. This can be explained in Ribon's general setting from the fact that the singular set is a smooth curve in $(z,\eps)$-space, hence has only one local component at $(0,0)$.
\end{remark}

We now extend the notion of natural eigenvalue function to generic families that are not in special form, using \Cref{cor:factor}:

\begin{definition}\label{def:eig}
Given a unitary generic vector field $\omega_\eps$ as per \Cref{def:gen_vf}, consider the 
coordinate $\tilde z$ provided by \Cref{cor:factor}.
Consider then the natural eigenvalue function $\lambda$ of the conjugated family under the change of variable $z \mapsto \tilde z$: this function is called the \emph{canonical eigenvalue function} of $\omega_\eps$. 
\end{definition}

Consider the canonical eigenvalue function $\lambda$ of a unitary generic vector field as we just defined. By invariance of eigenvalues, $\lambda(\delta)$ is the eigenvalue of some singularity of $\omega_\eps$ when $\eps = \delta^{k+1}$ and, as in \Cref{rem:lam:1}, the set of all eigenvalues of $\omega_\eps$ is given by the values of $\lambda$ at the $(k+1)$-th roots of $\eps$.

\begin{proposition}[form of the eigenvalue functions]\label{prop:eig}
Let $\omega_\eps$ be a unitary germ of generic $1$-parameter family of vector fields in the sense of \Cref{def:gen_vf}. 
\begin{enumerate}
\item If $\lambda$ is the canonical eigenvalue function of $\omega_\eps$, then $\lambda$ has a zero of order $k$ at the origin: 
\begin{equation} \lambda(\delta)= (k+1) \delta^k \sigma(\delta), \label{form_lambda}\end{equation}
with $\sigma(0) = 1$.
\item Conversely any function $\lambda$ of the form \eqref{form_lambda} with $\sigma(0)=1$ is the canonical eigenvalue function of some unitary generic $\omega_\eps$.
\end{enumerate}
\end{proposition}
\begin{proof}\hfill
\begin{enumerate}
\item
By definition, there exists a coordinate change, independent of the parameter, for which $\omega_\eps$ takes the form $(z^{k+1}-\eps)v_\eps(z)$ where $v_0(0)=1$, so $v_\eps(z)$ does not vanish in a neighborhood of $(0,0)$.
By \cref{eq:lam:0}: $\lambda(\delta)= (k+1)\delta^kv_{\delta^{k+1}}(\delta)$. This gives the first claim,  letting $\sigma(\delta)=v_{\delta^{k+1}}(\delta)$.
\item
For the converse, it is enough to consider $\omega_\epsilon :=  ( z^{k+1}-\epsilon) \sigma(z)$ for some $\sigma$ that does not depend on $\epsilon$. Then $\lambda(\delta)=(k+1)\delta^{k}\sigma(\delta)$.
\end{enumerate}
\end{proof}

\subsection{Classification under unitary conjugacies that fix the parameter}\label{sub:parfixclass}

We start with a preliminary result that will be used several times.
\begin{lemma}\label{lem:lambda:eq}
Assume that $\lambda$ and $\tilde \lambda$ are two holomorphic functions satisfying $\lambda(\delta) \sim (k+1)\delta^k \sim \tilde\lambda(\delta)$ as $\delta \to 0$.
Assume moreover that
\[\setof{ \lambda(\delta) }{ \delta^{k+1}=\eps } = \setof{ \tilde \lambda(\delta) }{ \delta^{k+1}=\eps }\]
Then $\lambda=\tilde\lambda$ near $0$.
\end{lemma}
\begin{proof}
For every small $\delta$, there exists $u=u(\delta)\in\U_{k+1}$ such that $\tilde\lambda(\delta)=\lambda(u\delta)$;
the complex number $u=u(\delta)$ a priori depends on $\delta$ but using that $\lambda(u(\delta)\delta)\sim u(\delta)^{-1}\lambda(\delta)$ as $\delta\to 0$, we get that $u(\delta) \sim \lambda(\delta)/\tilde\lambda(\delta) \tend \sigma(0)/\tilde\sigma(0) = 1$, so $u(\delta)=1$ for $\delta\neq 0$ small enough.
\end{proof}

As already mentioned, the collection of the eigenvalues at the singularities is invariant under holomorphic conjugacies.
Are they the only invariants in our setting of generic germs of unfoldings? We will see in \Cref{thm:lambda_inv:AB} that the answer is yes.
The version below is adapted to the classification under unipotent change of variables and is proved by a powerful method, which deserves to be more known. We provide a second proof in the appendix, by the classical path method.
Note that in \cite{Ri}, Proposition~5.10, gives a much more general result (in an abstract setting, in which the simple cases are hard to figure out at first reading).

\begin{theorem}[Classification without change of parameters]\label{thm:lambda_inv}
Two unitary generic families $\omega_\epsilon$ and $\widetilde \omega_{\epsilon}$ are conjugate near $(0,0)$ by a unipotent change of variable $z\mapsto \tilde z = \tilde z (z,\epsilon)$ fixing the parameter if and only if they have the same germ of canonical eigenvalue function.
\end{theorem}

\begin{proof}
The direct implication follows from \Cref{lem:lambda:eq}.

For the converse, thanks to \Cref{cor:factor}, we can first perform a unitary change of variable on each family, independent of the parameter, so that in each case the singular values are the solutions of $z^{k+1}=\eps$, and so that $\omega_\eps = (z^{k+1}-\eps)v_\eps(z)$ and $\tilde \omega_\eps = (z^{k+1}-\eps)\tilde v_\eps(z)$ with $v_0(0) = \tilde v_0(0)$.

We then use a method which was probably first developed in the 2-dimensional case by Teyssier \cite{Teyssier}. We do not know if this method existed earlier in the 1-dimensional case.
See \Cref{app:path} for an alternative proof of the converse, using the \emph{path method}.

For $t\in\C$, denote by $z\mapsto F(\eps,t,z)$ the time-$t$ map of the vector field $\omega_\eps$.
We will seek the conjugacy (actually, the inverse conjugacy, from $\tilde \omega_{\eps}$ to $\omega_\eps$) in the following form: $\phi_\eps(z) = F(\eps, g(z,\eps), z)$ for a well-chosen function $g(z)$.
By \cite{CMR}, Proposition~2.10 page~507, the pull-back of
$V_\eps = \omega_\eps(z) \frac{\partial}{\partial z} $ by $\phi_\eps$ is
\begin{equation}\label{eq:cj}
\frac{1}{1+V_\eps\cdot g_\eps} V_\eps
\end{equation}
where $V_\eps$ is viewed as a differential operator acting on the function $g_\eps(z) = g(z,\eps)$.
Here we want to turn $V_\eps=(z^{k+1}-\eps) v_\eps(z) \partial/\partial z$ into $\tilde V_\eps=(z^{k+1}-\eps) \tilde v_\eps(z) \partial/\partial z$.
Note that $V_\eps = \frac{v_\eps}{\tilde v_\eps} \tilde V_\eps$.
By \cref{eq:cj} the function $\phi_\eps$ conjugates $V_\eps$ to $\tilde V_\eps$ if and only if
\[ 1+V_\eps\cdot g_\eps = \frac{v_\eps}{\tilde v_\eps}
.\]
Since $V_\eps\cdot g_\eps = (z^{k+1} - \eps) v_\eps g_\eps'$,
we obtain the following equation for $g$:
\[ g_\eps' = \frac{\tilde v_\eps-v_\eps}{(z^{k+1} - \eps)\tilde v_\eps v_\eps}
.\]
(This can be re-expressed as $g_\eps' = \frac{1}{\omega_\eps} - \frac{1}{\tilde \omega_\eps}$).
Note that when $z^{k+1}-\eps=0$ then $z$ is a zero of both vector fields, and by hypothesis they have the same eigenvalues at this point. This implies that $v_\eps(z)|_{z^{k+1}-\eps=0}=
\tilde{v}_\eps(z)|_{z^{k+1}-\eps=0}$, hence $\tilde v_\eps-v_\eps$ factors by $(z^{k+1} - \eps)$.
Therefore the right hand side of the equation above is holomorphic and vanishes at $z=0$ and is $O(\eps)$ at $z=0$.
So $g_\eps$ can be obtained by a simple integration.
Taking the antiderivative that vanishes at $0$ provides a $0$-fixing change of variable.
\end{proof}

Note that we actually obtain a $0$-fixing change of variable and parameter.

As a corollary:

\begin{theorem}\label{thm:cl_e}
For any unitary generic family $\omega_\epsilon$, let $\lambda(\delta) = (k+1)\delta^k \sigma(\delta)$ its canonical eigenvalue function.
Then there exists a unipotent conjugacy preserving the parameter between $\omega_\epsilon$ and the vector field
\[\dot z = \Omega_\epsilon(z) = (z^{k+1}-\epsilon)\sigma(z)
.\]
\end{theorem}
\begin{proof}
Let $\epsilon = \delta^{k+1}$. Then $\delta$ is a singularity of $\Omega_\epsilon$ and its eigenvalue is $\partial \Omega_\epsilon/\partial z = (k+1) \delta^k \sigma(\delta)$ at $z=\delta$. So $\omega_\epsilon$ and $\Omega_\epsilon$ have a common eigenvalue function and we conclude using \Cref{thm:lambda_inv}.
\end{proof}

\Cref{thm:lambda_inv} also allows to determine the modulus space for Equivalence Problem~1:
\begin{theorem}\label{thm:mod_fix}
The modulus space of germs of unitary generic 1-parameter unfoldings of codimension $k$ parabolic vector fields, up to unipotent conjugacy fixing the parameter, is naturally in bijection with the set of holomorphic functions of the form
\[\lambda(\delta) = (k+1)\delta^k\sigma(\delta)\]
with $\sigma(0) = 1$. 
\end{theorem}

\subsection{Classification under general unitary conjugacies, canonical parameter}\label{sub:canon}

Here we study the effect of a change of parameter $\epsilon$ on the eigenvalue function $\lambda$ and provide a unique canonical normal form for $\lambda$. We also introduce a notion of \emph{prepared family} and draw consequences in terms of classification of $\omega_\epsilon$ by unitary conjugacy.

For the discussion below we ask the reader to keep in mind that analytic maps $\xi$ that commute near the origin with the rotation by $2\pi /(k+1)$, which we denote by $R_{2\pi/(k+1)}$, have a power series expansion whose only non-zero terms have exponent belonging to $1+(k+1)\Z$. 
Moreover, given $p(\delta)=\delta^{k+1}$, there existe a holomorphic map $\phi$ such that $p\circ \xi= \phi\circ p$.
Conversely, to any holomorphic germ $\phi$ corresponds exactly $(k+1)$ analytic germs $\xi$ satisfying $p\circ \xi = \phi \circ p$, they commute with $R_{2\pi/(k+1)}$ and for all $\eps$ small enough
\begin{equation}\label{eq:sc}
\setof{ \tilde\delta }{ \tilde\delta^{k+1} = \phi(\eps) }
= \setof{ \xi(\delta) }{ \delta^{k+1} = \eps }
= \setof{ \nu\xi(\delta_0) }{ \nu^{k+1} = 1 }
\end{equation}
where $\delta_0$ is any of the $(k+1)$-th roots of $\eps$.
The $k+1$ maps $\xi$ associated to a given $\phi$ differ by composition with a power of $R_{2\pi/(k+1)}$ (by commutation, the composition on the left or on the right are the same).
If $\phi'(0)=1$ then there is a unique associated $\xi$ with $\xi'(0)=1$, which we call the \emph{canonically associated} map.

\begin{proposition}\label{prop:class}
Consider two families of vector fields $\omega$ and $\tilde \omega$ that are generic and unitary in the sense of \Cref{def:unitary}. Consider their respective canonical eigenvalue functions $\lambda$ and $\tilde \lambda$. The families are conjugate by a unipotent change of variable and parameter if and only if there exists a holomorphic map $\xi$ that commutes with the rotation by $2\pi/(k+1)$, with $\xi'(0)=1$ and such that the following holds near $0$:
\begin{equation}
\lambda = \tilde \lambda \circ \xi.
\end{equation}
\end{proposition}
\begin{proof}
Let us prove that the condition is necessary.
Consider a unipotent change of variable and parameter
$(z,\eps)\to(\tilde z,\tilde\eps)$, with $\tilde\epsilon = \phi(\eps)$. 
By invariance of eigenvalues under conjugacy, for all $\eps$ small enough, and $\tilde \eps=\phi(\eps)$,
\begin{equation}\label{eq:tud}
\setof{ \lambda(\delta) }{ \delta^{k+1} = \eps }
=
\setof{ \tilde\lambda(\tilde\delta) }{ \tilde\delta^{k+1} = \tilde\eps }
.\end{equation}
By \Cref{eq:lam:1} $\lambda(\delta)\sim (k+1)\delta^k$ as $\delta \to 0$ and $\tilde\lambda(\tilde \delta)^k \sim (k+1)\tilde\delta^k$ as $\tilde\delta \to 0$.
Consider a canonically associated holomorphic change of parameter $\tilde\delta = \xi(\delta)$ semi-conjugate to the change of parameter $\tilde\epsilon=\phi( \eps)$ with $\xi'(0)=1$, as explained above.
By \cref{eq:sc,eq:tud}, for all $\delta_0$ small enough
\[ \setof{ \lambda(\mu\delta_0) }{ \mu^{k+1} = 1}
= \setof{ \tilde\lambda(\nu \xi(\delta_0)) }{ \nu^{k+1} = 1 }
= \setof{ \tilde\lambda(\xi(\nu  \delta_0)) }{ \nu^{k+1} = 1 }
\]
since $\xi$ commutes with the multiplication by any $(k+1)$-th root of $1$.
Note that $\tilde\lambda\circ \xi (\delta) \sim (k+1)\delta^k \sim \lambda(\delta)$ as $\delta\to 0$.
By \Cref{lem:lambda:eq}, $\lambda = \tilde\lambda \circ \xi$.

For the sufficiency of the condition, we proceed by reduction to the parameter fixing classification (\Cref{thm:lambda_inv}).
Indeed, if $\lambda = \tilde \lambda \circ \xi$ then let us perform the semi-conjugate change of parameter $\bar\eps = \phi(\eps) =\xi(\!\!\sqrt[k+1]\eps)^{k+1}$ on the first family $\omega$. This gives a family $\bar \omega$. Now the two families $\tilde\omega$ and $\bar \omega$ have the same canonical eigenvalue function, which is $\tilde\lambda$. By \Cref{thm:lambda_inv} there is thus a further unipotent change of variable fixing the parameter that brings $\tilde\omega$ to $\bar \omega$.
\end{proof}

The classification is thus reduced to a problem of classifying germs of functions $\lambda$ up to composition as follows: let
\begin{itemize}
\item $\cal G_k$ be the set of germs of functions $\lambda$ with $\lambda(\delta) \sim (k+1)\delta^k$ as $\delta\to 0$.
\item $\Xi_k$ be the group of germs $\xi$ that commute with the rotation by $2\pi/(k+1)$ and such that $\xi'(0)=1$.
\end{itemize}
Then $\Xi_k$ acts on $\cal G_k$ by right-composition.
In this language, \Cref{prop:class} can be rephrased as: 
\begin{corollary}
The modulus space of germs of unitary generic 1-parameter unfoldings of codimension $k$ parabolic vector fields, up to conjugacy by a unipotent change of variable and parameter, is naturally in bijection with $\cal G_k/\Xi_k$.
\end{corollary}

We can now seek a unique representative of the orbits of $\Xi_k$ in $\cal G_k$. One natural choice is provided in \Cref{prop:mod_space}, which we state and prove now.

\begin{proposition}[Canonical representative]\label{prop:mod_space}
A set of unique representatives for the equivalence classes of the quotient $\cal G_k/\Xi_k$ is given by the subset $\cal G'_k\subset \cal G_k$ which consists in the germs of analytic functions $\lambda(\delta)$ with first term $(k+1)\delta^k$ and with all terms of order $k+m(k+1)$ vanishing for $m\geq 1$, i.e.
\begin{equation}
\lambda(\delta) = \delta^k\left(k+1+ \sum_{j=1}^{k} \delta^j a_j(\delta^{k+1})\right) =(k+1) \delta^k \sigma(\delta)
\label{nf_lambda}
\end{equation}
with $a_j$ analytic.
\end{proposition}

It follows from the following statement:

\begin{lemma}\label{lemma_lambda}
For any holomorphic function $\lambda: (\C,0)\rightarrow (\C,0)$ of the form $\lambda(\delta) = (k+1)\delta^k\sigma(\delta)$ with $\sigma(0)=1$, there exists a unique germ of holomorphic diffeomorphism $h: (\C,0)\rightarrow (\C,0)$ of the form $h(\eta) =\eta g(\eta^{k+1})$ with $g$ analytic (in other words, $h$ commutes with the multiplication by any $(k+1)$-th root of unity), such that $h$ is tangent to the identity (i.e.\ $g(0)=1$) and such that all terms of the power series expansion of $\tilde\lambda = \lambda\circ h$ of degree $k+m(k+1)$ vanish for $m\geq 1$.
\end{lemma}

\begin{proof} 
Let us group the terms in the power series expansions of $\lambda$ and $\lambda\circ h$ according to the class modulo $k+1$ of their exponent.
This yields  $\lambda(\delta) =  (k+1) \delta^k \sum_{j=0}^k \delta^j a_j(\delta^{k+1}) $,  and $\lambda\circ h(\eta) =(k+1)  \eta^k\sum_{j=0}^k \eta^j b_j(\eta^{k+1})$   for some holomorphic germs $a_j$ and $b_j$ with $a_0(0)= b_0(0)=1$.  Substituting $h(\eta) = \eta g(\eta^{k+1})$ for $\delta$, one realizes that for all $j$, the function $b_j$ depends on $a_j$ but not on the other $a_i$: more precisely if we let $\zeta=\eta^{k+1}$ then 
\begin{equation}\label{eq:ajbj}
 b_j(\zeta) = g(\zeta)^{k+j} a_j(\zeta g(\zeta)^{k+1}) .
\end{equation}
The problem is thus equivalent to finding $g$ analytic such that $g(0)=1$ and  $(\eta g(\eta^{k+1}))^k a_0((\eta g(\eta^{k+1}))^{k+1}) = \eta^k$. 
Let $\ell(\delta) =  \delta \left(a_0(\delta^{k+1})\right)^{\frac1k}$  be the holomorphic germ where we take the branch of $k$-th root that maps $1$ to $1$. The function $\ell$ is analytic and invertible since $\ell'(0)=1$.
Also $$(\eta g(\eta^{k+1}))^k a_0((\eta g(\eta^{k+1}))^{k+1}) =\left( \ell(\eta g(\eta^{k+1}))\right)^k.$$
Hence the problem is equivalent to  $\eta \nu= \ell(\eta g(\eta^{k+1}))$ for some solution of $\nu^k=1$. Derivating with respect to $\eta$ at $\eta=0$ yields $\nu=1$.
The equation thus becomes $\eta = \ell(\eta g(\eta^{k+1}))$, i.e.\ $g(\eta^{k+1}) = \frac{\ell^{-1}(\eta)}{\eta}$, where $\frac{\ell^{-1}(\eta)}{\eta}$ is analytic. We need to prove that $\frac{\ell^{-1}(\eta)}{\eta}$ is a function of $\eta^{k+1}$ alone. From its definition, $\ell$ has the form 
$\ell(\delta)= \delta b(\delta^{k+1})$ for some analytic function $b$ with $b(0)=1$, yielding that $\ell^{-1}$ has the same form. Hence, $\frac{\ell^{-1}(\eta)}{\eta}$ is a function of $\eta^{k+1}$ alone, yielding the result.
\\
Uniqueness. Suppose that $h_1, h_2$ are two solutions. Then $h_2\circ (h_1)^{-1}$ sends $\tilde{\lambda}_1=\lambda\circ h_1$ to $\tilde{\lambda}_2=\lambda\circ h_2$. This is only possible if $h_2\circ (h_1)^{-1}$ is linear: to prove this, let $h_2\circ (h_1)^{-1}(\eta) = \eta g(\eta)$ and use \eqref{eq:ajbj} with $j=0$.
\end{proof}

\begin{theorem}[Canonical parameter]\label{thm:canon_param}
Let $\omega_\eps$ be a unitary generic\footnote{see \Cref{def:gen_vf,def:unitary}} family of vector fields. There exists a unique analytic change of parameter, $\eps\mapsto \tilde{\eps}$, and there exists a non-unique change of variable that depends on the parameter, such that:
\begin{itemize}
\item the change of variable and parameter is unipotent,
\item the conjugate $\tilde \omega$ of $\omega$ is positioned, i.e.: $\tilde\omega_{\tilde \eps}(\tilde z) = (\tilde z^{k+1}-\tilde \eps) (1+O(\tilde z,\tilde\eps))$;
\item the natural eigenvalue function $\tilde \lambda$ (we recall it is defined by $\tilde \lambda(\tilde \delta) = (\tilde\omega_{\tilde \delta^{k+1}})'(\tilde \delta)$) takes the form $\tilde{\lambda}(\tilde{\delta})= (k+1)\tilde{\delta}^k(1+O(\tilde{\delta}))$ with $\tilde{\delta}^{k+1}=\tilde{\eps}$, where $\tilde{\lambda}(\tilde{\delta})$  contains no terms in $\tilde{\delta}^{k+m(k+1)}$ for $m\geq1$.
\end{itemize}
The parameter $\tilde{\eps}$ will be called the \emph{canonical parameter}.
A family as above, i.e.\ whose parameter is canonical and which is positioned will be called \emph{prepared}.
\end{theorem}

This is a direct consequence of \Cref{lemma_lambda} and \Cref{cor:factor}.

\subsection{Normal forms}\label{sub:normal}

There are several notions of classification we can be interested in and, for each notion, there are interesting normal forms. Here we list a few, summing up what was done in the previous sections and adding a few more.

We still assume that $\omega_\eps$ is, in the sense of \Cref{def:unitary}, a unitary generic $1$-parameter unfolding of a vector field having initially a parabolic singularity: $\omega_0(z) = z^{k+1} + O(z^{k+2})$, $k\geq 1$.

We saw in \Cref{cor:factor} that for $\eps$ small enough and in a uniform neighborhood of $0$ in the $z$-coordinate, $\omega_\eps$ can be turned by a change of variable $z$ independent of the parameter, fixing $0$ and with derivative $1$ at $0$, into a (non-unique) form (that we called a positioned form)
\begin{equation}
\dot z = (z^{k+1}-\epsilon) \times h_\eps(z), \label{family_not_normal}
\end{equation}
for some function $h_\eps$, i.e.\ we can place the singularities exactly on the solutions of $z^{k+1}=\eps$.
However, this is \emph{not a normal form}: there are many possible changes of variables, leading to the function $h_\eps$ being highly non-unique.

\subsubsection{The local normal form}\label{subsub:tlnf}

\begin{theorem}
Let $\omega_\eps$ be a unitary generic unfolding in the sense of \Cref{def:unitary}. There exists a unitary change of coordinate and parameter to the unique normal form
\begin{equation}\label{unf}
\begin{gathered}
\dot{\tilde{z}} =\widetilde{\omega}_{\tilde\eps}(\tilde z)= (\tilde z^{k+1}-\tilde \eps)\times\sigma(\tilde z)
\text{, where }
\\
\sigma(\tilde z)= 1+\sum_{j=1}^k a_j(\tilde z^{k+1})\tilde z^j
\end{gathered}
\end{equation}
i.e.\ $\sigma(\tilde z)$ is independent of the parameter and its power series expansion at $0$ has no terms of exponent in $(k+1)\N^*$.
In this form the parameter is canonical.
\end{theorem}
\begin{proof} We first use \Cref{thm:canon_param} to transform the family to a form $\overline{\omega}_{\tilde\eps}(\bar{z})$
in which the parameter $\tilde{\eps}$ is canonical and the natural eigenvalue function $\lambda(\delta)=(k+1)\delta^k\sigma(\delta)$ contains no term of degree $k+m(k+1)$. 
Composing with a second change of coordinate provided in \Cref{thm:cl_e} (and independent of the canonical parameter $\tilde{\eps}$)  yields the result.

For uniqueness, note that a family in the form \cref{unf} has its parameter which is canonical, since its eigenvalue at $\tilde z=\tilde \delta$ for parameter $\tilde\eps = \tilde\delta^q$ is $(k+1)\tilde\delta^k\sigma(\tilde\delta)$.
Since $\sigma$ is deduced from $\lambda$, we conclude using uniqueness of the canonical parameter (proved in \Cref{thm:canon_param}).
\end{proof}

\medskip

Two other possibilities for which the singular points are located at the roots of $z^{k+1} -\eps$ are given by polynomial and rational vector fields on $\CP^1$, which we explore below.

\subsubsection{The polynomial normal form}

The \emph{polynomial normal form} has up to $k$ additional singular points on $\CP^1$ and a pole at infinity.

\begin{theorem}\label{thm-poly-normal-form}
Let $\omega_\eps$ be a unitary generic unfolding in the sense of \Cref{def:unitary}.

\begin{enumerate}
\item There exists a unipotent change of variable preserving the parameter $\omega_\eps$ to the unique polynomial normal form
\begin{equation}\dot{\tilde{z}} 
= \widetilde{\omega}_{\eps}(\tilde{z})=(\tilde{z}^{k+1}-\epsilon) \times Q_{\epsilon}(\tilde{z} ),
\label{pol_normal_form:1}
\end{equation}
where $Q_{\epsilon}$ is a polynomial in $\tilde{z}$ of degree at most $k$, and $Q_{0}(0)= 1$.
\item In this form the parameter is canonical in the sense of \Cref{thm:canon_param} if and only if $Q_\eps(0) =1$ in a neighborhood of $\eps=0$.
\item There exists a unipotent change of coordinate and parameter bringing $\omega_\eps$ to the unique polynomial normal form
\begin{equation}\dot{\tilde{z}} = \widetilde{\omega}_{\tilde{\eps}}(\tilde{z})=(\tilde{z}^{k+1}-\tilde{\epsilon}) \times Q_{\tilde{\epsilon}}(\tilde{z} ),
\label{pol_normal_form:2}
\end{equation}
where $Q_{\tilde{\epsilon}}$ is a polynomial in $\tilde{z}$ of degree at most $k$, $Q_{\tilde{\epsilon}}(0)\equiv1$ and $\tilde{\epsilon}$ is the unique canonical parameter of \Cref{thm:canon_param}.
\end{enumerate}
\end{theorem}

\begin{proof}
To prove claim (1), thanks to \Cref{thm:lambda_inv} it is enough to choose the coefficients of the polynomial $Q_{\tilde{\epsilon}}$ so that the eigenvalues of $\tilde{\omega}_\eps$ are the same as those of $\omega_\eps$. 

A polynomial of degree at most $k$ being uniquely determined by its values at $k+1$ points, there is a unique solution for $\eps\neq 0$. If $\delta_i$, $i=1, \dots, k+1$ are the singular points and $w_i=\sigma(\delta_i)$, such a polynomial is given by the Lagrange interpolation formula $$Q_\eps(z)= \sum_{i=1}^{k+1} w_i\prod_{j\neq i} \frac{(z-\delta_j)}{(\delta_i-\delta_j)}.$$ It depends holomorphically on $\eps\neq0$ and is well-defined (single-valued) since it is invariant under permutations of the $\delta_i$.  To prove the convergence when $\eps\to 0$ we consider the variables $\delta_j$ as independent, we let $\Delta = (\delta_1,\ldots,\delta_{k+1})$, and for all  $\Delta\in\C^{k+1}$ we introduce  this other formula for the Lagrange interpolation polynomial
$$Q_\Delta(z) := \frac{\left|\begin{matrix} 0&1&z&\cdots&z^k\\
\sigma(\delta_1) &1&\delta_1&\cdots& \delta_1^k\\
\vdots&\vdots&\vdots&\ddots&\vdots\\
\sigma(\delta_{k+1}) &1&\delta_{k+1}&\cdots& \delta_{k+1}^k\end{matrix}\right|}{\left|\begin{matrix} 1&\delta_1&\cdots& \delta_1^k\\
\vdots&\vdots&\ddots&\vdots\\
1&\delta_{k+1}&\cdots& \delta_{k+1}^k\end{matrix}\right|}.$$
Then the denominator is $\prod_{i< j}(\delta_j-\delta_i)$. Each $\delta_i-\delta_j$ divides the numerator since $w_i-w_j=O(\delta_i-\delta_j)$, and hence,  $\delta_i-\delta_j$ divides the difference of the lines starting by $\sigma(\delta_i)$ and $\sigma(\delta_j)$.

Claim (2): In the normal form of \cref{pol_normal_form:1}, it is very easy to check if the parameter is canonical in the sense of \Cref{thm:canon_param}. 
Indeed, write $Q_\eps(z) = a_0(\eps) + a_1(\eps)z + \cdots +a_k(\eps )z^k$. 
The only terms in $\delta^{k+m(k+1)}$, $m\in\N$, in the eigenvalue function $\lambda(\delta) = (k+1) \delta^k Q_{\delta^{k+1}}(\delta)$ come from $a_0(\eps)$.
Hence the parameter is canonical if and only if $Q_\eps(0) = 1$ for all $\eps$ close to $0$.

Claim (3): A unipotent and $0$-fixing change of variable and parameter from form \eqref{pol_normal_form:1} to \eqref{pol_normal_form:2} is simply of the form
$(z,\eps)\mapsto (\tilde{z},\tilde{\eps})=\left(zQ_\eps^{1/k}(0), \eps Q_\eps(0)^{1+\frac1k} \right)$ (using the $k$-th root of $Q_\eps(0)$ closest to $1$), which is analytic in $(z,\eps)$ (the factor in front of $z$ is so that the vector field factors by $\tilde z^{k+1}-\tilde \eps$).
\end{proof}

\subsubsection{The rational normal form} 

\begin{theorem}\label{thm-rational-normal-form}
Let $\omega_\eps$ be a unitary generic unfolding in the sense of \Cref{def:unitary}. There exists a unipotent change of coordinate and parameter bringing $\omega_\eps$ to the unique \emph{rational normal form}
\begin{equation}
\dot{\tilde{z}} = \widetilde{\omega}_{\tilde{\eps}}(\tilde{z})=(\tilde{z}^{k+1}-\tilde{\epsilon})/ R_{\tilde{\epsilon}}(\tilde{z}),
\label{{rat_normal_form}}
\end{equation}
where $R_{\tilde{\epsilon}}$ is a polynomial in $\tilde{z}$ of degree at most $k$, $\tilde{\epsilon}$ is the unique canonical parameter of \Cref{thm:canon_param} and $R_0(0)=1$.
\end{theorem}
\begin{proof}
The idea of the proof is the same as that of \Cref{thm-poly-normal-form}. We can first transform $\omega_\eps$ through $(z,\eps)\mapsto (\tilde{\tilde{z}}, \tilde{\eps})$ to the form given by \Cref{thm:canon_param}, in which the parameter is canonical. The further change of coordinate will preserve the parameter.  The coefficients of the polynomial $R_\eps$ will be chosen so that the eigenvalues of $\tilde{\omega}_\eps$ are the same as those of $\omega_\eps$. For the Lagrange interpolation formula we use the function $1/\sigma(\delta)$. The second change of coordinate then follows from \Cref{thm:lambda_inv}.
\end{proof}

There is also an analogue of claims (1) (parameter fixing version) and (2) (characterization of ) of \Cref{thm-poly-normal-form}, which we leave to the reader.

\begin{remark}
It is impossible in \Cref{thm-rational-normal-form} to require $R_{\tilde{\eps}}(0)\equiv 1$ as shown by the following example: 
the unique normal form of \Cref{thm-rational-normal-form} for $\dot z= (z^2-\eps)(1+(1+\eps) z)$, (which is in the unique normal form described in \Cref{thm-poly-normal-form}) is $\dot{\tilde{z}} = \frac{\tilde{z}^2-\eps}{1+a(\eps) +b(\eps)\tilde{z}}$, with $a(\eps)= \frac{\eps(1+\eps)^2}{1-\eps(1+\eps)^2}$ and $b(\eps)= -\frac{1+\eps}{1-\eps(1+\eps)^2}$.
\end{remark}

\begin{remark}
In \Cref{thm-poly-normal-form}, we were able to give a simple criterion on the family of polynomials $Q_{\eps}(\tilde z)$, for the normal form $(\tilde z^{k+1}-\eps)\times Q_\eps(\tilde z)$ to have its  parameter that is canonical: namely that $Q_\eps(0)= 1$ holds in a neighborhood of $\eps=0$.
In \Cref{thm-rational-normal-form}, it is harder to identify for which families of polynomials $R_{\tilde\eps}(\tilde z)$ the normal form $(\tilde{z}^{k+1}-\tilde{\epsilon})/ R_{\tilde{\epsilon}}(\tilde{z})$ has its parameter that is canonical, since the calculations become messy.
But we could redefine the notion of canonical parameter by using $1/\lambda$ instead of $\lambda$ in the definition, i.e.\ we ask that the natural eigenvalue function $\lambda(\delta)$ is such that $1/\lambda(\delta)$ contains no terms in $\delta^{-k+m(k+1)}$ for $m\geq 1$.
\Cref{lemma_lambda} can be adapted to this case.
The new parameter would be canonical in this new sense if and only if $R_\epsilon(0)= 1$ holds in a neighborhood of $\eps=0$.
\end{remark}

\subsubsection{A Kostov-type normal form}

A fourth normal form, which we call \emph{Kostov-type normal form} is indirectly  given in Rib\'on's paper \cite{Ri},  in the proof of his Proposition~5.14:
\begin{equation}
\dot z = P_\epsilon(z)/(1 + A(\epsilon) z^k)
,
\end{equation}
where $A(\epsilon)$ denotes the sum of the inverses of the eigenvalues (by \Cref{lemma:periods} it is analytic in $\eps$), where $P_\epsilon$ is a monic centered polynomial in $z$ of degree $k+1$ with 
\begin{equation}
P_\eps(z)= z^{k+1} + b_{k-1}(\eps) z^{k-1} + \dots + b_1(\eps)z+b_0(\eps)
,
\end{equation}
and $P_0(z) = z^{k+1}$. This normal form is for parameter preserving changes of variables. If we further add a genericity condition (not present in \cite{Ri}), then we will see that a change of parameter can bring $b_0(\eps)= -\eps$. The uniqueness is not proved in \cite{Ri}. We address this question here.

\begin{theorem} Let $\omega_\eps$ be a unitary generic unfolding, in the sense of \Cref{def:unitary}.
There exists a unique unitary change of coordinate and parameter $(z,\eps)\mapsto (\tilde{z},\tilde{\eps})$ bringing $\omega_\eps$ to the unique normal form 
\begin{equation}
\dot{\tilde{z}} = P_{\tilde{\eps}}(\tilde{z})/(1 + A(\tilde{\eps}) \tilde{z}^k),
\end{equation}
with
\begin{equation}
P_{\tilde{\eps}}(\tilde{z})= \tilde{z}^{k+1} + b_{k-1}(\tilde{\eps}) \tilde{z}^{k-1} + \dots + b_1(\tilde{\eps})\tilde{z}-\tilde{\eps}.
\end{equation}
\end{theorem}

\begin{proof}
By \Cref{cor:factor} we can perform a first change of variable independent of the parameter, and of derivative $1$ at $0$, placing the family in positioned form, i.e.\ 
\[
\omega_\eps(z)= (z^{k+1}-\eps) v_\eps(z)
.\]
Below we assume that $\omega_\eps$ is already in positioned form.
The existence of the change of coordinate is then a simple consequence of a theorem of Kostov \cite{K}, as follows.

We enlarge the $1$-parameter family it to the generic $k$-parameter family $\dot z = \omega_{\eta}(z)= (z^{k+1}+ \eta_{k-1}z^{k-1} + \dots+\eta_1z-\eps) v_\eps(z)$ with the multi-parameter $\eta=(\eta_{k-1}, \dots, \eta_1,\eps)$.
Kostov's Theorem states that there exists a change of coordinate and parameters to a normal form
\begin{equation}\dot{\tilde{z}}=\widetilde{\omega}_{a}(z) = \widetilde{P}_{a}(\tilde z)/(1 + \tilde{A}(a) \tilde z^k) 
\label{form_Kostov}
\end{equation}
with a new multi-parameter $a=(a_{k-1}, \dots, a_1, a_0)$ and
$\widetilde P_a(\tilde z) = \tilde z^{k+1} + a_{k-1} \tilde z^{k-1} + \dots + a_1 \tilde z+a_0 $. The $a_j$ are functions of $\eta$. We then substitute $\eta_1= \dots= \eta_{k-1}=0$. Then the $a_j$ become functions of $\eps$ alone.
Since $\tilde \omega_0(z) = z^{k+1} + \cal O(z^{k+2})$, it follows that $\frac{\partial \tilde{z}}{\partial z}(0,0)$ is a $k$-th root of unity.
Composition of such a change with a rotation of order $k$ preserves the form \eqref{form_Kostov}. Hence it is possible to choose the change of coordinate so that $\frac{\partial \tilde{z}}{\partial z}(0,0)=1$.

Let us now use the genericity assumption. Since $a_0$ is the product of the singularities after the change of variable, and since, before the change of variables, the singularities have modulus $|\eps|^{1/(k+1)}\, \gg |\eps|$, they have the same order of magnitude as $\eps\to 0$.
So their product also does, thus $\frac{\partial a_0}{\partial \eps}\neq0$. Hence we can choose the new parameter $\tilde{\eps}=-a_0(\eps)$. Moreover from $\frac{\partial \tilde{z}}{\partial z}(0,0)=1$, we get $\frac{\partial \tilde{\eps}}{\partial \eps}(0)=1$.

The uniqueness follows from \Cref{thm:uniqueness} below.
\end{proof}

\begin{theorem}\label{thm:uniqueness} Let $\dot z = \omega_\eps(z) =P_\epsilon(z)/(1 + A(\epsilon) z^k)$ and $\dot{\tilde{z} } = \widetilde{\omega}_{\tilde{\eps}}(\tilde{z}) = \widetilde{P}_{\tilde{\epsilon}}(\tilde{z})/(1 + \widetilde{A}(\tilde{\epsilon}) \tilde{z}^k)$ be two $1$-parameter families of vector fields such that $P_\eps$ and $\widetilde{P}_{\tilde \eps}$ are monic centered of degree $k+1$, $P_0(z) =z^{k+1}$, $\widetilde{P}_0(\tilde{z})=\tilde{z}^{k+1}$, $P_\eps(0)=-\eps$, and $\widetilde{P}_{\tilde{\eps}}(0)=-\tilde{\eps}$. 
Suppose the two families are locally conjugate through a unipotent change of coordinate and parameter $(z,\eps)\mapsto (\tilde{z}, \tilde{\eps})= (\varphi_\eps(z), h(\eps))$. Then $h\equiv {\rm id}$,
$\omega_\eps(z)= \widetilde{\omega}_\eps(z)$ (i.e. $A(\eps)=\tilde A(\eps)$ and $P_{\eps} = \tilde P_{\eps}$) and there exists\footnote{$\C\{\eps\}$ denotes the set of convergent power series in $\eps$.} $T\in \C\{\eps\}$ such that 
$
\varphi_\eps= \Phi_\eps^{T(\eps)}$
is the flow of $\omega_\eps$ at time $T(\eps)$.
Moreover the identity is the only $0$-fixing unitary change of variable and parameter between the two families of vector fields.
\end{theorem}

\begin{proof} We use a method of infinite descent as in the proof of Theorem~3.5 in \cite{RT}. 
Since the proof is completely similar, we will be brief on the details. Before starting the infinite descent, we must reduce the problem.

\medskip \noindent{\bf Reduction of the problem.}
It is easily checked that the flow $\Phi_0^t$ of $\dot z = z^{k+1}/(1+A(0)z^k)$ at time $t$ has the form
\begin{equation}\Phi_0^t(z) = z(1+ g_t(z^k))= z+ tz^{k+1} +tO(z^{2k+1}).\label{flow_0}\end{equation}
Let $\Phi_\eps^t$ be the flow of the first equation at time $t$, let $$\psi_\eps(t,z)= \varphi_\eps(z)\circ\Phi_\eps^t,$$ and let
$$K(\eps,t)=\frac{\partial^{k+1}}{\partial z^{k+1}}\psi_\eps(t,0).$$
We want to solve $K(\eps,t)=0$ by the implicit function theorem. We know that there exists $t_0$ such that  $K(0,t_0)=0$ because of the form of $\Phi_0^t$ in \eqref{flow_0}. Indeed, $K(0,t)= \varphi_0^{(k+1)}(0)+t(k+1)!$, yielding
$\frac{\partial K}{\partial t}(0,0)= (k+1)!\neq0$. Hence there exists a unique analytic germ $T(\eps)$ such that $K(\eps,T(\eps))\equiv0$ and $T(0)=t_0$. 
We change $z\mapsto \Phi^{T(\eps)}z$ in the first system. 

\medskip \noindent{\bf The infinite descent.} After the two reductions, we can suppose that  $\varphi_0 = id$
 and that $\varphi_\eps^{(k+1)} (0) \equiv 0$. 
We now show that $\varphi_\eps= {\rm Id}$ and $h(\eps)=\eps$. Note that $A\equiv \widetilde{A}\circ h$  since the sum of the residues at the singular points is invariant. Let 
$$\begin{cases} P_\eps(z)= z^{k+1} + b_{k-1}(\eps) z^{k-1} + \dots + b_1(\eps)z+b_0(\eps),\\
\widetilde{P}_{\tilde{\eps}}(\tilde{z}) = \tilde{z}^{k+1}+ c_{k-1}(\eps) \tilde{z}^{k-1} + \dots + c_1(\eps)\tilde{z}+c_0(\eps),\\
\varphi_\eps = z + \sum_{j\geq0} f_j(\eps)z^j,\end{cases}$$ where all $b_j, c_j, f_j\in \C\{\eps\}$, and we simply write $b_j$ instead of $b_j(\eps)$, etc. for the functions $b_j, c_j, f_j$ and $h$. 
We introduce the principal ideal $I= \langle \eps\rangle$ in $\C\{\eps\}$. We show by induction that $b_j-c_j, h-\eps, f_j\in I^n$ for all $n\in \N^*$, from which it will follow that they are identically zero. 
Note that $h-\eps=b_0-c_0$.

The conjugacy condition is
\begin{align} \begin{split}
&(1+Az^k)\left(\left(z+\sum_{j\geq0} f_jz^j\right)^{k+1} + \dots +c_1\left(z+\sum_{j\geq0} f_jz^j\right)+ c_0\right) - \\
&\qquad \left(1+A\left(z+\sum_{j\geq0} f_jz^j\right)^k\right)\left(z^{k+1} +  \dots + b_1z+b_0\right)\left(1+\sum_{j\geq1} jf_jz^{j-1}\right)=0.\label{conjugacy_eq}\end{split}\end{align}
which we simply write as $\sum_{j\geq0}g_jz^j=0$. Hence all $g_j$ must be identically $0$. 
The $g_j$ are quite complicated but they have a very simple structure of linear terms and this is what we will exploit. 
\begin{itemize}
\item From the two reductions, it is clear that $h, b_j, c_j, f_j\in I$. This is our starting point.  
\item The only linear terms in the equations $g_j=0$ for $j=0, \dots, k-1$, are $b_j-c_j$. Hence $b_j-c_j\in I^2$. 
\item The equations $g_{k+j}=0$ with $0\leq j\leq k$ yield $f_j\in I^2$, since the only linear terms are
$A(c_j-b_j) + (k+1-j)f_j=0$ when $j<k$ and $Af_0+f_k$ when $j=k$. 
\item Remember that $f_{k+1} \equiv 0$ because of the reduction. 
\item The equations $g_\ell=0$ with $\ell>2k+1$ yield $f_{\ell-k}\in I^2$, since the only linear terms in $g_\ell$ are $-(\ell-2k-1)(f_{\ell-k}+Af_{\ell-2k})$.
\item Hence, all $b_j-c_j, f_j\in I^2$. 
\item We now suppose that $b_j-c_j, f_j\in I^n$, and let us to show that there are in $I^{n+1}$. 
\item The equations $g_j=0$ for $j=0, \dots k-1$, yield $b_j-c_j\in I^{n+1}$. 
\item The equations $g_{k+j}=0$ with $0\leq j\leq k$ yield $f_j\in I^{n+1}$. 
\item The equations $g_\ell=0$ with $\ell>2k+1$ yield $f_{\ell-k}\in I^{n+1}$.
\item Hence, all $b_j-c_j, f_j\in I^{n+1}$. 
\end{itemize}
This concludes the first part of the proof. 

\medskip\noindent  For the last part, let $\Phi_\eps^t$ be the flow of $\omega_\eps$ at time $t$.
Then
\begin{equation}\label{eq:khsfcv}
\frac{\partial}{\partial \eps} \frac{\partial}{\partial t} \Phi_\eps^t (z) = \frac{\partial}{\partial t} \frac{\partial}{\partial \eps} \Phi_\eps^t (z).
\end{equation}
Since $\frac{\partial}{\partial t} \Phi_\eps^t (0) = \omega_\eps(\Phi_\eps^t (0)) = -\eps + \cal O(\Phi_\eps^t(0)^{k+1},\Phi_\eps^t(0) \eps, \eps^2) $, and $\Phi_0^t(0) = 0$, so $\Phi_\eps^t(0) = \cal O(\eps)$,
the left hand side of \cref{eq:khsfcv} is equal to $-1+O(\eps)$ when $z=0$. Hence, integrating the right hand side at $z=0$ yields $\left.\frac{\partial}{\partial \eps} \Phi_\eps^t (0)\right|_{\eps=0}=-t$. So if the change of variable $\Phi^{T(\eps)}$ is $0$-fixing, then necessarily $T(\eps) \equiv 0$.
\end{proof}

\subsection{Principal parts and non-unipotent changes of variable and parameter}\label{sub:ppal}

Consider a family $\omega_\eps$ that is generic as per \Cref{def:gen_vf}. We recall that this condition can be written as
\begin{equation}\label{eq:ABdef}
\omega_\eps(z) = A z^{k+1} - B\eps + O(z^{k+2},\eps z,\eps^2)
\end{equation}
with $A,B\in\C^*$, which we call the \emph{principal part}.

We recall that the attracting axes are the solutions $z$ to $Az^{k+1}/z \in \R_+$.
We call \emph{explosion axes} the $k+1$ asymptotic directions in which the singularities are located when $\eps\to0$ along the positive real axis. They are the solutions $z$ of $Az^{k+1}/B \in \R_+$.
There are $k$ attracting axes and $k+1$ explosion axes.

\subsubsection{Effect of a change of variable on the principal parts and on the axes.}\label{ss:chvarpp}

Denote by
\[(\tilde z,\tilde \eps) = \Psi(z,\eps) = (\psi_{\eps}(z),h(\eps))\]
a change of variable.

The principal part in the new coordinate is of the form $\tilde A \tilde z^{k+1} -\tilde B\tilde\eps$ and
the coefficients $\tilde A$ and $\tilde B$ depend only on the linear part of the change of variable and parameter, and more precisely only on the diagonal part: let
\begin{equation}
D_{(0,0)} \Psi = \left[\begin{array}{cc} \alpha & \bullet \\ 0 & \beta \end{array}\right]
\end{equation}
denote the differential of $\Psi$ at $(z,\eps) = (0,0)$, where the dot stands for any complex number. Then using $\omega_\eps(z) = \tilde \omega_{\tilde \eps} (\tilde z) \frac{\partial z}{\partial \tilde z}$ we get
\begin{equation}\label{eq:AB}
\renewcommand{\arraystretch}{1.5}
\begin{array}{rcr}
\tilde A & = & \alpha^{-k} A, \\
\tilde B & = & \alpha \beta^{-1} B.
\end{array}
\end{equation}

The constants $\alpha$ and $\beta$ give us
two degrees of freedom to simplify the principal part: in particular
\begin{lemma}
  There exists precisely $k$ pairs $(\alpha,\beta)$ of non-zero complex numbers such that 
  \[\tilde A=\tilde B=1,\]
  they are of the form $(\alpha,\alpha B)$ with $\alpha$ varying in the set of $k$-th roots of $A$, and there is one for each choice of repelling axis of $\omega_0$ that we bring to $\R_+$ under $z\mapsto \alpha z$.
\end{lemma}

\subsubsection{Parameter-fixing changes of variable}

Note that $A=B=1$ if and only if $\R_+$ is both a repelling axis and an explosion axis.

A parameter fixing conjugacy is forced to have $\beta=1$.
It will rotate every repelling axis and every explosion axis by $\arg \alpha$.
Note that $\tilde A {\tilde B}^k = A B^k$, so $AB^k$ is an invariant quantity (it is also valid when the parameter is not fixed, provided $\beta=1$).
The condition $AB^k \in \R_+$ corresponds to the fact that a repelling axis and an explosion axis coincide. See \Cref{fig:axes45}. Note that only one repelling axis and one explosion axis can coincide, because their directions are evenly distributed on the circle, and there are respectively $k$ and $k+1$ axes of each type.

For parameter fixing changes of variables, or more generally under the condition $\beta=1$, there remains only one degree of freedom to simplify the principal part. We see two reasonable normalizations in this case:
\begin{lemma}
There are exactly $k$ values of $\alpha$ such that with $(\alpha,\beta) =(\alpha,1)$ we have $\tilde A=1$,
one for each choice of repelling axis that we want to bring to $\R_+$ under the map $z\mapsto \alpha z$.

There are exactly $k+1$ values of $\alpha$ such that with $(\alpha,\beta) =(\alpha,1)$ we have $\tilde A=\tilde B$, i.e.\ such that the principal part is $\tilde A\times(z^{k+1}-\eps)$,
one for each choice of explosion axis that we want to bring to $\R_+$ under the map $z\mapsto \alpha z$.
\end{lemma}
See \Cref{fig:axes3}.

\begin{figure}[htbp]
\begin{tikzpicture}
\node at (0,0) {\includegraphics[width=12cm]{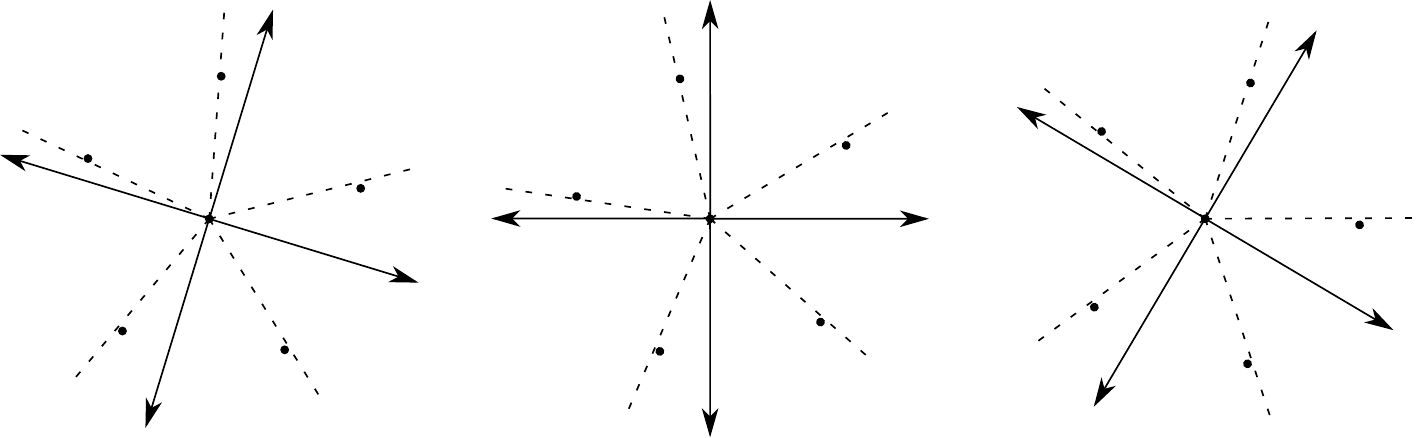}};
\node at (-4.2,-2.2) {initial};
\node at (0,-2.2) {$\arg \tilde A = 0$};
\node at (4.2,-2.2) {$\arg \tilde A = \arg \tilde B$};
\end{tikzpicture}
\caption{Two possible normalizations if we fix the parameter. Left: Initial situation, repelling axes with the arrows, explosion axes as dashed lines, examples of position for $\eps>0$ of the singularities as black dots. They are close, but not necessarily on, the dashed lines. Middle: one of the $k$ possible choices of rotation in the $z$-plane bringing a repelling axis to $\R_+$. Right: one of the $k+1$ possible choices of rotation in the $z$-plane bringing an explosion axis to $\R_+$.}\label{fig:axes3}
\end{figure}

\begin{figure}[htbp]
\subfigure[]{\begin{tikzpicture}
\node at (0,0) {\includegraphics[width=4cm]{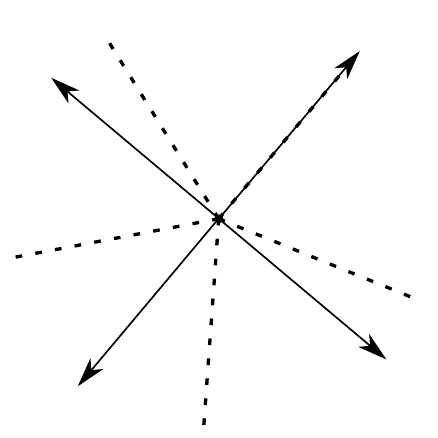}};
\end{tikzpicture}}\subfigure[]{\begin{tikzpicture}
\node at (0,0) {\includegraphics[width=4cm]{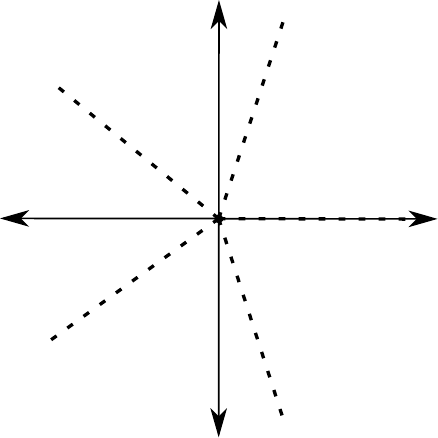}};
\end{tikzpicture}}
\caption{(a): $\arg A B^k = 0$ if and only if a repelling axis and an explosion axis coincide.
(b): $\arg A= \arg B= 0$ if and only if the real axis is both a repelling axis \emph{and} an explosion axis.}
\label{fig:axes45}
\end{figure}

The second normalization ($A=B$) seems more interesting to us because it still allows to define a natural eigenvalue function, $\lambda(\delta)=(k+1)\delta^k \sigma(\delta)$ (with the difference that $\sigma(0)=A$).

\subsubsection{Non-unipotent changes of variables preserving unitary generic vector fields}

If $\omega_\eps$ is unitary, i.e.\ if $\omega_\eps = z^{k+1}-\eps +\cal O(z^{k+2},z\eps,\eps^2)$, then $A=B=1$. Under a general change of variable and parameter $\Psi$ fixing $0$ as in \Cref{ss:chvarpp}, the new family is unitary (i.e.\ $\tilde A=\tilde B=1$) if and only if $\alpha^k=1$ and $\beta=\alpha$, i.e.\ if and only if
\begin{equation}
D_{(0,0)} \Psi = \left[\begin{array}{cc} \zeta^n & \bullet \\ 0 & \zeta^n \end{array}\right]
\end{equation}
where $\zeta = e^{2\pi i/k}$ and $n\in\Z$.
The simplest way to achieve this is of course to take
\begin{equation}\label{eq:simple}
(\tilde z,\tilde \eps) = (\zeta^n z, \zeta^n \eps)
\end{equation}
and $\Psi$ above is just the composition (in any order) of a unipotent change and of \cref{eq:simple}.
At the level of the arguments of $\alpha$, $\beta$, $A$ and $B$, it corresponds for $z$ to bringing another attracting axis to $\R_+$, but then $\R^+$ is not anymore an explosion axis, which is why we also need to adjust $\eps$.

Let us detail the effect of the simple change of \cref{eq:simple} on a few key notions:
for the natural eigenvalue functions we get $\tilde\lambda(\tilde \delta) = \lambda(\delta/\zeta^n)$;
in particular, if $\eps$ is a canonical parameter for the family $\omega_\eps$ then $\tilde \eps = \zeta^n \eps$ is also a canonical parameter for the family $\tilde \omega_{\tilde \eps}$.

We get from \Cref{prop:class}:
\begin{proposition}\label{prop:class:AB}
Consider two families of vector fields $\omega$ and $\tilde \omega$ that are unitary and generic in the sense of \Cref{def:unitary} and their respective natural eigenvalue function $\lambda$ and $\tilde \lambda$. The families are conjugate by a change of variable and parameter if and only if there exists a holomorphic map $\xi$ that commutes with the rotation by $2\pi/(k+1)$ and such that the following holds near $0$:
\begin{equation}
\lambda = \tilde \lambda \circ \xi.
\end{equation}
\end{proposition}

And from \Cref{sub:canon} (where $\cal G'_k$ is defined):
\begin{corollary}
The modulus space of germs of generic parabolic unfoldings of codimension $k$ parabolic vector fields, up to conjugacy by a change of variable and parameter, is naturally in bijection with
$\cal G'_k/\U_k$ where the action of $\rho\in\U_k$ is \emph{pre-composition} with multiplication by $\rho^{-1}$.
\end{corollary}

For parameter-fixing changes, we are interested in maintaining the normalization $A=B$, which still allows for the definition of the natural eigenvalue function.
Since $\beta=1$, we have $\tilde A/\tilde B = \alpha^{-(k+1)} A/  B$, so if $A=B$ then $\tilde A= \tilde B$ if and only if $\alpha^{k+1}=1$, i.e.\ the differential of the change is
\begin{equation}
D_{(0,0)} \Psi = \left[\begin{array}{cc} \xi^n & \bullet \\ 0 & 1 \end{array}\right]
\end{equation}
with $\xi = e^{2\pi i/(k+1)}$ (note the difference with $\zeta$ above) and $n\in\Z$.
The simplest such change is $\tilde z = \xi^n z$.
The effect of this change on the natural eigenvalue function is $\tilde \lambda(\delta) = \lambda(\delta/\xi^n)$.
This allows the following short statement of the analogue of \Cref{thm:lambda_inv} for general changes of variables fixing the parameter on (non necessarily unitary) generic families:
\begin{theorem}[Classification without change of parameters]\label{thm:lambda_inv:AB}
Two generic families $\omega_\epsilon$ and $\widetilde \omega_{\epsilon}$ are conjugate near $(0,0)$ by a change of variable $z\mapsto \tilde z = \tilde z (z,\epsilon)$ fixing the parameter if and only if for all $\eps$ near $0$, they have the same sets of eigenvalues.
\end{theorem}

From \Cref{thm:cl_e}:
\begin{corollary}\label{thm:cl_e:AB}
For any generic family $\omega_\epsilon$, let $\lambda(\delta) = (k+1)\delta^k \sigma(\delta)$ be an associated eigenvalue function. Then there exists a conjugacy preserving the parameter between $\omega_\epsilon$ and the vector field
\[\dot z = \Omega_\epsilon(z) = (z^{k+1}-\epsilon)\sigma(z)
.\]
\end{corollary}

Let $\hat {\cal G}_k$ denote the set of germs of functions $\lambda$ with a root of order exactly $k$ at $0$.
Let $\U_{d}$ denote the multiplicative group of $d$-th roots of unity.
From \Cref{thm:mod_fix} we get:
\begin{corollary}\label{thm:mod_fix:AB}
The modulus space of germs of generic 1-parameter unfoldings of codimension $k$ parabolic vector fields, up to conjugacy fixing the parameter, is naturally in bijection with $\hat{\cal G}_k/\U_{k+1}$ where\footnote{The fact that $\rho$ is of order $k+1$ and not $k$ is not a mistake.} $\rho \in \U_{k+1}$ acts by \emph{pre-composition} by multiplication by $\rho^{-1}$.
\end{corollary}

\subsubsection{Auxiliary question}

We may wonder when a generic family $\omega_\eps$ is conjugate near $(0,0)$ to the simple family $\tilde\omega_\eps(z) = z^{k+1}-\eps$ studied in \Cref{sec:ideal:global}. The natural eigenvalue for this family is $\tilde \lambda(\delta)=(k+1)\delta^k$. \Cref{prop:class} implies:

\begin{corollary}\label{cor:cl_simp}
Let $\lambda$ be an eigenvalue function for $\omega_\eps$. Then $\omega_\eps$ is conjugate to the family $z^{k+1}-\eps$ by a change of variable and parameter near $(0,0)$ if and only if $\lambda(\delta)= \delta^k\sigma(\delta^{k+1})$ for some germ of analytic function $\sigma$ with $\sigma(0)\neq 0$, i.e.\ if and only if $\lambda(\zeta \delta) = \zeta^{-1}\lambda(\delta)$ holds near $0$, with $\zeta = e^{2\pi i/(k+1)}$.
\end{corollary}

\begin{proof}
 By \Cref{prop:class}, such a conjugacy exists if and only if there exists a germ $\xi$ 
 of the form\footnote{The holomorphic germs $\xi$ that commute with $\delta\mapsto \zeta \delta$ are the germs of this form.} $\xi(\delta)=\delta g(\delta^{k+1})$ with $g(0)\neq 0$, such that $\lambda(\delta) = \tilde \lambda\circ\xi(\delta) = (k+1)\xi(\delta)^k$, i.e.\ $\lambda(\delta) = (k+1)\delta^k g(\delta^{k+1})^k$. 
 For every germ $\sigma$ with $\sigma(0)\neq 0$ there is a germ $g$ such that $(k+1)g(\eps)^k = \sigma(\eps)$ holds near $0$.
\end{proof}

\section{A geometric model}\label{sec:geom}

\subsection{Introduction}\label{sub:intro_geo}

In this section we construct, for $1$-parameter vector fields $\omega_\eps$ satisfying the unitary genericity conditions of \Cref{def:unitary} and recalled below, a domain whose image in rectifying coordinate is an analogue of the  star shaped region found in \Cref{sec:ideal:global}, see for instance \Cref{fig:f6}. For this we will use the fact that the situation is close, up to a change of variable and parameter, to the vector field $\dot z = z^{k+1}-\epsilon$.
Since we are doing a local analysis, i.e.\ $z\in B(0,r)$, we have to remove a neighborhood (varying with $\eps$ but of bounded size) of the vertices of the shape.
There are complications, for instance the sum of the periods does not anymore need to be $0$, which means we cannot anchor the half-lines to a $k+1$-gon (before removing neighborhood of vertices).

Let us recall the unitary genericity condition:
\[
  \omega_\eps(z) = z^{k+1}-\eps + \cal O(z^{k+2},z\eps,\eps^2).
\]
We assume that $\omega_\eps(z)$ is defined for $(z,\eps)\in B(0,r)\times B(0,\rho)$.
Since we are only dealing with a local result concerning the family $\omega_\eps$ we may decrease in the sequel the values of $r$ and $\rho$ when needed. Of course this will be done at most finitely many times.
Note that in this section we will not perform changes of parameter. 

\subsection{Rectifying coordinate and translation surfaces}\label{sub:translation_geo}

We recall that the associated rectifying coordinate, which is called straightening coordinate, or $t$-coordinate, or complex time, is given by
\[ t=\int\frac{dz}{\omega_\eps(z)}
,\]
and that, in this coordinate, the vector field $\dot z = \omega_\eps(z)$ becomes $\dot t =1$.
This cooordinate is only defined on the complement of the singularities, is well-defined locally but usually multivalued if considered over a non simply connected open subset of the complement of the singularities.
  
Recall that a Riemann surface is a topological surface\footnote{We require topological surfaces to be Hausdorff separated.} with an atlas for which the transition maps are holomorphic.
\begin{definition} A \emph{translation surface} is a topological surface with an atlas for which the transition maps between charts are translations.
\end{definition} 
\noindent A translation surface is automatically a Riemann surface.
A translation surface naturally carries a flat Riemannian metric of the form $ds=|dt|$ in the charts. We will use the associated geodesic distance.

\medskip

The patches of rectifying coordinates of a holomorphic vector field $\omega$ on a Riemann surface $S$ define an atlas of a translation surface, on $S$ minus the singularities of $\omega$. For instance, in the case of the vector field $\dot z = z^{k+1}-\eps$ taken on $S=\C$, the translation surface is isomorphic to the star shaped domain of \Cref{sec:ideal:global}, where each strip has its two sides glued together by a translation (of vector $\mu_j$ for the corresponding $j$).

\medskip

An \emph{isomorphism} between two translation surfaces is a homeomorphism that, when expressed in the charts of the atlas, is locally a translation. It is in particular holomorphic.

\medskip

We claim the following: 
\begin{lemma} Two vector fields $\omega$ on $S$ and $\omega'$ on $S'$ are isomorphic if and only if the associated translation surfaces are isomorphic and the isomorphism has a continuous extension to a homeomorphism between $S$ and $S'$.
\end{lemma}
\begin{proof}
One direction is immediate and the other essentially comes from the erasable singularity theorem applied to the isomorphism expressed in charts of the two Riemann surfaces.
\end{proof}

In the sequel we take $S=B(0,r)$ and call $\cal V_\eps$ the translation surface associated to $\omega_\eps$, assuming that $\omega_\eps(z)$ is defined for $(z,\eps)\in B(0,r)\times B(0,\rho)$.

\subsection{About the periods}

In this article we only consider vector fields that do not vanish identically on any connected component of their domain, i.e.\ their set of singularities is (closed and) discrete.

\begin{definition}\label{def:period}
The \emph{period} $\mu_j$ associated to a singularity $z_j$ of a vector field $\dot z = \omega(z)$ is the integral of $\int\frac{dz}{\omega(z)}$ around a positively oriented small loop enclosing $z_j$, so by the residue theorem, $\mu_j = 2\pi i \operatorname{Res}\left(1/\omega,z_j\right)$. If the singularity is not parabolic, i.e.\ $\omega'(z_j)\neq 0$, then the period is just $\mu_j = 2\pi i / \omega'(z_j)$. Otherwise it can be any complex number (in particular it is finite).
\end{definition}

The canonical eigenvalue function $\lambda$ was defined in \Cref{def:eig}. Its values at the $(k+1)$-th roots of $\eps$ give the eigenvalues of $\omega_\eps$.

\begin{lemma}\label{lemma:periods}
  The sum of the periods is analytic in $\eps$ and given by $2\pi i A(\eps)$, where
\begin{equation}
A(\eps):=\sum_{j=0}^{k+1} \frac1{\lambda(\delta_j)}.\label{def:a}
\end{equation}
The limit of the sum of the periods as $\eps\to 0$ is equal to $2\pi i$ times the residue of $1/\omega_0$ at the origin. 
\end{lemma}
\begin{proof} 
It follows from the fact 
$A(\eps)=\frac1{2\pi i} \oint_C\frac{dz}{\omega_\eps(z)},$
where $C$ is a circle around the $k+1$ singular points.
\end{proof}

Let us note the following bonus:

\begin{lemma}\label{lemma:periods_bonus}
For \emph{any} holomorphic germ $\lambda$ with a root of order exactly $k$ at the origin, the sum of $1/\lambda(\delta)$, where $\delta$ ranges on the solutions of $\delta^{k+1}=\epsilon$, has a limit when $\epsilon\tend 0$, equal to $k+1$ times the constant term in the Laurent series of $1/\lambda$ at the origin.
\end{lemma}
\begin{proof} 
It follows from the previous lemma and the existence of a vector field with periods associated to $\lambda$ (second part of \Cref{prop:eig}).
More simply, it can also be directly checked: $1/\lambda(\delta)$ is a Laurent series in $\delta$ with exponents ranging from $-k$ to $+\infty$. The sum thus cancels out all negative exponents.
\end{proof}

\subsection{Periodic domains}\label{sub:per:dom}

We will use the construction of \cite{KR}, which was inspired from a previous version of the present article: \cite{CR0}.

Consider a holomorphic vector field $\omega$ on a simply connected subset of the plane and a singular point of the vector field, and assume that $z$ is simple, i.e.\ that the eigenvalue $\lambda$ of $\omega$ at $z$ is not $0$. 
We then associate to $z$ a \emph{periodic domain}
\[U(z)\]
which is the maximal rotation domain around $z$ of the vector field $\mu\omega$ where $\mu = 2\pi i/\lambda$ is the \emph{period} of $z$. For convenience, we call $\mu \omega$ the \emph{circular} vector field of the periodic domain and we call $i\mu\omega$ the \emph{radial} vector field of the periodic domain.
A periodic domain is necessarily homeomorphic to a punctured disk.

\begin{lemma}
Two different simple singularities of $\omega$ have disjoint periodic domains.
\end{lemma} 
\begin{proof} Let $z_1$ and $z_2$ two singularities, with respective eigenvalues $\lambda_1$ and $\lambda_2$. Theorem~4 in \cite{Duff} states that no two cycles of $e^{i\alpha}\omega$ and $e^{i\alpha'}\omega$ can intersect if $\alpha\not\equiv \alpha'\bmod \pi$.
So if $\lambda_1/\lambda_2\notin \mathbb R$, then  the periodic domains are disjoint.
If $\lambda_1/\lambda_2\in \mathbb R$ and the periodic domains are not disjoint, then $z_1$ and $z_2$ are simultaneously centers for the same $e^{i\theta}\omega$.
The set $U(z_i)$ is foliated by periodic trajectories separating $z_i$ from infinity. Let us orient such a periodic trajectory, not by the vector field, but anticlockwise. Then the winding number of $e^{i\theta}\omega$ along this oriented closed curve is independent of the trajectory in $U(z_i)$ and is equal to $1$.
Assume by contradiction that there is a common point in $U(z_1)$ and $U(z_2)$.
Then both $z_1$ and $z_2$ belong to the bounded component $U$ delimited by the closed trajectory of this point.
On the other hand,
since $e^{i\theta}\omega$ is holomorphic, all its singular points have positive Poincaré index (which coincides with the vanishing degree of $\omega$ at the singularity).
So $U$ can only contain one singular point, contradicting our assumption.
\end{proof}

We now focus on our generic 1-parameter families $\omega_\eps$, on a disk $B(0,r)$, with $\eps\in B(0,\rho)$.
For any $\eta>0$, by reducing $r$, then $\rho$, we saw that we can assume that $\omega_\eps$ is defined on $\overline{B(0,r)}$ and that $\omega_\eps$ is close to the ideal vector field $\tilde\omega_\eps(z) = z^{k+1}-\eps$ in the following sense:
\begin{enumerate}
\item There are $k+1$ singularities $z_j$, they are simple, and $\left|\frac{z_j}{\tilde z_j}-1\right|<\eta$.
\item Their periods satisfy $\left|\frac{\mu_j}{\tilde \mu_j}-1\right|<\eta$.
\item $\forall z\in\partial B(0,r)$, $\left|\frac{\omega_\eps(z)}{z^{k+1}}-1\right|<\eta$
and
$\left|\frac{\tilde\omega_\eps(z)}{z^{k+1}}-1\right|<\eta$.
\end{enumerate}

\begin{lemma}\label{lem:order}
For all $r$ small enough, there exists $\rho>0$ such that $\forall \eps\in B(0,\rho)$, the order of contact of any trajectory at a point where it is tangent to $\partial B(0,r)$ is $1$.
\end{lemma}
\begin{proof}
The vector field is then close to $z^{k+1}$, for which $\partial B(0,r)$ maps in time coordinate to a curve that, when rescaled by a factor of $r^k$ is $C^n$ close (for any given $n$) to a circle of radius depending only on $k$.
In time coordinate, trajectories are straight lines, so the order of contact is $1$, which is preserved under change of variables.
\end{proof}

Recall that the ideal singularities form the vertices of a regular $k+1$-gon. 
We can label them $\tilde z_j$ with $j\in \Z/(k+1)\Z$ according to a circular order and this gives a circular order on the singularities $z_j$ of $\omega_\eps$.

\begin{lemma}\label{lem:ujvszj}
  For any $\eta>0$ one can choose $r$, then $\rho$, so that for all $\eps\in B(0,\rho)-\{0\}$, the periodic region $U(z_j)$ has a unique contact point $u_j$ with $\partial B(0,r)$, the order of contact is $1$, and
  \[|\arg(u_j/\tilde z_j)| < \eta
  .\]
\end{lemma}
In particular, the contact points $u_j$ have the same circular ordering as the $z_j$.
\begin{proof}
The vector field $\mu_j \omega_\eps$ is tangent to $\partial U(z_j)$ at any point of this boundary, and moreover orients this boundary in the anticlockwise direction.
The contact point $u_j$ must thus satisfy
\[\arg \frac{i u_j}{\omega_\eps(u_j)\mu_j} \equiv 0 \bmod 2\pi\]
(which expresses that the rotated vector field $\mu_j\omega_\eps$ is tangent to $\partial B(0,r)$ at $u_j$ and that the orientations of $\partial B(0,r)$ and $\partial U(z_j)$ coincide).
Given $\mu_j$, this equation on $u_j$ has exactly $k$ solutions.
(This gives $k$ possibilities for each of the the $k+1$ contact points $u_j$.)
Since $\mu_j$ is close to $\tilde \mu_j$ and $\omega_\eps$ to $\tilde \omega_0 (z) = z^{k+1}$ on $\partial B(0,r)$, $u_j$ is close to one of the $k$ solutions $u\in \partial B(0,r)$ of
\[\arg \frac{i u}{u^{k+1}\tilde \mu_j} = 0,\]
which include $r \tilde z_j/|\tilde z_j|$.
The result then follows by \Cref{lem:order} and the lemma below, which prevents in particular $u_j$ from jumping from one solution to the other.
\end{proof}

\newcommand{\symb}{\circ}
\begin{lemma}
  Assume $\dot z =\omega_{\symb}(z)$ is a holomorphic vector field defined in a neighborhood $B(0,r')$ of $\ov B(0,r)$ ($r'>r$).
  Assume that some zero $z_{\symb}$ of $\omega_{\symb}$ in $B(0,r)$ is simple,
  and that its periodic region $U(z_{\symb})$ relative to $B(0,r)$ has a unique contact point $u_{\symb}$ with $\partial B(0,r)$ and that the order of contact is $1$.
  Then for any sufficiently close holomorphic vector field $\omega$ on $B(0,r')$, the periodic region of $\omega$ relative to the zero $z_\omega$ of $\omega$ nearest to $z_{\symb}$ also has a unique contact point $u_\omega$ with $\partial B(0,r)$,  the order of contact is $1$, and $u_\omega \tend u_\symb$ as $\omega\to \omega_{\symb}$.
\end{lemma}
\begin{proof}
  The periodic region $U(z_\symb)$ is compactly contained in the periodic region $U'(z_\symb)$ relative to $B(0,r')$.
  There is a neighborhood $V$ of the closure of $U(z_{\symb})$, independent of $\omega$, such that for any $\omega$ close enough to $\omega_{\symb}$ on $B(0,r')$, the antiderivative $\phi_\omega$ of $\frac{\omega'(z_\omega)}{2\pi i\omega(z)}$ is well-defined, modulo addition of an integer, on $V-\{z_\omega\}$, and conjugates the vector field to the horizontal vector field on an open subset of $\C/\Z$.
  For $V$ chosen small enough, the image of $\partial B(0,r)$ under $\phi_{\omega_{\symb}}$ is a curve with a unique tangency with a horizontal closed line of $\C/\Z$, and the contact is of order $1$.
  This is preserved under perturbation.
\end{proof}

In the rectifying coordinate of $\omega_\eps$, the periodic domains $U(z_j)$ are cylinders, i.e.\ $U(z_j)$ is isomorphic as a translation surface to the quotient of the infinite half-strip of equation $\Im(-t/\mu_j)>0$, $\Re(-t/\mu_j)\in[0,1]$ by the map $z\mapsto \mu_j + z$.

\subsection{The construction}\label{sub:constr}

Consider the simply connected set
\[U = B(0,r)\setminus \bigcup_j \ov U(z_j).\]
Consider in each periodic domain $U(z_j)$ the radial vector field $-\lambda_j^{-1}\omega_\eps$, orthogonal to the rotation vector field $\mu_j \omega_\eps$.
Consider its unique trajectory that starts from the contact point $\{u_j\} = \partial U(z_j) \cap \partial B(0,r)$. It tends to $z_j$ as the time tends to $+\infty$.
For different $j$, the trajectories are necessarily disjoint since the periodic domains are disjoint.
Let us cut $B(0,r)$ along these $k+1$ curves and their limit points $z_j$: we get a simply connected set $D'$ containing $U$.

Consider a determination $z\mapsto t(z)$ of $\int dz/\omega_\eps(z)$ over $D'$.
The next proposition is illustrated by \Cref{fig:dr0,fig:dr}.

\begin{figure}[htbp]
\begin{tikzpicture}
\node at (0,0) {\includegraphics[scale=0.6]{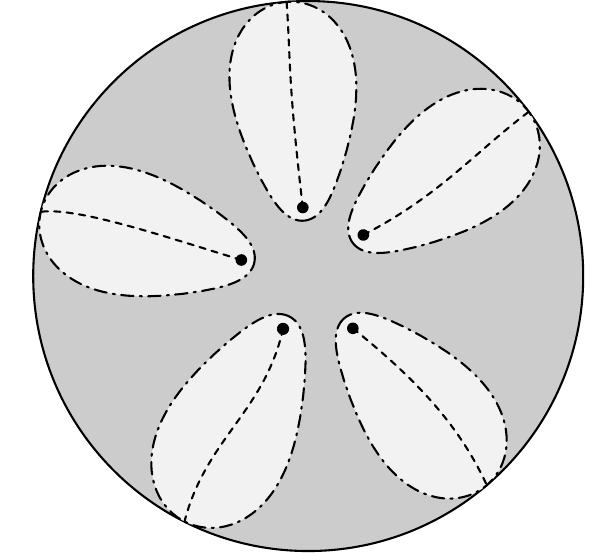}};
\node at (0.28,0.1) {$z_j$};
\node at (2.49,1.82) {$u_j$};
\end{tikzpicture}
\caption{The dark grey set is $U$. The disk minus the slits is $D'$.}
\label{fig:dr0}
\end{figure}

\begin{figure}[htbp]
\begin{tikzpicture}
\node at (0,0) {\includegraphics[scale=0.6]{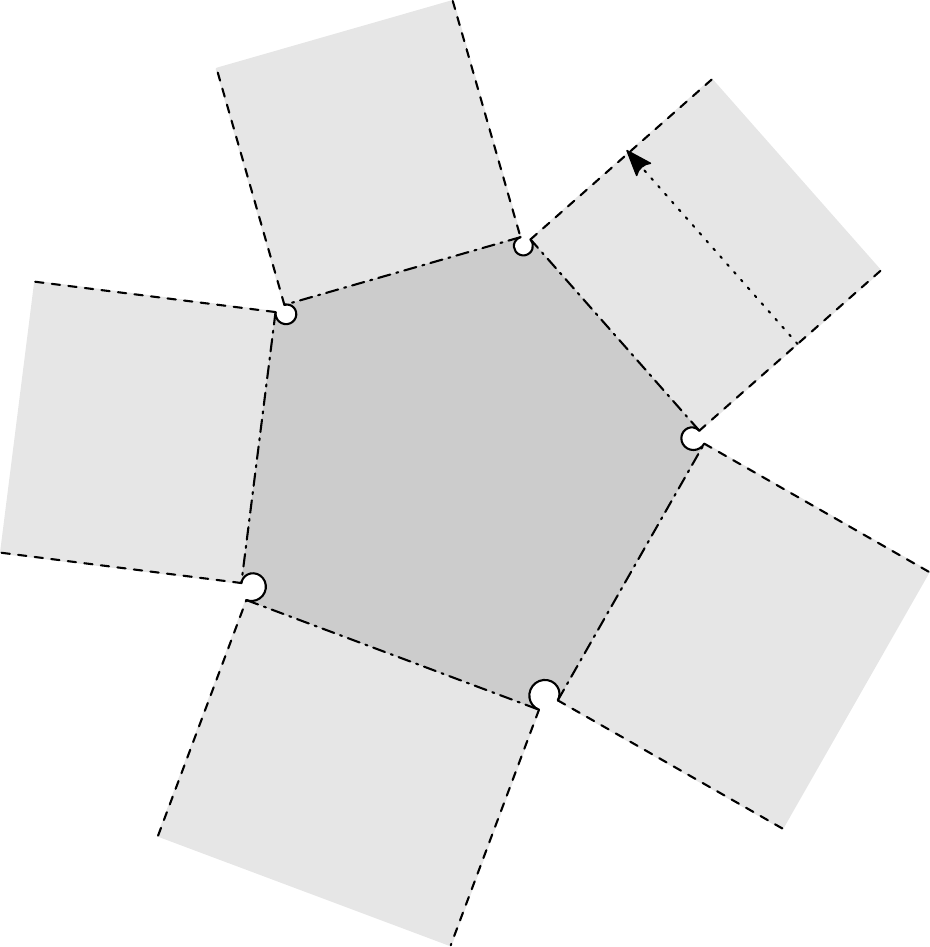}};
\node at (2.9,2.5) {$-\mu_j$};
\node at (0.45,1.6) {\footnotesize $C'_{j+1/2}$};
\draw (0.45,1.8) -- (0.55,2.2);
\end{tikzpicture}
\caption{The image in rectifying coordinate of the set $U$ of \Cref{prop:drsc} is the dark greyed set. The light grey set is the image of $D'-U$. The curves $C'_{j+1/2}$ are called the \emph{eyelets}.}
\label{fig:dr}
\end{figure}

\begin{proposition}\label{prop:drsc}
When $r$ and $\rho$ are small enough, the map $t$ is injective on $U$.
The image of $U$ is a simply connected set bounded by a Jordan curve made of alternating straight segments of vector $-\mu_j$ and nearly circular arcs $C'_{j+1/2}$ of diameter that stays bounded as $\eps\to 0$ and that we call the \emph{eyelets}. See \Cref{fig:dr}.
\end{proposition}
\begin{proof} To prove injectivity, by the winding number property, it is enough to show that the image of $\partial U$ is a simple closed curve.
The image of the boundary of the periodic domain of $z_j$ is a straight line segment tracing the vector $-\mu_j$.
The image of a portion of $\partial B(0,r)$ between two consecutive contact points $u_j$ is a curve that is $C^n$-close (for any given $n$ when $\rho$ and $r$ are small) to a circular arc of circle of diameter that depends explicitly on $r$, and of angle close to $2\pi-2\pi/(k+1)$.
The segments are tangent to the ends of these curves. When $r$ and $\rho$ are small enough, none of these parts can intersect.
\end{proof}

The image in rectifying coordinate of a slit periodic domain is a semi-infinite strip whose base is one of the straight segments bounding $t(U)$ and which is orthogonal to it.

\begin{proposition}\label{prop:drsc2}
When $r$ and $\rho$ are small enough, the map $t$ is injective on $D'$, in particular the image of $D'$ is a simply connected set.
\end{proposition}
\begin{proof}
The proof of is similar to \Cref{prop:drsc}: let us double the slit lines so that $\partial D'$ can be considered as a simple closed curve on a space projecting to $\ov B(0,r)$.
Then the image of $\partial D'$ in $t$-coordinate consists in $k+1$ disjoint curves, each of which tends to $\infty$ at both ends (see \Cref{fig:dr}).
Since $D'$ is connected, these curves necessarily bound a common domain, which is the image of $D'$ and for which $t$ is a bijection from $D'$ to it, by the winding number property.
\end{proof}

\begin{remark}
This model allows to give a geometric proof that the only invariant under conjugacy that preserves the parameter is the collection of eigenvalues: one first realizes that the underlying translation surfaces are isomorphic (if we shrink the domain $B(0,r)$ appropriately, and $\rho$ accordingly). Then one has to justify that, among the possible isomorphisms, a choice can be made that depends analytically with $\eps$.
This is detailed in a previous version of this article: see \cite{CR0}, Section~3.6.
Of course Teyssier's method used in the proof of \Cref{thm:lambda_inv} convincingly covers more general situations than this geometric method, or at least it is not obvious how to generalize the latter.
\end{remark}

\FloatBarrier

\section{Bifurcation diagram}\label{sec:bif2}

In full generality, the purpose of a bifurcation diagram of a family of vector fields is to describe the changes in the topological orbital behavior of the flow, as the parameter changes.
By topological orbital behavior we mean the equivalence class up to topological conjugacy in the $z$ variable, and to reparametrization of the \emph{real} time parameter. 
Quite often in $2$-dimensional real flows, the bifurcation diagram consists in finitely many open regions where the flow is \emph{structurally stable} (there is a homeomorphism varying continuously\footnote{an isotopy} with $\eps$ conjugating, up to real time reparametrization, flows for nearby parameters). When the parameter space has real dimension 2 (which is the case here) the regions of structural stability are usually separated by curves, called \emph{bifurcation curves}, that may intersect at higher codimension bifurcations.
This will be the case here.

In this section we provide a bifurcation diagram for a generic unfolding of a parabolic point 
over a sufficiently small disk $B(0,r)$ for all values of $\eps$ in a sufficiently small disk $B(0,\rho)$. In most statements, we will choose $r$ first, then $\rho$. The homeomorphisms in the definition of structural stability will have to consist in homeomorphisms of the whole set $B(0,r)$ to itself.

\begin{remark}
It is customary in the context of a local study of vector fields to first reduce to a normal form, for instance to $\omega_\eps(z) = (z^{k+1}-\eps)\sigma(z)$, which makes some descriptions easier.
However, this deforms the disk $B(0,r)$ to a more general domain, that moreover varies with $\eps$.
Since here we limit ourselves to the study to restrictions of vector fields to disks, it will be more general to study vector fields that are not in normal form.
Ideally we would allow for domains that are close, but not equal, to disks and depend on $\eps$, but we do not expect essentially different phenomena to appear.
The proximity condition with disks would include that the boundary of these domains is composed by exactly $2k$ arcs, namely $k$ entry arcs, $k$ exit arcs, separated by exactly $2k$ points where the vector field is tangent to the boundary.
\end{remark}

Recall that the notion of being \emph{unitary and generic} for a family of vector fields $\dot z = \omega_\eps(z)$ with $\omega_0(z) = z^{k+1} + O(z^{k+2})$, $k\geq 1$, was given in \Cref{def:gen_vf}: it means that
\[\omega_\eps(z) = z^{k+1} - \eps + O(z^{k+2},\eps z, \eps^2)\]
and that we call \emph{ideal vector field} the following:
\[\tilde \omega_\eps(z) = z^{k+1} - \eps.\]
We restrict the vector field to a ball $B(0,r)$ compactly contained in its domain of definition.
This allows to use analytical tools at the boundary, and to extend the trajectories at and beyond the boundary if needed.

\begin{definition}\label{def:bif:etc}
Consider a unitary generic family as above.
\begin{itemize}
\item We say that a trajectory of the restriction of the vector field $\omega_\eps$ to $B(0,r)$ has a \emph{double\footnote{This cannot be mistaken with a tangency point of higher contact order, since, in our situation, trajectory always have a contact of minimal order, i.e.\ order $1$, with $\partial B(0,r)$ by \Cref{lem:order}.} tangency} (to $\partial B(0,r)$) if it tends at its two extremities to points of $\partial B(0,r)$ where the vector field is tangent to $\partial B(0,r)$.
These points may coincide, in which case the trajectory extends to a periodic trajectory.
\item We denote by $\cal {DT}$ the set of parameters $\eps\neq 0$ for which there is a double tangent trajectory, and call it the \emph{double tangency locus}.
\item A trajectory from $\partial B(0,r)$ to itself is a \emph{boundary connection}.
\item We denote by $\cal C$ the set of parameters $\eps\neq 0$ for which there is a boundary connection and call it the \emph{connection locus}.
\item The set of parameters for which structural stability does not hold is called the \emph{bifurcation locus} and denoted by $\cal B$.
\end{itemize}
\end{definition}

\begin{remark}
The notion of boundary connection is one possible replacement for the notion of a homoclinic loop of a polynomial vector field on $\CP^1$.
The notion of outgoing separatrix of a polynomial vector field is played by the trajectories that exit $B(0,r)$, while incoming separatrices are replaced by trajectories entering $B(0,r)$.
In each case, what were isolated curves on $\CP^1$ are now bunches of curves.
See \Cref{fig:d12}.
\end{remark}

\subsection{Statements}\label{sub:bif:statements}

Before describing the bifurcation diagram of the family $\omega_\eps$ on a disk, it is interesting to recall the bifurcation diagram of the ideal vector field on the whole set $\CP^1$.
It consisted in bifurcation curves of parameters for which $\tilde \omega_\eps$ has
homoclinic loops through infinity.
In terms of the Douady-Estrada-Sentenac invariant, as the parameter crosses these curves, the zig-zag of trajectories and its attachment to the separatrices changed topological type: $2k$ different structurally stable topological types were occuring. 
The general spirit for the bifurcation diagram of $\omega_\eps$ will be the same (see \Cref{local_bif_diag}):
\begin{itemize} 
\item there are still $2k$ large sectorial regions of $B(0,\rho)$ where the different topological types of zig-zag of trajectories joining the singular points are present in the disk in the same order as for the ideal vector field on $\CP^1$; 
\item separated, not by curves, but by $2k$ cusp-like thin regions bounded by analytic arcs, where the zig-zag is broken and where there exist boundary connections separating the singularities;
\item all bifurcations are part of the double tangency locus. This includes the boundaries of each thin region and other bifurcation curves inside the thin regions.
\end{itemize}
We will see that the bifurcation locus and the double tangency locus coincide. It is
is composed of $4k^2$ (not necessarily distinct) real analytic arcs starting at the origin.

\begin{remark}
Usually, when determining local bifurcation diagrams, one first studies the type of the singular points (termed \emph{local bifurcations}), and then the more global bifurcations. 
Here, it suffices to study the global bifurcations. Indeed the center domain inside a disk ends up in a periodic trajectory that is tangent to the boundary and that can be seen as a double tangent trajectory to the boundary: we start at the tangency point, we follow the trajectory, and we come back to a tangency point.
In the case of $\dot z= z^{k+1}-\eps$ over $\CP^1$, the local bifurcation of a singular point when becoming a center occurs precisely when there is a homoclinic loop surrounding the center.
\end{remark} 

The whole bifurcation diagram is governed by the behavior of $\omega_\eps$ along $\partial B(0,r)$, which we describe now, and which is illustrated in \Cref{fig:bord}~(b). 

For this, we use the arguments of the attracting, repelling and intermediate axes of $\omega_0$, whose definition we recall: these are the set of $z\in\C^*$ such that, $z^{k+1}/z$ is respectively real positive, real negative, and imaginary, see \Cref{fig:bord}~(a). Their arguments $\phi$ are of the form $\frac{n\pi}{2k}$, where the class of $n$ modulo $4$ tells the type: $0$ for repelling, $2$ for attracting, $1$ or $3$ for intermediate.
We let
\begin{equation}\label{eq:def:phi:m}
 D = \setof{\phi_m = \frac{m + 1/2}{k}\pi}{0\leq m <2k}\subset \R/2\pi\Z
\end{equation}
be the set of argument $\phi_m$ of the intermediate axes.
In the sequel the notation $\phi_a$ will be used  for the argument of an attracting axis and $\phi_r$ for a repelling axis.

\begin{figure}[htbp]
\begin{center}
\subfigure[Attracting axes in blue, repelling axes in red, intermediate axes as dotted half-lines, for the singular point $z=0$ of $\omega_0$.]{
\includegraphics[width=4.5cm]{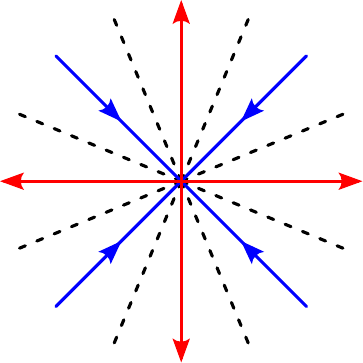}
}
\\
\subfigure[Behavior of $\omega_\eps$ along $\partial B(0,r)$ in the case $k=4$: there are $2k$ points of tangency $T_\phi$ (see \Cref{lem:tgcy:pts}) where $\phi$ is the argument of an intermediate axis, separated by $2k$ regions where the vector field is alternately entering (blue) or exiting (red) $B(0,r)$. We denote these regions by $A_\phi$ or $R_\phi$, where $\phi$ is the argument of the attracting/repelling axis that they contain. All these objects depend on $\eps$, which we omitted on this figure for clarity.]{
\begin{tikzpicture}
\node at (0,0) {\includegraphics[width=5cm]{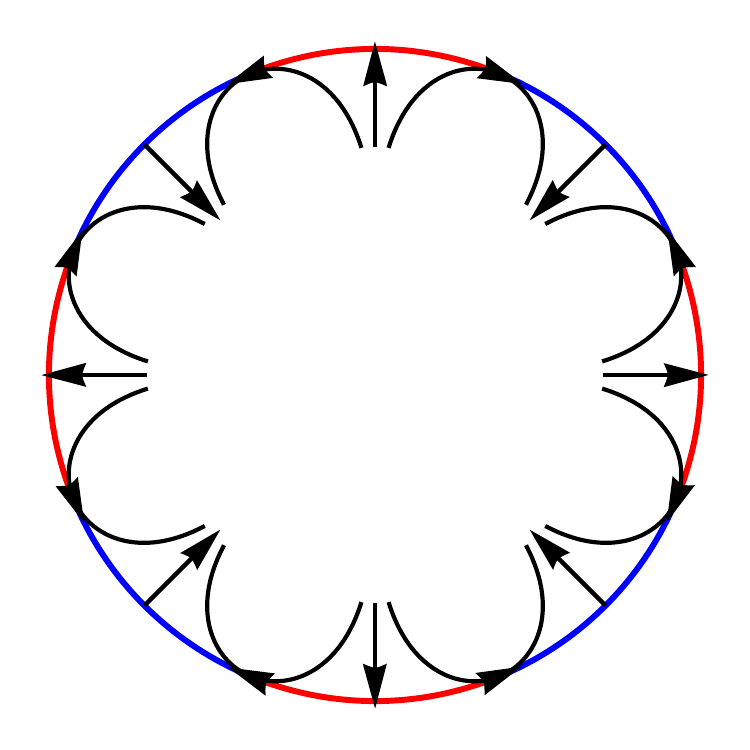}};
\node at (2.55,1.0) {$T_{\pi/8}$};
\node at (1.0,2.35) {$T_{3\pi/8}$};
\node at (-1.0,2.35) {$T_{5\pi/8}$};
\node at (-2.6,1.0) {$T_{7\pi/8}$};
\node at (-2.6,-1.0) {$T_{9\pi/8}$};
\node at (-1.0,-2.4) {$T_{11\pi/8}$};
\node at (1.1,-2.4) {$T_{13\pi/8}$};
\node at (2.7,-1.0) {$T_{15\pi/8}$};

\node at (6,0) {\includegraphics[width=5cm]{bord-v2.pdf}};
\node at (8.6,0) {$R_0$};
\node at (8,1.9) {$A_{\pi/4}$};
\node at (6,2.6) {$R_{\pi/2}$};
\node at (4,1.9) {$A_{3\pi/4}$};
\node at (3.2,0) {$R_{2\pi/2}$};
\node at (4,-1.9) {$A_{5\pi/4}$};
\node at (6,-2.6) {$R_{3\pi/2}$};
\node at (8.1,-1.9) {$A_{7\pi/4}$};
\end{tikzpicture}
}
\end{center}
\caption{Case $k=4$}
\label{fig:bord}
\end{figure}

\begin{lemma}\label{lem:tgcy:pts}
  For all $\eta>0$,
  provided $r$ and $\rho$ are small enough, then for all $\eps \in B(0,\rho)$, there are exactly $2k$ tangency points between the vector field and $\partial B(0,r)$ (i.e.\ points where the vector field and $\partial B(0,r)$ are tangent)
  and their arguments are $\eta$-close to the $2k$ arguments $\phi$ of the intermediate axes.
  They depend real-analytically on $\eps$ over $B(0,\rho)$ and we denote them by
  \[ T_{\phi}(\eps) \in\partial B(0,r) \text{ with }
     \left| \arg(T_\phi(\eps)) -\phi \right|<\eta,\quad \phi\in D
    .
  \]
\end{lemma}

The proof is elementary, given in \Cref{sub:bif:pf}.
The tangency points $T_{\phi}(\eps)$ come in two flavors: those for which the vector field is at this point in the same direction as the anticlockwise orientation of $\partial B(0,r)$ and which we call of \emph{top type}, and those for which it is in the clockwise direction, and which we call \emph{bottom type}.\footnote{
The top/bottom nomenclature comes from the following:
in complex time coordinate, $\partial B(0,r)$ is locally mapped to a circular curve, while the trajectory is horizontal and goes to the right.
A top tangency point corresponds to a point of this curve whose imaginary part is locally maximal, while for a bottom point it is locally minimal.}
With the numbering of \cref{eq:def:phi:m}, if $m$ is even then $T_{\phi_m}(\eps)$ is of top type and if $m$ is odd then $T_{\phi_m}(\eps)$ is of bottom type.

\begin{definition}\label{def:double:tg}
  Given a double tangent trajectory with an entry point $T_{\phi}(\eps)$ and an exit point $T_{\phi'}(\eps)$, the angles $\phi,\phi'$ modulo $2\pi$ are called the \emph{entry and exit angles} (or entry and exit indices) of $\gamma$.
  Given $\phi$ and $\phi'$ in $D$ we denote by $\cal{DT}_{\phi,\phi'}$ the set of parameters $\eps$ which have a double tangent trajectory with entry and exit indices $\phi$ and $\phi'$. 
  In particular:
  \[\cal{DT} = \bigcup_{\phi,\phi'\in D} \cal{DT}_{\phi,\phi'}.\]
\end{definition}

Repelling and attracting directions alternate, are equally spaced, and the $\phi_m = \frac{m+\frac{1}{2}}{k}\pi\in D$ are situated in between, at equal distance.
So $\phi_m$ has a nearest attracting direction, which is $\phi_m+(-1)^m\frac{\pi}{2k}$ and a nearest repelling direction, which is $\phi_m-(-1)^m\frac{\pi}{2k}$.

Recall that the negated periods $-\tilde \mu_j = -2\pi i/(k+1)z_j^{k}$, $z_j^{k+1}=\eps$ of the ideal vector field $\tilde \omega_\eps(z) = z^{k+1}-\eps$ form the sides of a regular $k+1$-gon that we called the \emph{ideal period-gon}.
\begin{definition}\label{def:bif:dir}
We call \emph{bifurcation directions} the values of $\arg \eps$ such that the ideal period-gon has a vertical axis of symmetry.
This set is
\begin{equation*}
\setof{\frac{\pi}{2} + \frac{\pi}{2k} + \frac{n\pi}{k} }{ 0\leq n\leq 2k-1} = \begin{cases}
\{\frac{2n\pi}{2k}\mid 0\leq n\leq 2k-1\}, & k\ \text{odd},
\\
\{\frac{(2n+1)\pi}{2k}\mid 0\leq n\leq 2k-1\}, & k\ \text{even}.
\end{cases}
\end{equation*}
\end{definition}

\begin{theorem}\label{thm:bifurcation}
Provided $r$, then $\rho$, are chosen small enough, the following properties hold.
\begin{enumerate}
\item\label{item:tc0} The bifurcation locus of the family $\omega_\eps$ coincides with the double tangency locus $\cal{DT}$.
\item\label{item:tc1} Each of the $4k^2$ sets $\cal{DT}_{\phi,\phi'}$ is a real analytic arc from $(0,0)$ to $\partial B(0,\rho)$.
We call it a \emph{bifurcation curve}.
\item\label{item:tc2} The curve $\cal{DT}_{\phi,\phi'}$ has a half-tangent at the origin, of direction $\theta_{\phi,\phi'}$ that can be determined as follows: 
Let $\phi_a$ be the attracting direction nearest to $\phi$. 
Let $\phi_r$ be the repelling direction nearest to $\phi'$.
Then $\theta_{\phi,\phi'} = \phi_{\mathrm{mid}} +\frac{\pi}{2}$ where $\phi_{\mathrm{mid}}$ is the the midpoint, in the circle of directions, of the arc from $\phi_a$ to $\phi_r$ in the anticlockwise orientation.
\item\label{item:tc3} The union of $\theta_{\phi,\phi'}$ over all indices in $D^2$ is the set of bifurcation directions as in \Cref{def:bif:dir}.
\item\label{item:tc4} Let $s=|\eps|^{\frac1{k+1}}$ and $\theta=\arg\eps$. The curve $\cal {DT}_{\phi,\phi'}$ has an equation of the form $\theta = \theta_{\phi,\phi'} +  f_{\phi,\phi'}(s)$, where $f_{\phi,\phi'}$ is analytic in $s$  and $f_{\phi,\phi'}(0)=0$.
\end{enumerate}
\end{theorem}

The (closed) arcs of $\partial B(0,r)$ between two consecutive tangency points are regions in the interior of which the vector field either enters or exits $B(0,r)$.
We denote them respectively by $A_\phi(\eps)$ and $R_\phi(\eps)$, where $\phi$ is the argument of the attracting/repelling axis that intersects the arc.
See \Cref{fig:bord}.

Recall that we defined in \Cref{def:bif:etc} the connection locus $\cal C$ as the set of parameters for which there is a boundary connection.
A boundary connection has an entry arc, which is one of the $A_{\phi_a}(\eps)$, and an exit arc, which is one of the $R_{\phi_r}(\eps)$.
\begin{definition}\label{def:connection}
We denote by $\cal C_{\phi_a,\phi_r}$ the set of parameters $\eps\in B(0,\rho)$ such that there is a boundary connection with entry arc $A_{\phi_a}$ and exit arc $R_{\phi_r}$.
\end{definition}
Hence
\[\cal C = \bigcup_{\phi_a,\phi_r}\cal C_{\phi_a,\phi_r}
,\]
where the union is not necessarily disjoint.

\begin{figure}[htbp]
\begin{center}
\includegraphics[width=6cm]{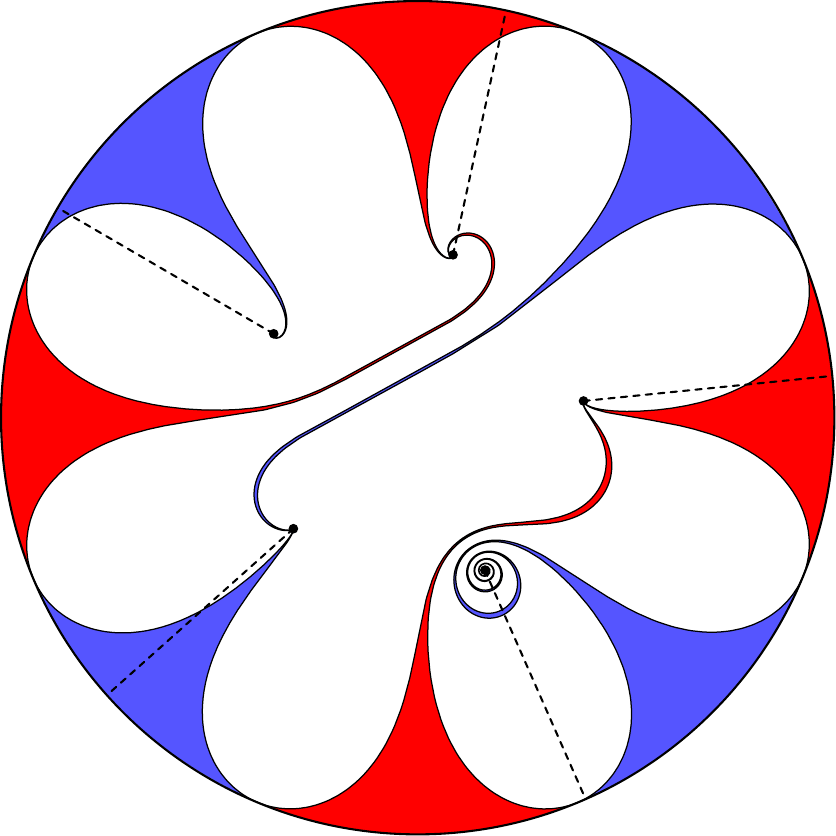}
\includegraphics[width=6cm]{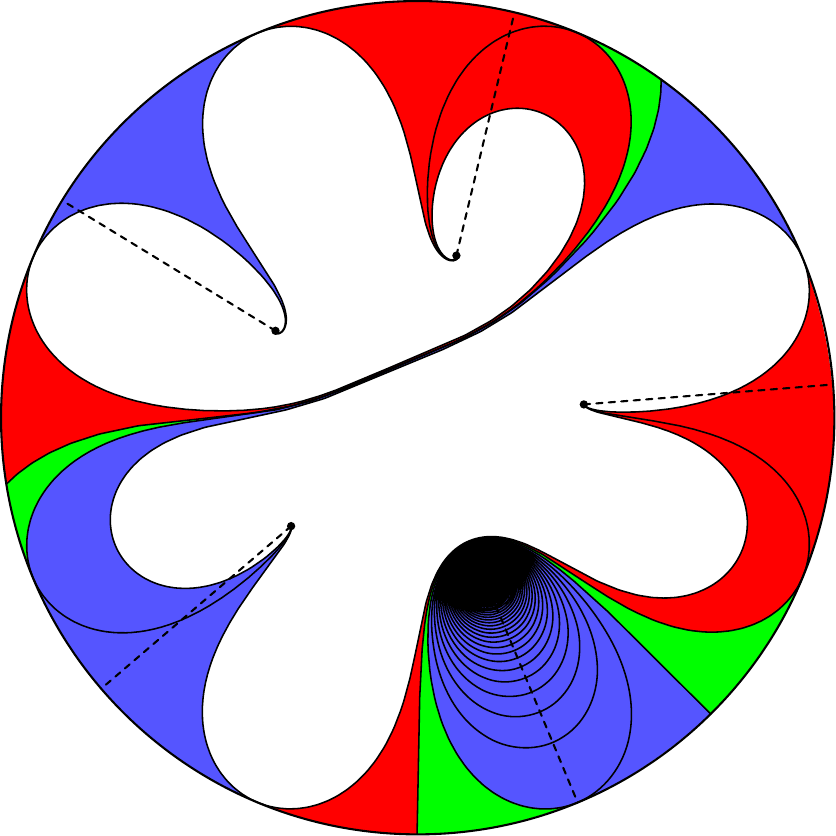}
\caption{Restricting the vector field to B(0,r), the notion of separatrix becomes a bunch of exiting or entering trajectories. They cover regions indicated in colors.
These images have been accurately drawn by numerical methods in the case of the vector field $\dot z = z^5-\eps$ on $B(0,1)$ for two values of $\eps$.
Green: boundary connections. Blue union green: incoming trajectories. Red union green: outgoing trajectories.
In the left image, one can define is a zig-zag of trajectories linking the singularities, by taking any one trajectory in each (open) white region. In the right image, the boundary connections prevent this.
}
\label{fig:d12}
\end{center}
\end{figure}

It is immediate that $\cal {DT}\subset \cal C$.

\begin{proposition}\label{prop:CDT}
Assume $r$, then $\rho$, are chosen small enough.
Let $\phi$ and $\phi'$ be the angles of two intermediate axes of $\omega_0$.
Let $\phi_a$ be the attracting direction nearest to $\phi$ and $\phi_r$ the repelling direction nearest to $\phi'$.
Then
\[\cal{DT}_{\phi,\phi'}\subset \cal C_{\phi_a,\phi_r}.\]
\end{proposition}

\begin{theorem}\label{thm:connection}
Still under the hypothesis that $r$, then $\rho$, are chosen small enough:
\begin{enumerate}
\item\label{item:cx:asymp} $\cal C_{\phi_a,\phi_r}$ is asymptotic at $0$ to the direction $\phi_{\mathrm{mid}} + \frac{\pi}2$ where $\phi_{\mathrm{mid}}$ is the midpoint, in the circle of directions, of the arc from $\phi_a$ to $\phi_r$ in the anticlockwise orientation.
\item\label{item:cx:1} Consider an attracting direction $\phi_a$ and a repelling direction $\phi_r$. Denote by $\phi_{2\ell-1},\phi_{2\ell}\in D$ the intermediate axis directions closest to $\phi_a$ and $\phi_{2m},\phi_{2m+1}\in D$ the ones closest to $\phi_r$.
The set $\cal C_{\phi_a,\phi_r}$ is closed in $B(0,\rho)\setminus\{0\}$ and bounded by the two bifurcation curves $\cal {DT}_1 =\cal {DT}_{\phi_{2\ell-1},\phi_{2m}}$ and 
$\cal {DT}_2 = \cal {DT}_{\phi_{2\ell},\phi_{2m+1}}$, which are both tangent to the same direction at $0$ (see Point~\ref{item:tc2} of \Cref{thm:bifurcation}).
The curve $\cal {DT}_1$ precedes $\cal {DT}_2$ in trigonometric order and $\cal C_{\phi_a,\phi_r}$ is the cusp-like domain contained in between.
In this section, the curves $\cal {DT}_1$, $\cal {DT}_2$ and the associated bifurcations are called \emph{essential}.
\item\label{item:cx:width} $\cal C_{\phi_a,\phi_r}$ has an opening of order $|\eps|^{\frac{k}{k+1}}$.
\item\label{item:cx:2} The curves $\cal {DT}_{\phi_{2\ell},\phi_{2m}}$ and $\cal {DT}_{\phi_{2\ell-1},\phi_{2m+1}}$ are contained in the interior of $\cal C_{\phi_a,\phi_r}$ (i.e.\ they lie strictly between $\cal {DT}_1$ and $\cal {DT}_2$). Their relative position within $\cal C_{\phi_a,\phi_r}$ depends on the family $\omega_\eps$ and they may coincide (see Figure~\ref{local_bif_diag} for a potential local bifurcation diagram in a sectorial neighborhood of $\mathcal{C}_{\phi_a,\phi_r}$). We call them, and the associated bifurcations, \emph{inessential}.
\end{enumerate} 
\end{theorem}

The terms essential and inessential are justified as follows: crossing $\mathcal{DT}_1$ and $\mathcal{DT}_2$ leads to the appearance or disappearance of boundary connections, while crossing the other curves does not.
See \Cref{sub:descr:bif} for more details.

\begin{figure}
\begin{tikzpicture}
\node at (0,0) {\includegraphics[width=12cm]{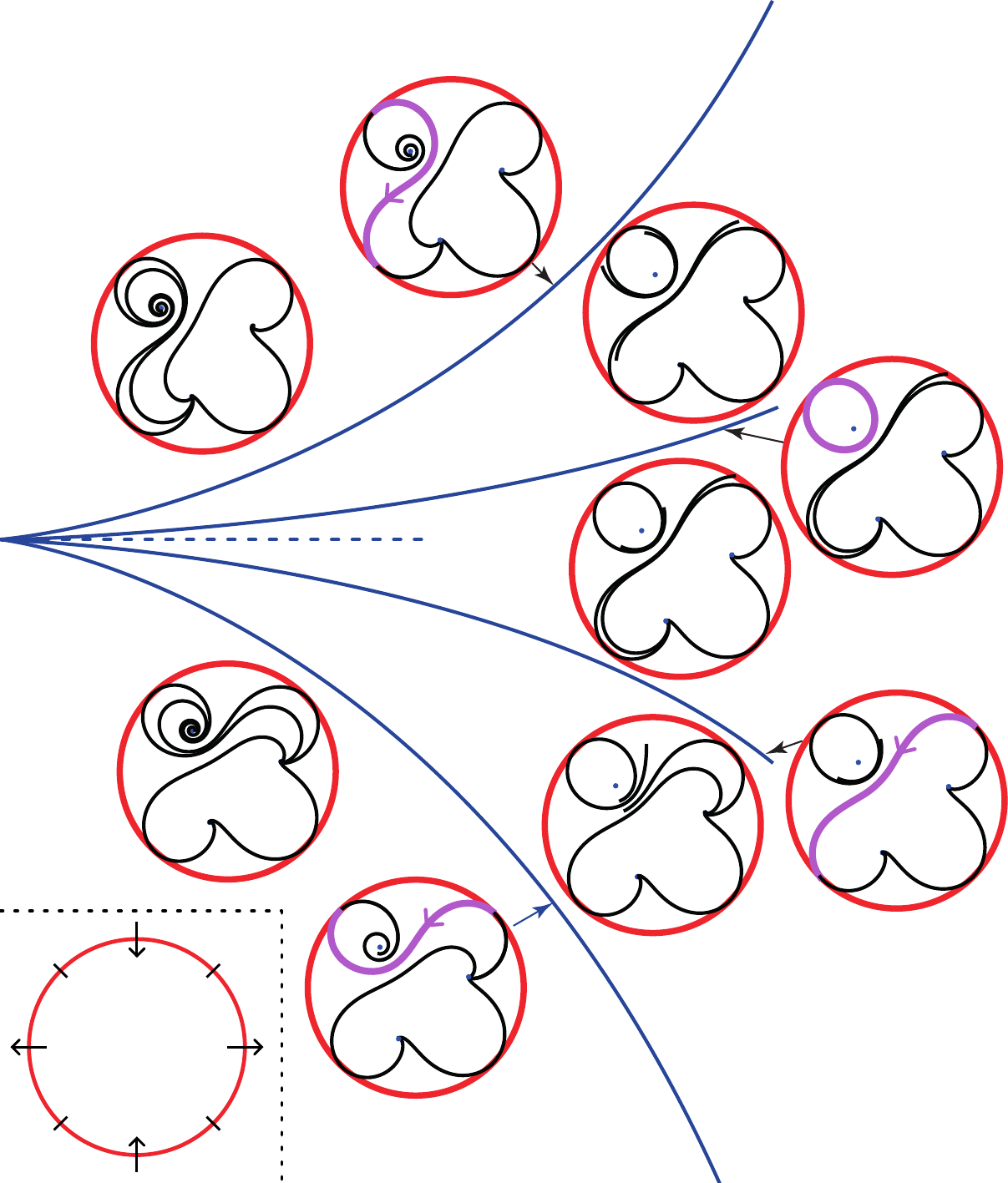}};
\node at (-6,0.95) {$0$};
\node at (2.1,6.5) {$\cal{DT}_{\phi_1,\phi_2}$};
\node at (-0.7,1.5) {$\cal{DT}_{\phi_1,\phi_1}$};
\node at (-0.7,-0.6) {$\cal{DT}_{\phi_0,\phi_2}$};
\node at (2.7,-5.5) {$\cal{DT}_{\phi_0,\phi_1}$};
\node at (-0.4,0.6) {$\pi/4$};
\node at (-3.1,-4.2) {$\phi_0$};
\node at (-4.35,-4.6) {\small $\pi/2$};
\node at (-5.6,-4.2) {$\phi_1$};
\node at (-5.2,-5.45) {\small $\pi$};
\node at (-5.6,-6.7) {$\phi_2$};
\node at (-3.1,-6.7) {$\phi_3$};
\end{tikzpicture}
\caption{Case $k=2$. A possible local bifurcation diagram around a cusp-like region, of asymptotic angle $\pi/4$ in $\eps$-space.
For presentation purposes, we turned the parameter picture by $-\pi/4$, so as to make this axis horizontal
The boundary connection region $\cal C_{\pi/2,\pi}$ is bounded by the essential bifurcations of double tangency $\cal{DT}_{\phi_1,\phi_2}$ and $\cal{DT}_{\phi_2,\phi_3}$.  For a generic vector field, the inessential bifurcation curves do not coincide.
In every case, they must sit in between $\cal{DT}_{\phi_1,\phi_2}$ and $\cal{DT}_{\phi_2,\phi_3}$.
The red circles represent the disk $B(0,r)$ in phase space, black and purple curves are trajectories, the purple ones being double tangencies.
See \Cref{local_bif_diag_t-coord} for an interpretation in terms of the rectifying coordinate.}
\label{local_bif_diag}
\end{figure}

\begin{figure}[htbp]
\begin{center}
\begin{tikzpicture}
\node at (0.05,0.35) {\includegraphics[width=\textwidth]{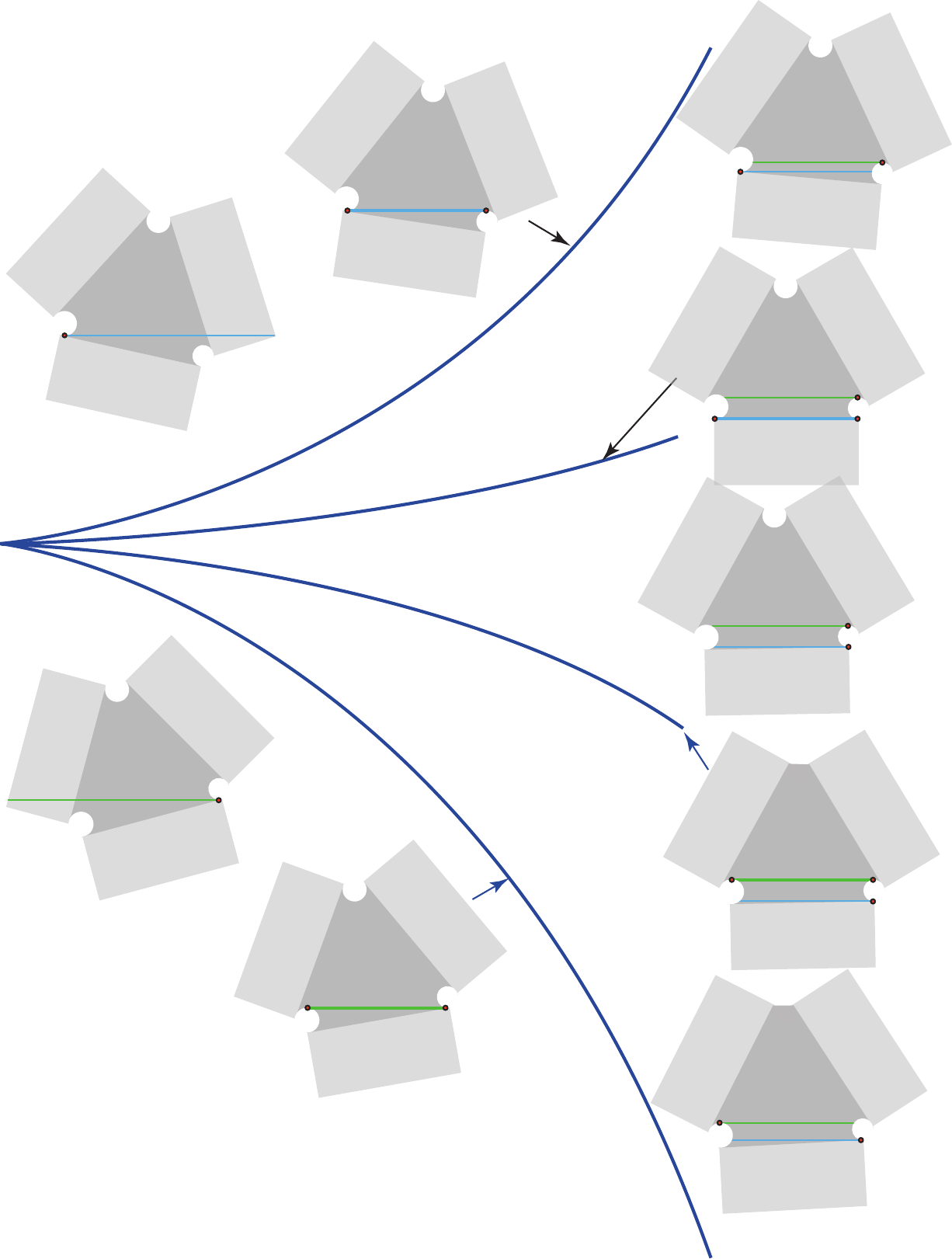}};
\node at (1.8,7.2) {$\cal{DT}_{\phi_1,\phi_2}$};
\node at (-0.3,2.45) {$\cal{DT}_{\phi_1,\phi_1}$};
\node at (-0.1,0.85) {$\cal{DT}_{\phi_0,\phi_2}$};
\node at (2.15,-4.15) {$\cal{DT}_{\phi_0,\phi_1}$};
\end{tikzpicture}
\end{center}
\caption{The local bifurcation diagram of Figure~\ref{local_bif_diag} in the rectifying coordinate.  The double tangent trajectories are drawn thicker. The dots indicate the tangency points.}
\label{local_bif_diag_t-coord}
\end{figure}

\begin{remark}
While $\mathcal{DT}$ and $\mathcal{C}$ depend on the radius of the domain $B(0,r)$, or its shape for more general domains, the existence of cusp like regions with separating solutions for all topological type is independent of the radius or shape of the domain $B(0,r)$, and this is what makes them interesting.
Finally, in the setting of discrete dynamics, generic $1$-parameter perturbations of a parabolic fixed point $f(z) = z^{k+1}+\cal O(z^{k+2})$ should be well approximated by a corresponding generic family of vector fields as in \Cref{def:gen_vf} and we believe that the part of the bifurcation locus related to the origin will be contained in $\cal C$ and can be described as a union of sets related to each curve $\cal {DT}$.
\end{remark}

\begin{proposition}\label{prop:bif2:6}
Consider the following family
\[ \dot z = \omega_\eps(z) = z^{k+1}+a z^{k+2} - \eps + a' \eps z + O(z^{k+3},\eps z^2,\eps^2)
,\]
where $(a,a')\in\C^2$.
For a generic pair $(a,a')$ (more precisely for $(a,a')$ outside a finite set of real hyperplanes in $\C^2 \cong \R^4$), the $4k^2$ arcs in \Cref{thm:bifurcation} are all disjoint.
\end{proposition}
As a consequence, this is also the case for ``most'' vector field families. 

\subsection{Description of the bifurcations}\label{sub:descr:bif}

From now on we will make abundant use of the geometric model of time coordinate described in \Cref{sec:geom}, which we sum up also in the proofs section (\Cref{sub:bif:pf}).

From the model, we will deduce that the bifurcations occur exactly when there are double tangent trajectories, which correspond exactly in time coordinate to the horizontal segments $H$ from the topmost or bottommost point of an eyelet $E$ to the topmost or bottommost point of another $E'$.
See \Cref{local_bif_diag_t-coord,fig:ess:bif:model,fig:iness:bif:model}.

This requires the eyelets to have the same imaginary part, up to a bounded amount, so the $k+1$ segments $S_j$ bounding the image of the periodic domains, and supported by vectors $-\mu_j$, are close to forming a regular $k+1$-gon with vertical axis of symmetry.

As $\arg \eps$ varies while $|\eps|$ is kept constant, the position of the center of the eyelets varies fast, compared to the position of the topmost/bottommost points of the eyelets relative to the center. On first approximation these centers are on a circle of radius tending to $\infty$ as $\eps$ tends to $0$, and their argument is approximated by $-\frac{k}{k+1}\arg \eps$ plus a quarter turn.
It follows that for parameters close to a bifurcation parameter, increasing $\arg \eps$ while keeping $|\eps|$ fixed, we will see the eyelet $E'$ plunge downwards (in a direction where the imaginary part decrease) while $E$ ascends.

From this, we get the following description of the possible bifurcations:
The case where the horizontal trajectory $H$ links a topmost point of $E$ to a bottommost point of $E'$ (or the opposite) is illustrated in \Cref{fig:ess:bif:model}.
For slighlty smaller values of $\arg \eps$, the connection between the two eyelets disappears: there are no more trajectories between them. For slightly bigger values of $\arg \eps$, there is a bunch of parallel horizontal lines linking the two eyelets, none of which is double tangent. They appear in green in \Cref{fig:ess:bif:model} while the double tangent trajectory is in yellow. We call this an \emph{essential bifurcation}.

The case where the horizontal trajectory $H$ links a topmost point of $E$ to a topmost point of $E'$ (or two bottommost points) subdivides into more subcases, according to which eyelet has a bigger height (they may also have the same height for the $\eps$ where the double tangency occurs).
In any case the bunches of boundary connections still exist for nearby $\eps$, which is why we call this an \emph{inessential bifurcation}.
We illustrated one case in \Cref{fig:iness:bif:model}, where the right eyelet is less tall.
Then, for slightly smaller values of $\arg \eps$, there are exiting trajectories that do not exist anymore for bigger values of $\arg \eps$, for which instead a bunch of entering trajectories appears.

\begin{figure}[htbp]
\begin{center}
\includegraphics[width=\textwidth]{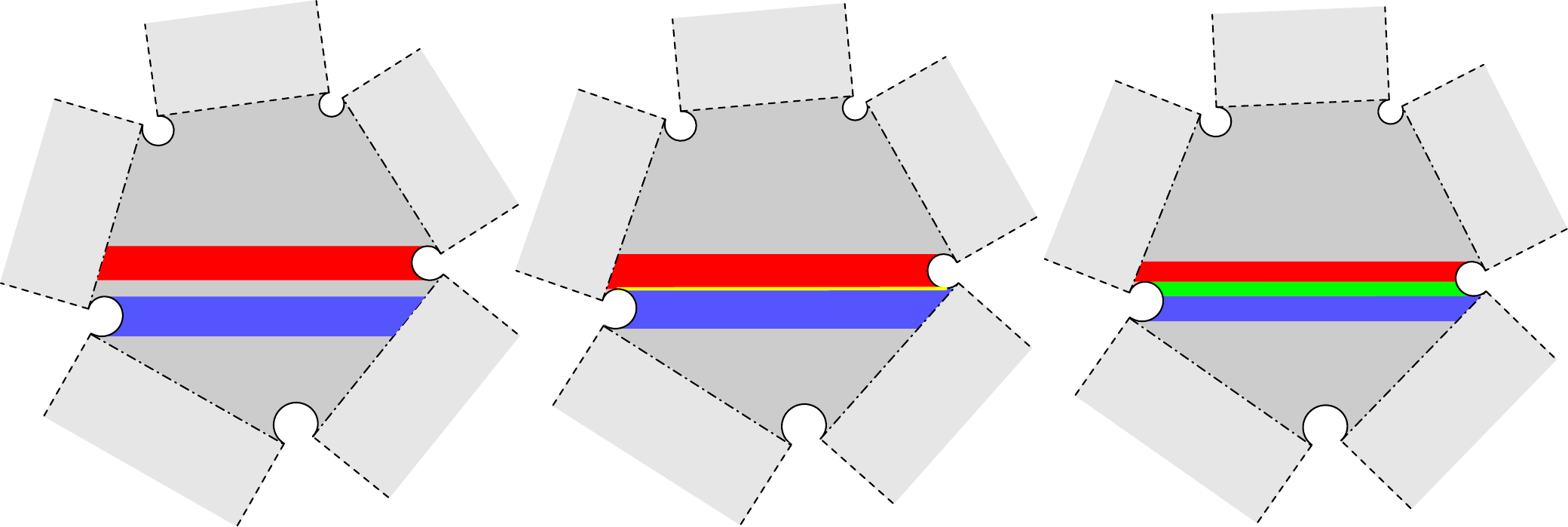}
\end{center}
\caption{An essential bifurcation. Exiting trajectories in red, entering trajectories in blue, boundary connections are in green, while the double tangent trajectory is in yellow.}
\label{fig:ess:bif:model}
\end{figure}

\begin{figure}[htbp]
\begin{center}
\includegraphics[width=10cm]{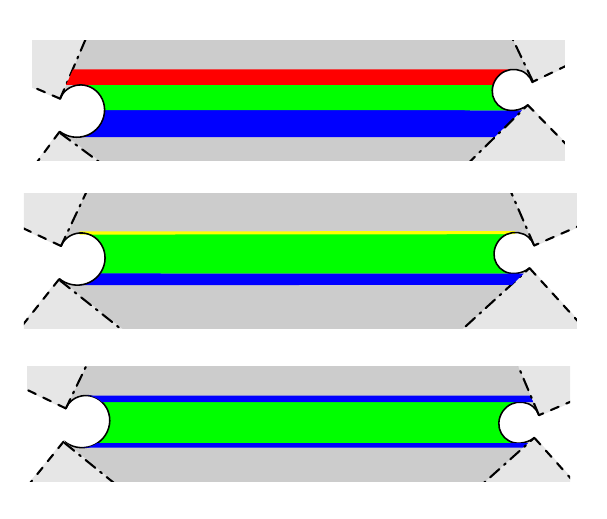}
\end{center}
\caption{An inessential bifurcation. The conventions are the same as in \Cref{fig:ess:bif:model}.}
\label{fig:iness:bif:model}
\end{figure}

\FloatBarrier

\subsection{The case of the ideal vector field}\label{sub:ideal:case}

In the case of $\dot z = z^{k+1}-\eps$, we can take advantage of the detailed study of \Cref{sec:ideal:global} to give a simple argument and description, that can serve as a training for the general case.

\begin{theorem}\label{thm:bifan}
For the ideal vector field, each group of bifurcations tending to $0$ along a given bifurcation direction $\theta$ contains an odd number $n_\theta\geq3$ of curves, one of which is a straight ray. The other ones come in symmetric pairs on each side of the straight ray, with a tangency at $\eps=0$ of order $2-1/(k+1)$. See \Cref{fig_2b} for an illustration.

The number $n_\theta$ can be computed explicity. For $\arg(\eps) = \theta$, the period gon is symmetric with respect to the vertical axis and there exist horizontal segments joining the vertices two by two. The number $n_\theta$ is simply 
\begin{equation}\label{def_n_j}
n_\theta=1+2m,
\end{equation}
with $m$ being equal to the number of different lengths of this set of horizontal segments, namely 
\begin{equation}
m=\begin{cases} \frac{k}2,&k\;{\rm even},\\
\lfloor\frac{k+1}4\rfloor,&k\:{\rm odd,\: } \theta = \frac{\pi j}{k},\, j\;{\rm even},\\
\lfloor\frac{k+3}4\rfloor,&k\:{\rm odd,\: } \theta = \frac{\pi j}{k},\, j\;{\rm odd}.
\end{cases}
\end{equation}
\end{theorem}

In the rest of \Cref{sub:ideal:case} we prove this theorem.

\begin{proposition}\label{prop:ideal}
  In the case of $\dot z = z^{k+1}-\eps$, the slits are straight segments from the singularities $z_j$ to the contact points $u_j$, which are $u_j = r z_j/|z_j|$.
  The periodic region $U(z_j)$ is symmetric with respect to the line $L_j$ supporting this segment.
  Denoting by $R$ the rotation by $2\pi/(k+1)$ and $z_{j+1} = R(z_j)$ we have $U(z_{j+1}) = R(U(z_j))$.
  The domains $U$ and $D'$ are thus invariant by $R$.
  The time coordinate $t$ defined on $D'$ and vanishing at $0$ commutes with $R$, hence the geometric model $t(D')$ is invariant by $R$.
  The symmetry along $L_j$ is conjugated by $t$ to a symmetry along some axis through $0$.
\end{proposition}

\begin{proof}
In \Cref{sec:ideal:global}, we described the many symmetries of this vector field, over $\CP^1$, and the image under the time coordinate of the plane slit along the lines from $z_j$ to $\infty$ that are radial w.r.t. $0$. This image was called the \emph{star shaped domain} and we denote it by $S$ here. 
It consists in a regular $k+1$-gon centered on $0$, which we called the \emph{period-gon}, on the sides $s$ of which are attached $k+1$ semi-infinite strips of sides orthogonal to $s$.
The time coordinate commutes with the rotation $R$ by $2\pi /(k+1)$ and conjugates any reflection symmetry preserving the set of singularities to a reflection symmetry of $S$.

Restricting the vector field to $B(0,r)$ removes from $S$ a collection of $k+1$ neighborhoods of the $k+1$ vertices $v$ of $S$, that are bounded by curves $C'_v$ that tend, when $\eps\to 0$ to arcs of circles centered on the vertices of $S$, of radius $1/kr^k$ and of angle $2\pi - 2\pi/(k+1)$. We have $C'_{R(v)} = R(C'_v)$. The tangent to $C'_v$ at an endpoint $t\in\C$ is orthogonal to the line bounding $S$ and containing $t$.
We denote by $S'$ the image of $B(0,r)$ in $S$.

Consider two consecutive vertices $v$ and $v'$ of $S$.
Consider the half-strip $H$ of $S$ attached to $[v,v']$.
It is associated to a singularity $z_j$ (a point tends to $z_j$ if and only if its image in $S$ tends to infinity along $H$).
Consider the maximal half-strip $H'$ contained in $H\cap S'$. 
It follows from the above that $H'$ touches an endpoint of $C'_v$ and an endpoint of $C'_{v'}$.
The preimage of $H'$ is thus $U(j)$.
Examples can be found in \Cref{fig:bifk4,fig:bifk5A,fig:bifk5B}.
\end{proof}

\begin{figure}[htbp]
\begin{center}
\includegraphics[width=5.6cm]{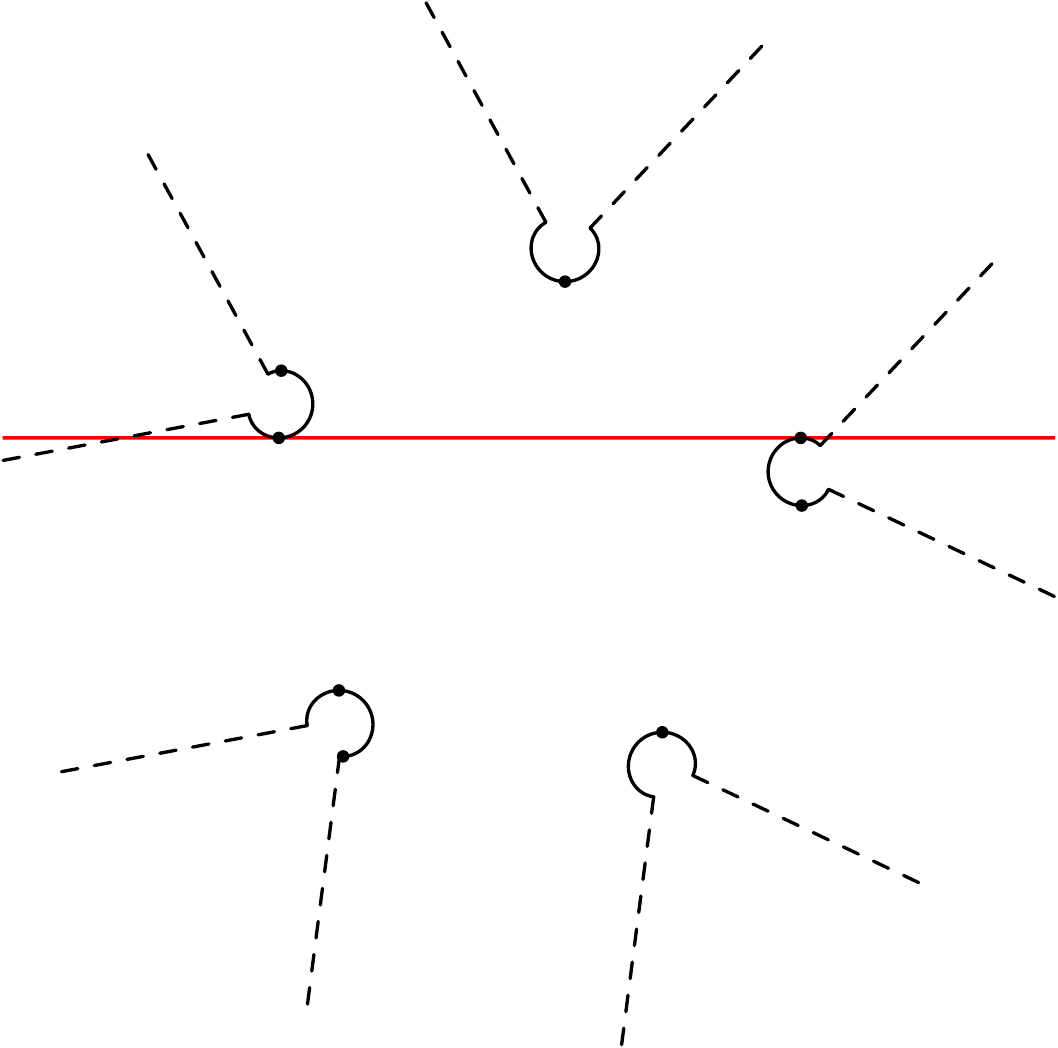}
\quad
\includegraphics[width=5.6cm]{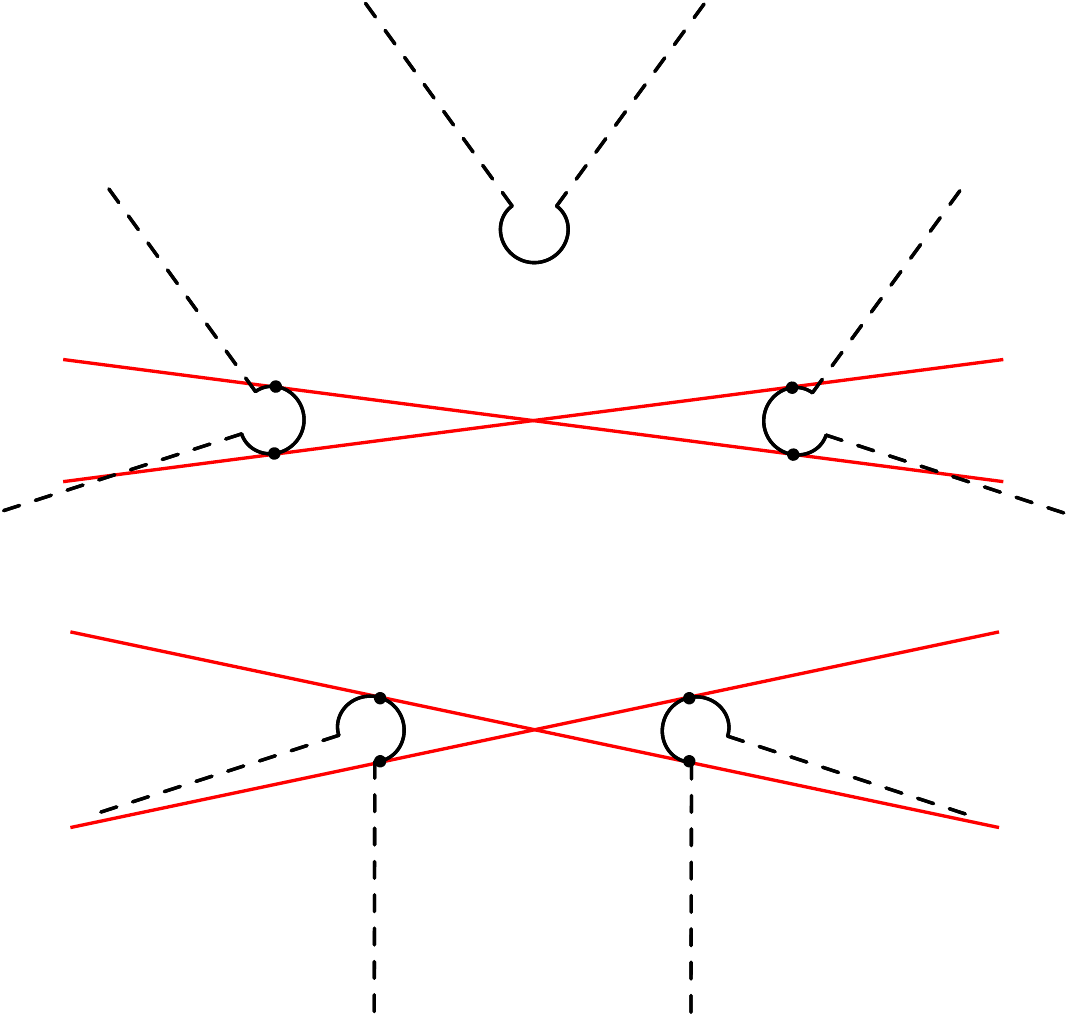}
\vskip1em
\includegraphics[width=5.6cm]{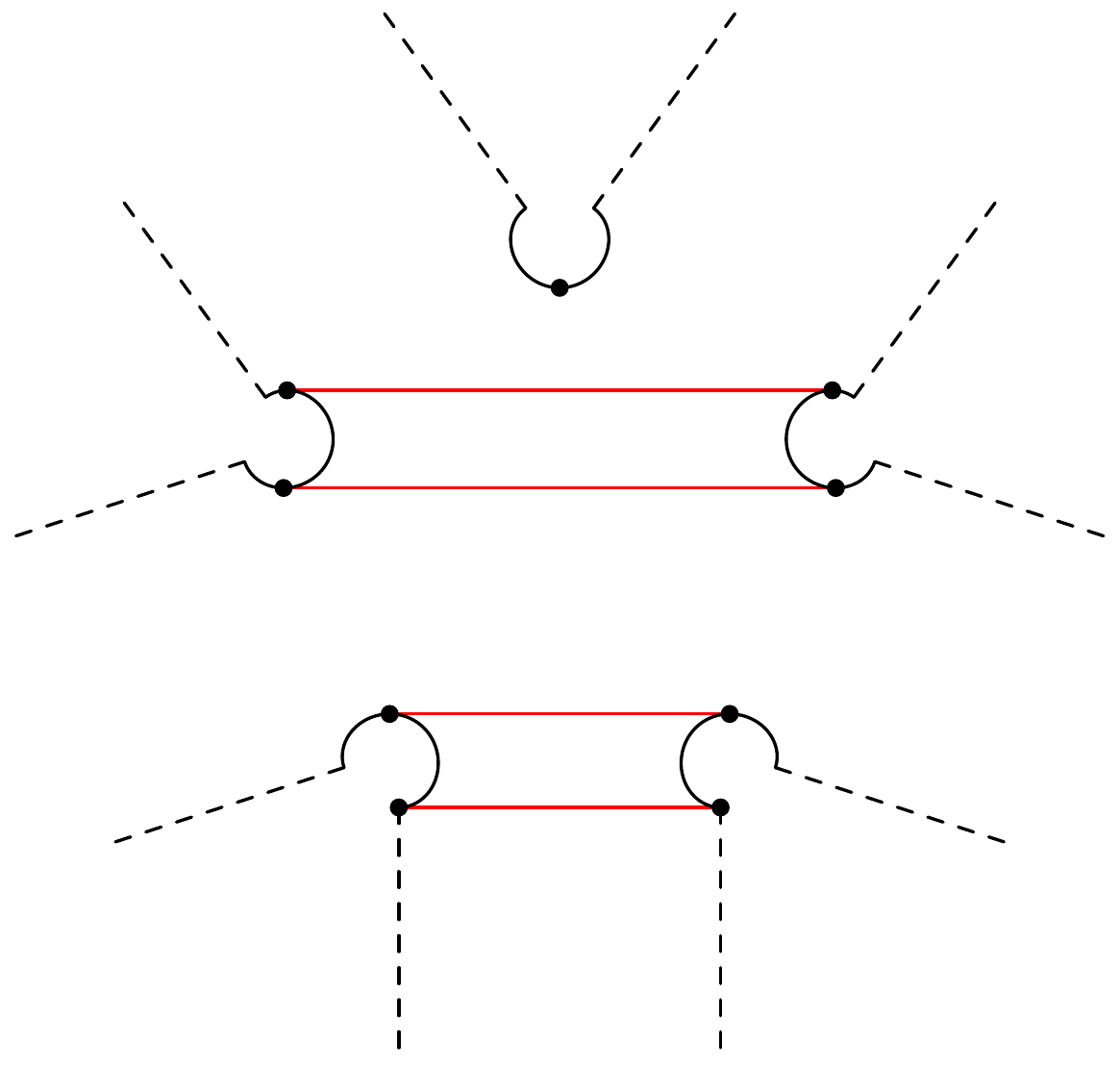}
\quad
\includegraphics[width=5.8cm]{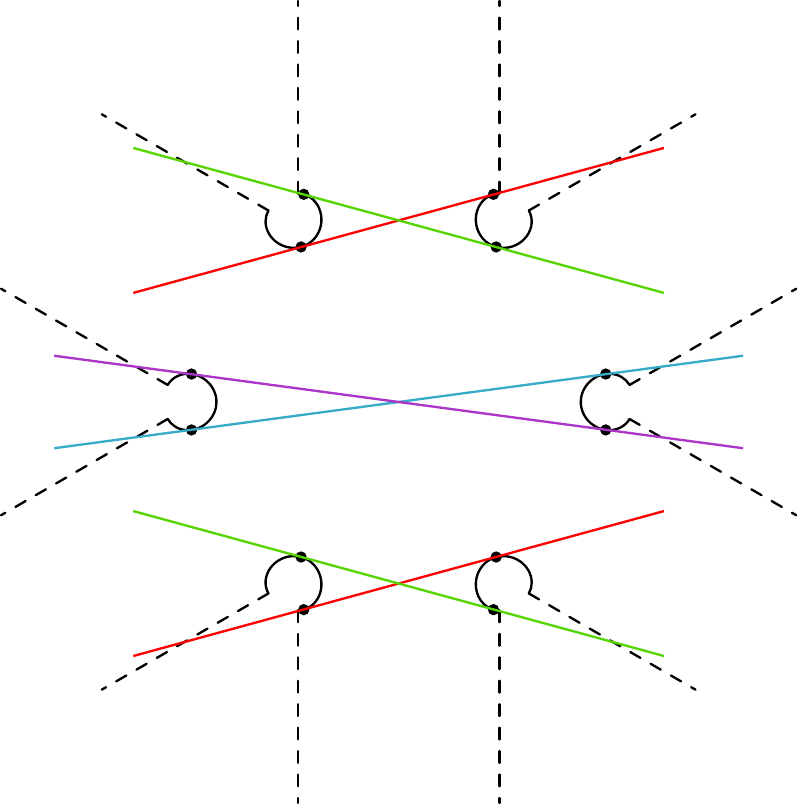}
\caption{
\footnotesize
In the case of the ideal vector field, adding $\delta\theta$ to $\arg \eps$ while maintaining $|\eps|$ just rotates the image $t(D')$ by $-\delta \theta/(k+1)$.
When the star is symmetric with respect to the vertical axis, many tangencies occur simultaneously. These pictures have been accurately numerically computed. Note how the eyelets fail to be exactly circular arcs.
\\
Top-left: For this value of $\theta$, only one double tangent trajectory occurs: the part of the red horizontal line between the two tangency points.
The black dots mark the points where the tangent to the curves is horizontal: they are the images in time coordinate of the tangency points. \\
Bottom-left: For this value of $\theta$, four inessential  tangencies occur at the same time. (They may occur at different times when we consider the general case.)
\\
Top-right and bottom-right: we drew in red the lines that are tangent to two eyelets and that are nearly horizontal. Those lines will be the horizontal lines with essential double tangencies, for values of $\theta$ that are close. In the top picture, no two tangencies occur at the same time, while on the bottom one, there are two orientations for which two tangencies occur at same time, indicated as follows: parallel lines have the same color. In any case, two lines having double tangencies with the same pair of eylets necessarily correspond to bifurcations occuring at different times.}
\label{fig_W0}
\end{center}
\end{figure}

\begin{figure}[htbp]
\begin{center}
\includegraphics[width=6cm]{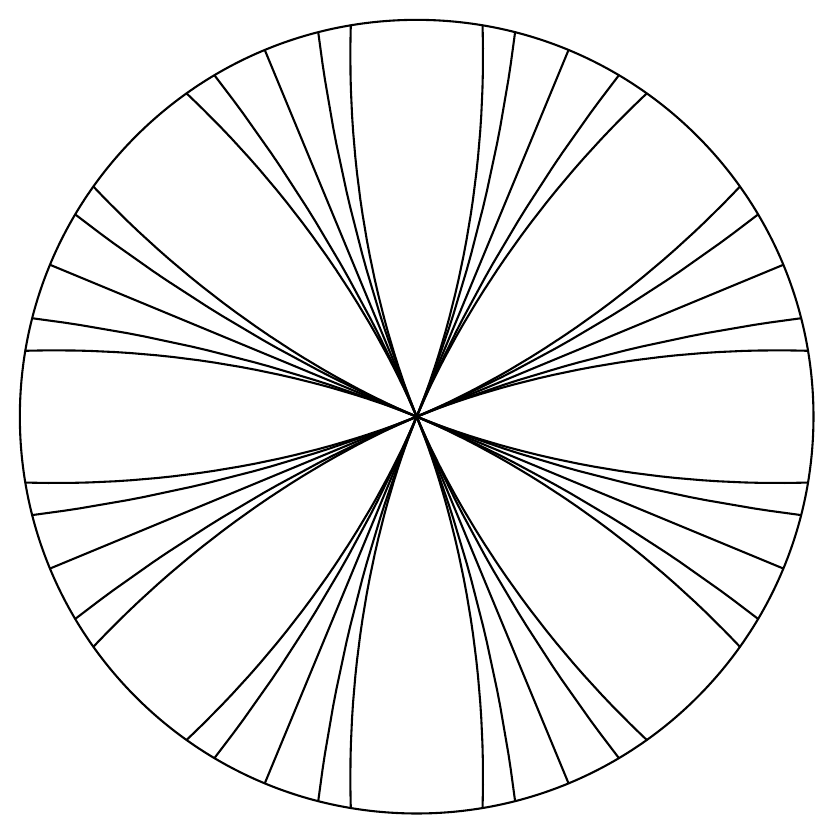}
\quad
\includegraphics[width=6cm]{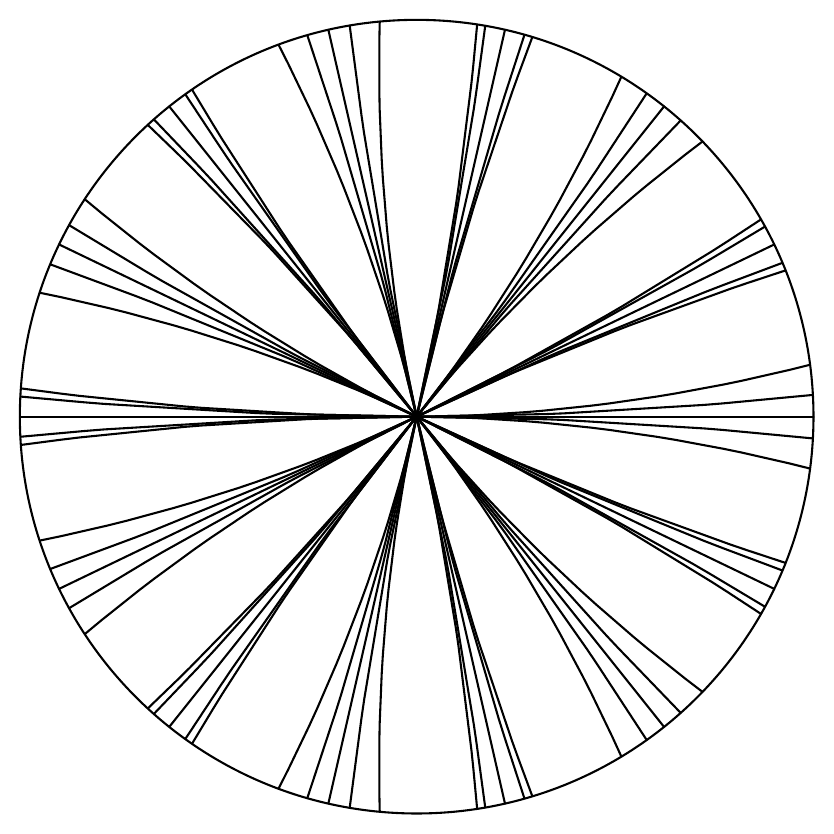}
\caption{Double tangency loci near $\eps=0$ for $r=1$ and $k+1=5$ (left image, $\rho=0.35$) and $k+1=8$ (right image, $\rho=0.26$). Those pictures have been accurately computed by a numerical method.}
\label{fig_2b}
\end{center}
\end{figure}

\begin{figure}[htbp]
\begin{tikzpicture}
\node at (0,0) {\includegraphics*[scale=0.45]{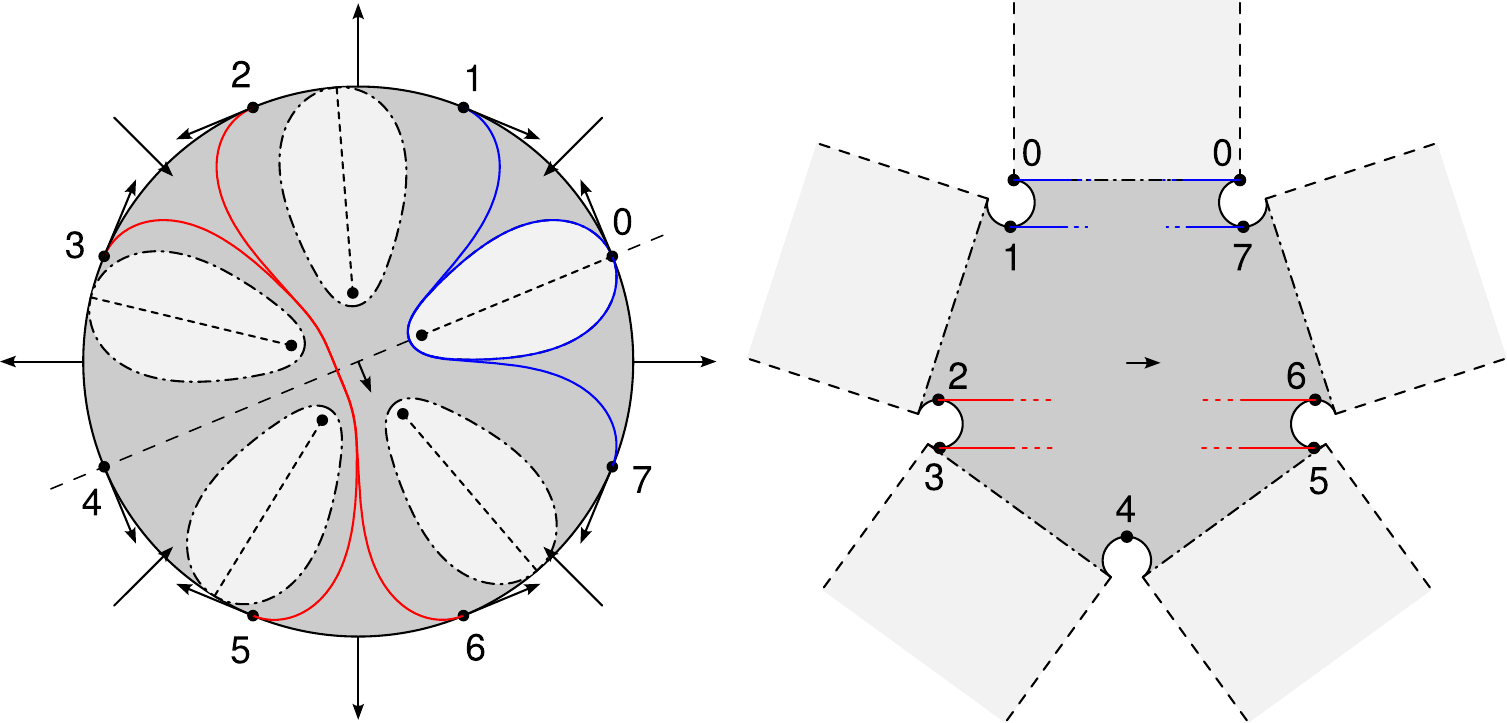}};
\node at (-5.55,-1.03) {$\Delta$};
\end{tikzpicture}
\caption{
In each figure in this series of three, we drew on the left an example with accurate trajectories for the ideal vector field $\dot z = z^{k+1}-\eps$, for an $\eps$ such that vector field has homoclinic connection, hence there is an antisymmetry axis $\Delta$ (axis of reversion).
The numbers $m$ label the tangency points (which depend slightly on $\eps$).
The arrows indicate the attracting and repelling axes, and the vector field direction at the tangency points.
Double tangent trajectories are drawn in colors.
\\
On the right we drew the image of the left part in the geometric model. However, we chose another value of $|\eps|$, one much closer to $1$, otherwise the nearly circular arcs are too small.
Here $k+1=5$.}
\label{fig:bifk4}
\end{figure}

\begin{figure}[htbp]
\begin{tikzpicture}
\node at (0,0) {\includegraphics*[scale=0.45]{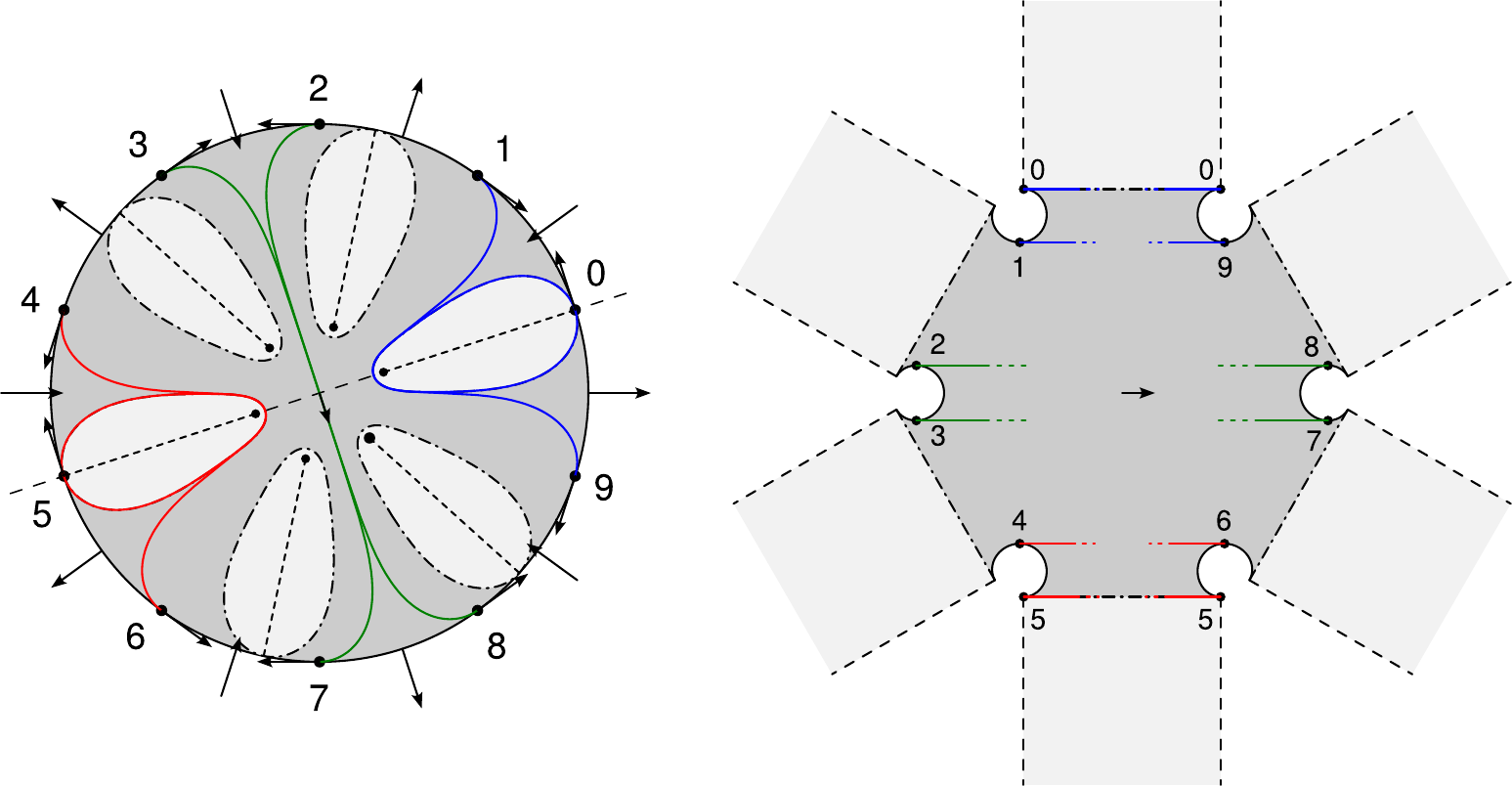}};
\node at (-6,-0.85) {$\Delta$};
\end{tikzpicture}
\caption{$k+1=6$}
\label{fig:bifk5A}
\end{figure}

\begin{figure}[htbp]
\begin{tikzpicture}
\node at (0,0) {\includegraphics*[scale=0.45]{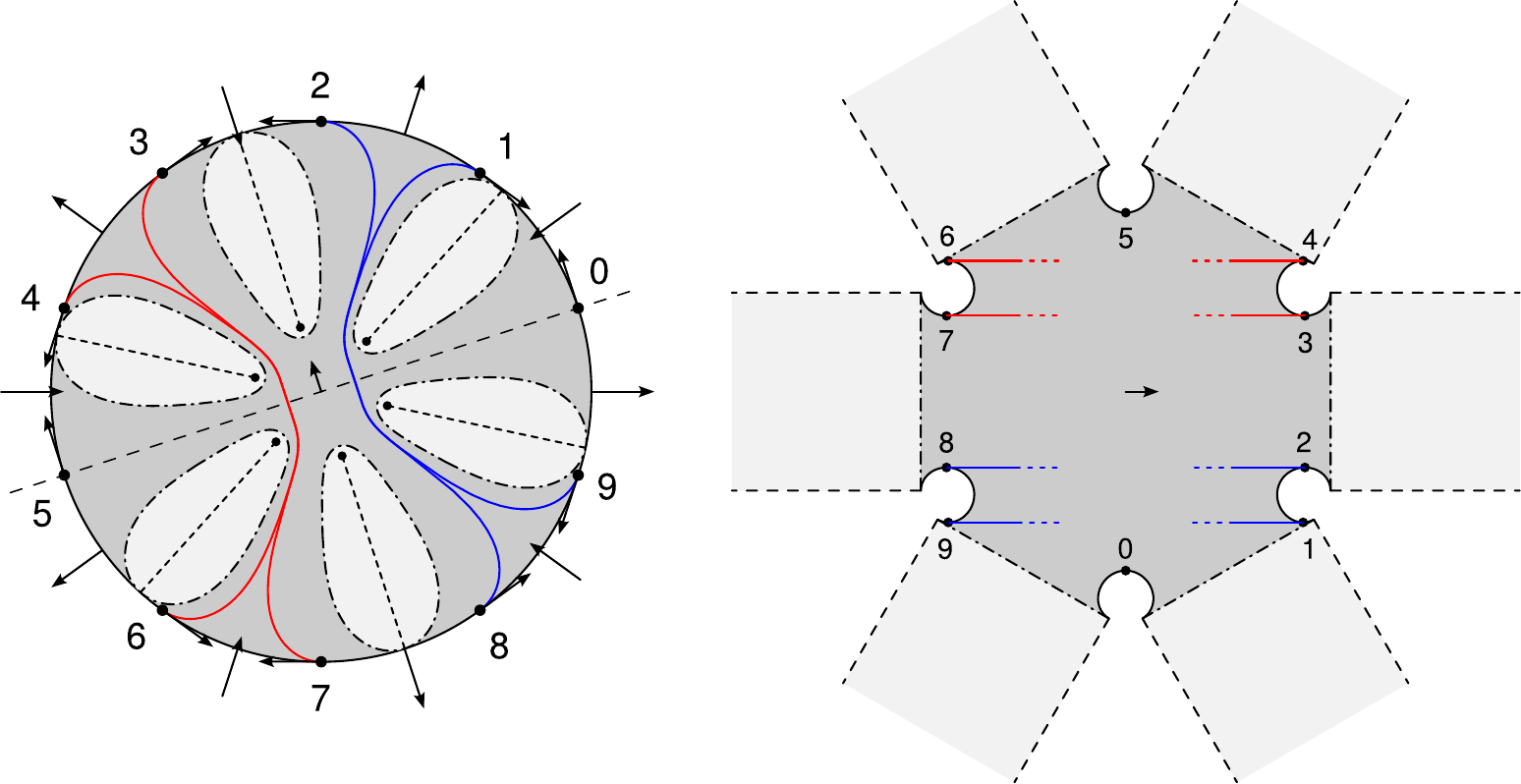}};
\node at (-6,-0.85) {$\Delta$};
\end{tikzpicture}
\caption{$k+1=6$}
\label{fig:bifk5B}
\end{figure}

It follows that when $\arg \eps$ is such that the period-gon has a vertical axis of symmetry, all the inessential bifurcations associated to the bifurcation direction $\arg \eps$ occur simultaneously, see the bottom left part of \Cref{fig_W0}.
Concerning the essential bifurcations they occur for $\arg \eps$ near but not equal to the bifurcation direction $\theta$. The distance of the vertex pair influences the angle difference, in such a way that if a vertex pair $v_a,v_b$ with the same imaginary part is chosen for $\arg \eps=\theta$, and another pair $v_c,v_d$ with the same imaginary part, then an essential bifurcation occurs simultaneously for them if and only if $|v_a-v_b|=|v_c-v_d|$ (there is one simultaneous bifurcation with $\eps$ close to and less than $\theta$, and one with $\eps$ close to and greater than $\theta$). See the right part of \Cref{fig_W0}. 

On \Cref{fig_2b} we accurately drew the bifurcation diagram (tangency locus) near $0$ for two values of $k$.

\subsection{Proofs}\label{sub:bif:pf}

In this section we prove the statements of \Cref{sub:bif:statements}.

\medskip

We start with \Cref{lem:tgcy:pts}, about basic properties of the tangency points between the vector field and $\partial B(0,r)$.

\begin{proof}[Proof of \Cref{lem:tgcy:pts}]
For $\eps=0$, the vector field $\omega_0$ expands as $\omega_0(z) = z^{k+1}+O(z^{k+2})$. For a given $r$, the points $z=re^{i\theta}$ for which the vector field is tangent to $\partial B(0,r)$ are characterized by the equation:
\[\arg \omega_0(re^{i\theta}) \equiv \theta + \pi/2 \bmod \pi. \]
For $r$ small enough, there are exactly $2k$ roots of this real analytic equation in $\theta$, which are close to the solutions of $(k+1)\theta \equiv \theta + \pi/2 \bmod \pi$.
Moreover, they are simple roots of the equation so, by the implicit function theorem, for small values of $\eps$ there are still exactly $2k$ solutions, which are still simple roots, and they depend real-analytically on $\eps$.
\end{proof}

\subsubsection{First consequences of the geometric model, limit ideal situation}\label{ss:first:csq}

Recall the periodic domains $U(z_j)$ of the singularities $z_j$ having contact points $u_j$ with $\partial B(0,r)$.
Recall also the ideal vector field $\tilde \omega_\eps (z) = z^{k+1}-\eps$ to which we will compare $\omega_\eps$.
Its negated periods, $-\tilde \mu_j = -2\pi i/(k+1)z_j^{k}$, $z_j^{k+1}=\eps$, form the sides of a regular $k+1$-gon that we called the \emph{ideal period-gon}.

We draw here consequences from the description of the translation surface associated to the vector field given in \Cref{sub:constr} and that we called the \emph{geometric model}.
We remind briefly here some essential features.
It consists in the image of the periodic domains, which are half-infinite strips attached to 
a central part, the image of $B(0,r)$ minus the periodic domains.
The central part is a simply connected set bounded by a Jordan curve made of alternating straight segments of vector $-\mu_j$, bounding a half-strip, and nearly circular arcs $C'_{j+1/2}$ of diameter that stays bounded as $\eps\to 0$, that we call \emph{eyelets}, and which are tangent to the segments at the point where they meet.
They are the image of the arcs $C_{j+1/2}$ on $\partial B(0,r)$ between $u_{j}$ and $u_{j+1}$.
The vectors $-\mu_j$ are in correspondence with the vectors $-\tilde \mu_j$ for the ideal vector field, in the sense that $\mu_j/\tilde \mu_j\tend 1$ as $\eps\to 0$. In particular their modulus tends to infinity.
However, the difference $\mu_j-\tilde\mu_j$ will tend to infinity in the generic case, so one should not imagine that the distance from the eyelets to the vertices of the ideal period-gon can be chosen bounded.
(It is to be noted that, \emph{scaling the model down}, as $\eps\to 0$, it tends to the star shaped model of the ideal vector field: the eyelets vanish and the segments tend to the ideal period-gon.)
In \Cref{fig:d16} we give the image in the geometric model of the right part of \Cref{fig:d12}.

\begin{figure}[htbp]
\begin{center}
\includegraphics[width=7.5cm]{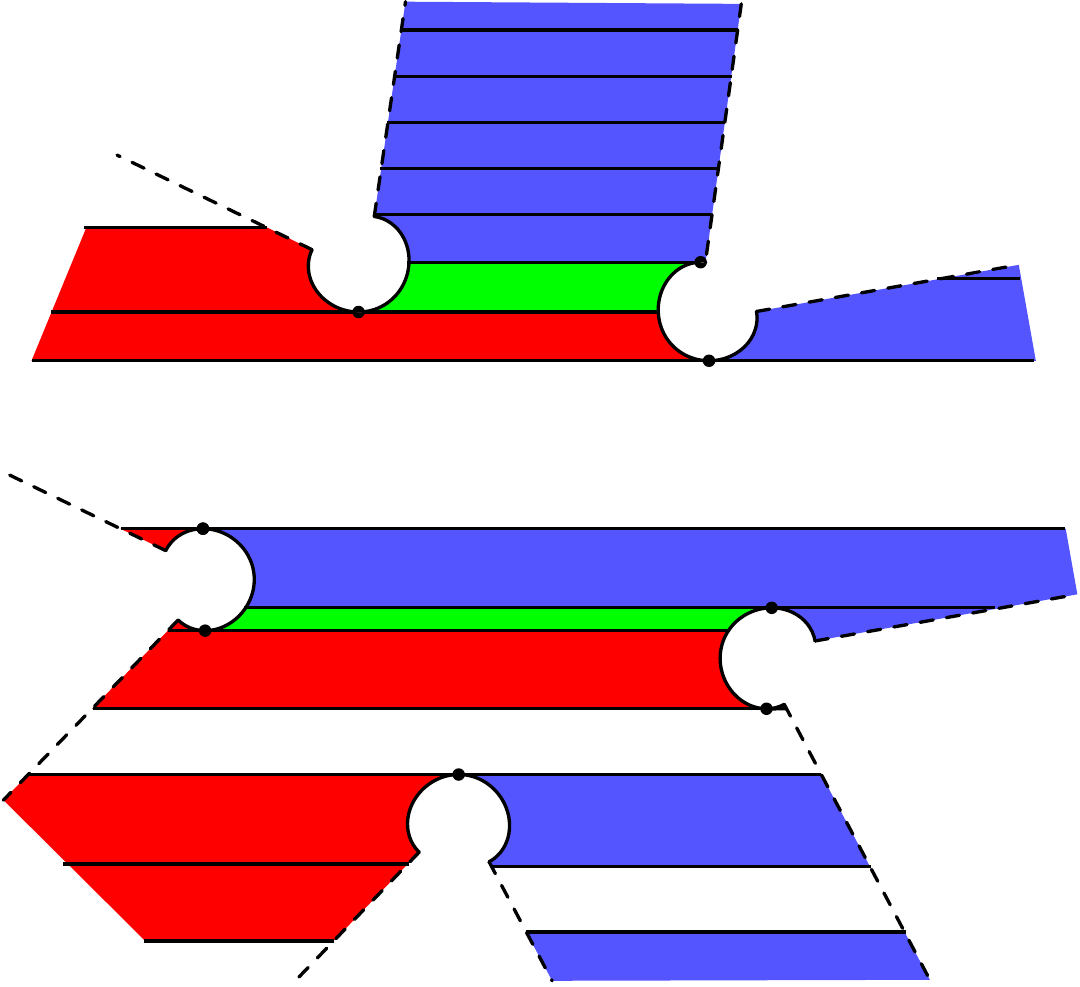}
\caption{Version of the right part of \Cref{fig:d12} in $t$-coordinate. The size of the eyelets constituting the boundaries is exaggerated, otherwise they would be invisible. Tangency points of trajectories (horizontal lines) with the boundary are indicated as black dots. There are no double tangency so this phase portrait is structurally stable, even though there are two groups of boundary connections (in green).}
\label{fig:d16}
\end{center}
\end{figure}

\begin{lemma}\label{lem:dtt}
Provided $r$ and $\rho$ are small enough, then for all $\epsilon\neq 0$ in $B(0,\rho)$ and for any boundary connection $\gamma$, its image in the geometric model is a horizontal segment that starts from some eyelet $C'_{j+1/2}$ and ends on some other eyelet.
The trajectory $\gamma$ is either equal to the boundary of some $U(z_j)$, in which case $z_j$ is neutral ($\mu_j$ is real), or is disjoint from the sets $U(z_j)$.
\end{lemma}

\begin{proof}
The periodic domains are traps in the following sense.
If the singularity is attracting it is indeed a trap in the usual sense: every orbit entering $U(z_j)$ or starting in $U(z_j)$ remains forever there in the future.
If the singularity is repelling, $U(z_j)$ is a trap for the opposite of the vector field (equivalently, traveling backward in time).
If $z_j$ is neutral, then no orbit can cross the boundary of $U(z_j)$.

As a consequence the image of the trajectory is a single segment (it does not cross the half-lines glued by the translations $-\mu_j$). It thus goes from a $C'_{j+1/2}$ to a $C'_{j'+1/2}$ and $j'\neq j$ since $C'_{j+1/2}$ is contained in the boundary of a convex set.

Another key feature that we want to highlight is the following: no trajectory that crosses the boundary of a $U(z_j)$ can both start and end on $\partial B(0,r)$.
The conclusion follows. (See the bottom row of \Cref{fig:asymp:dir}.)
\end{proof}

\begin{corollary}\label{cor:d0}
Consider the two sets that $\gamma$ cuts $B(0,r)$ into. Each of these two sets contains consecutive singularities of the vector field, and at least one.
\end{corollary}

\begin{proof}
Recall that the singularities are circularly ordered.
The first claim follows from the horizontal segment starting and ending on $\bigcup_j C'_{j+1/2}$ without entering the images $T_j$ of $U(z_j)$.\footnote{An alternative proof of the first claim, so as to avoid the use of the geometric model, may go as follows: prove without the model that the periodic regions $U(z_j)$ touch $\partial B(0,r)$ in a unique point $u_j$ and that the circular order of the $u_j$ are the same as the circular order of $j\in \Z/(k+1)\Z$.}

The second claim from the fact that it starts and ends on different $C'_{j+1/2}$.
\end{proof}

\begin{proposition}\label{prop:asymp}
The connection locus $\cal C$ is asymptotic to the set of bifurcation directions.
\end{proposition}
\begin{proof}
The curves $C'_{j+1/2}$ have a uniformly bounded diameter, while the vectors $-\mu_j\sim -\tilde \mu_j$ have a size that tends to $\infty$ when $\eps$ tends to $0$. Consider a parameter for which there is a boundary connection.
The image of the boundary connection in rectifying coordinates is a horizontal segment and links two different $C'_{j+1/2}$.
It follows that the ideal period-gon, rescaled to have its vertices of modulus one, has two vertices of imaginary part whose difference tends to $0$ as $\eps\to 0$ while $\eps\in \cal C$.
\end{proof}

Recall that the double tangency locus $\cal{DT}$ is contained in $\cal C$, hence is also asymptotic to the set of bifurcation directions.

\begin{definition}\label{def:limit:ideal}
We associate to a boundary connection $\gamma$  a \emph{limit ideal situation} that consists in an ideal vector field and one of its homoclinic trajectories.
It is defined as follows.
The ideal vector field is $\tilde\omega_{\tilde \eps} = z^{k+1}-\tilde\eps$ for $\left|\tilde\eps\right| = |\eps|$, and $\arg\tilde\eps$ equal to the bifurcation direction.
The image of $\gamma$ in the geometric model connects the eyelet $C_{j_0+1/2}$ to $C_{j_1+1/2}$. The homoclinic trajectory is the one whose image in the star-shaped domain is a horizontal segment between the vertices $v_{j_0+1/2}$ and $v_{j_1+1/2}$ of the ideal period-gon.
See the bottom row of \Cref{fig:asymp:dir}.
\end{definition}

\begin{definition}\label{def:top:bottom}
The image of the set of tangency points $T_\phi(\eps)$ in the geometric models is the set of points on the eyelets where the latter have a horizontal tangent.
Since these points come in two types: (locally) topmost and bottommost, we qualify the tangency points $T_\phi(\eps)$ with the type, \emph{topmost or bottommost}, of their image on the eyelet.
If we label the intermediate axes as in \cref{eq:def:phi:m}, then $T_\phi(\eps)$ is topmost for $m$ even and bottommost for $m$ odd.
\end{definition}

\begin{remark}\label{rem:inv}
  Conversely, to a given limit ideal situation, i.e.\ the choice of a bifurcation direction $\theta_0$ and of a pair of vertices with the same imaginary parts of the corresponding ideal period-gon, there corresponds exactly $4$ pairs $(m_0,m_1)$ of entry and exit indices such that this is the limit situation for parameters in $\cal {DT}_{\phi_{m_0},\phi_{m_1}}$.
  Indeed the limit situation determines a homoclinic connection of the ideal vector field, with an incoming (resp.\ outgoing) direction that is an attracting (resp.\ repelling) direction of $\omega_0$, and for each such direction, there are exactly two intermediate directions to which it is nearest. 
  Note that for each choice of topmost or bottommost type for the entry and exit indices, there is a unique corresponding pair $(m_0,m_1)$.
  See \Cref{fig:asymp:dir:1}.
\end{remark}

\begin{proof}[Proof of \Cref{item:tc0} of \Cref{thm:bifurcation}]
We saw in \Cref{sub:descr:bif} that a double tangency induces a bifurcation as $\arg \eps$ varies.
Let us prove that conversely, a parameter without double tangencies is in the stable region.
If any singularity were a center, then the boundary of its periodic domain would be double tangent. So there is no center.

In the central part of the geometric model, one can trace two horizontal segment starting from each topmost and each bottommost points of eyelets, until they reach the boundary of the central part.
By hypothesis, in absence of double-tangent trajectories, none of these segments will hit another topmost or bottommost point.
When perturbing $\eps$ this picture deforms in a continuous way, and we can conjugate the flows too,  with a reparametrization of time that depends on the lengths of the open segments, which vary with parameter and position.

The initial choice of conjugacy (up to time reparametrization) on the central part imposes the conjugacy on the boundary of the periodic domains.
They have an explicit extension on the rest of the periodic domain to a conjugacy without reparametrization of time: indeed periodic domains are isomorphic to vertical half-infinite strips with the boundaries glued by a non-horizontal translation, on which the constant vector field $dz/dt = 1$ is acting.
\end{proof}

\begin{figure}[htbp]
\begin{tikzpicture}
\node at (0.05,5.95) {\includegraphics[scale=0.5]{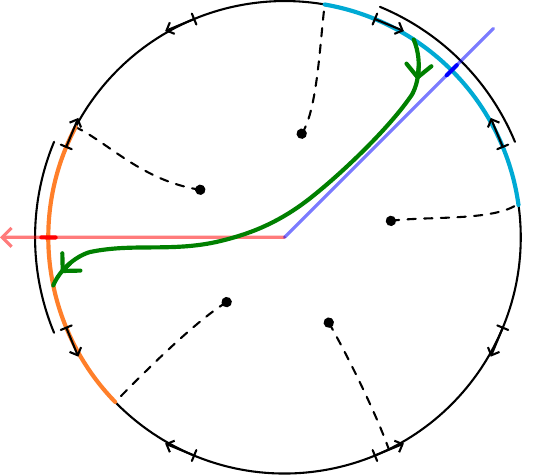}};
\node at (0.9,6) {\footnotesize $1$};
\node at (0.25,6.6) {\footnotesize $2$};
\node at (-0.3,6.3) {\footnotesize $3$};
\node at (-0.1,5.5) {\footnotesize $4$};
\node at (0.5,5.5) {\footnotesize $5$};

\node at (6,6) {\includegraphics[scale=0.5]{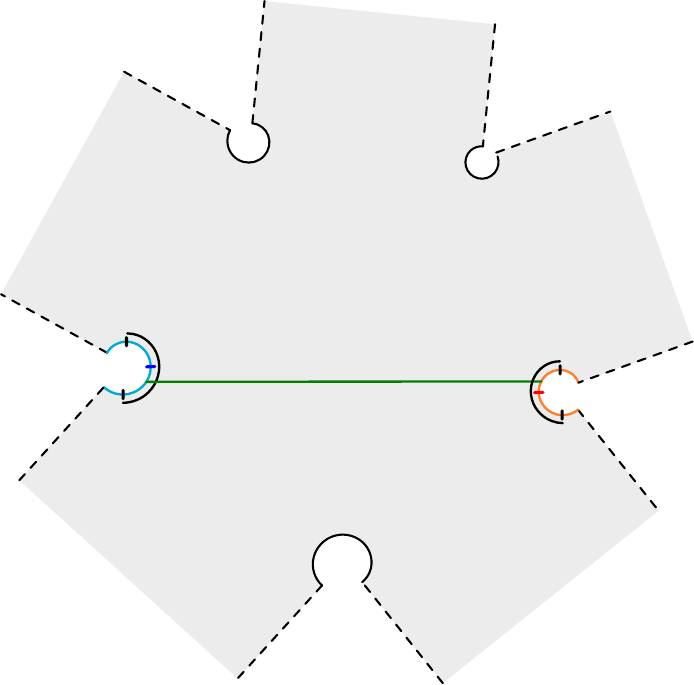}};
\node at (4,7.25) {\footnotesize $1$};
\node at (4.3,4.3) {\footnotesize $2$};
\node at (7.4,4.1) {\footnotesize $3$};
\node at (8.25,6.8) {\footnotesize $4$};
\node at (6.3,8.3) {\footnotesize $5$};

\node at (0,0) {\includegraphics[scale=0.5]{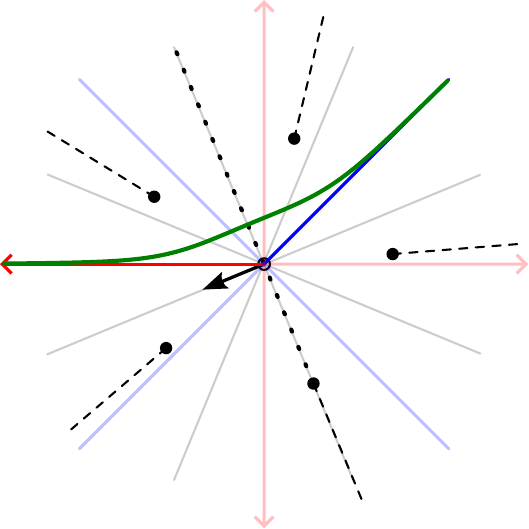}};

\node at (6,0) {\includegraphics[scale=0.5]{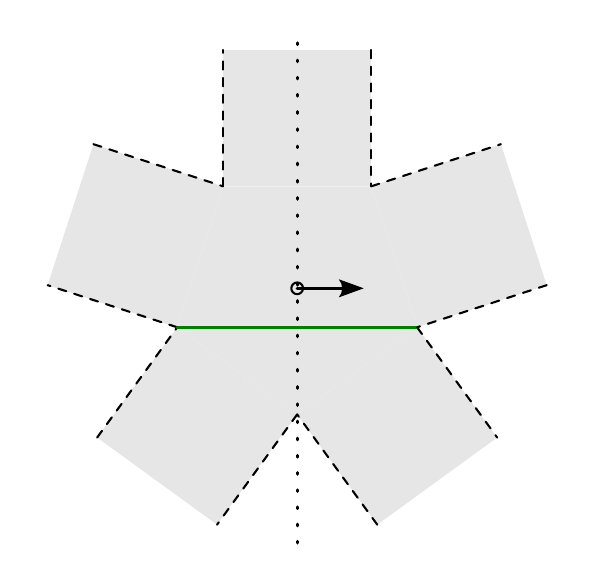}};
\end{tikzpicture}
\caption{
Left column: vector fields.
Right column: their geometric models (a.k.a.\ straightening coordinate, time coordinate).
Top row: sketch of the vector field $\omega_\eps$ with a boundary connection $\tau$ in green; the cyan and orange arcs are respectively the $C_{1+1/2}$ and $C_{3+1/2}$ on the left and $C'_{1+1/2}$ and $C'_{3+1/2}$ on the right.
The entry arc $A_{\phi_a}(\eps)$ and exit arc $R_{\phi_a}(\eps)$ are indicated as arcs of a slightly bigger circle.
Bottom row: limit ideal situation, i.e.\ the ideal vector field $\tilde\omega_\eps=z^{k+1}-\eps$ for $\arg \eps$ equal to the nearby bifurcation direction, the green line is the homoclinic connection associated to $\tau$ and the black arrow gives the direction of the vector field at the origin.
\\
In all frames:
Red arrows are repelling axes and blue segments are attracting ones.
Gray lines are the intermediate axes.
Dashed black lines represent cuts for the straightening coordinate. 
Dotted lines are anti-symmetry axes of the ideal vector field.
}
\label{fig:asymp:dir}
\end{figure}

\begin{figure}[htbp]
\begin{tikzpicture}
\node at (0,6) {\includegraphics[scale=0.5]{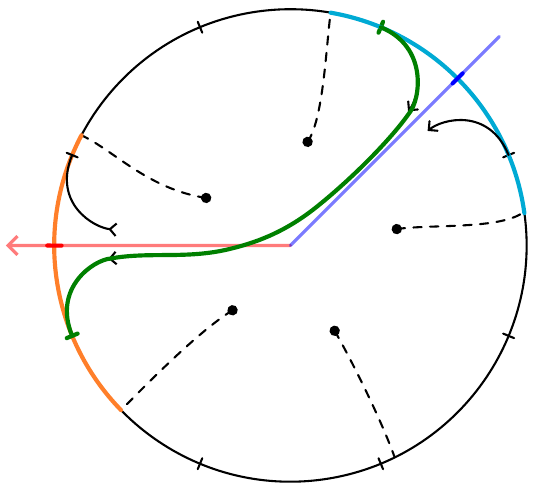}};
\node at (0.9,6) {\footnotesize $1$};
\node at (0.25,6.6) {\footnotesize $2$};
\node at (-0.3,6.3) {\footnotesize $3$};
\node at (-0.1,5.5) {\footnotesize $4$};
\node at (0.5,5.5) {\footnotesize $5$};

\node at (6,6) {\includegraphics[scale=0.5]{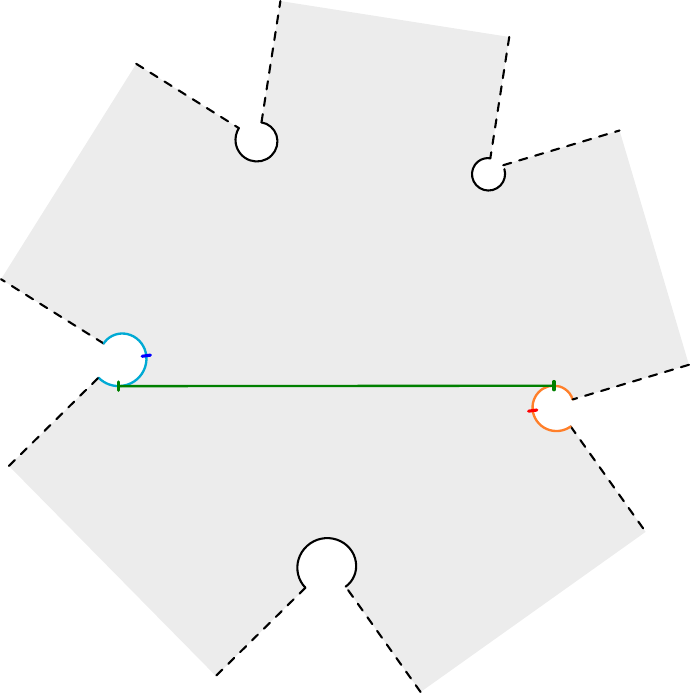}};
\node at (4,7.25) {\footnotesize $1$};
\node at (4.25,4.35) {\footnotesize $2$};
\node at (7.35,4.1) {\footnotesize $3$};
\node at (8.25,6.75) {\footnotesize $4$};
\node at (6.35,8.3) {\footnotesize $5$};
\end{tikzpicture}
\caption{
Version of the top row of \Cref{fig:asymp:dir} in a case with a double tangent trajectory.
}
\label{fig:asymp:dir:1}
\end{figure}

\begin{figure}[htbp]
\begin{tikzpicture}
\node at (0,6) {\includegraphics[scale=0.75]{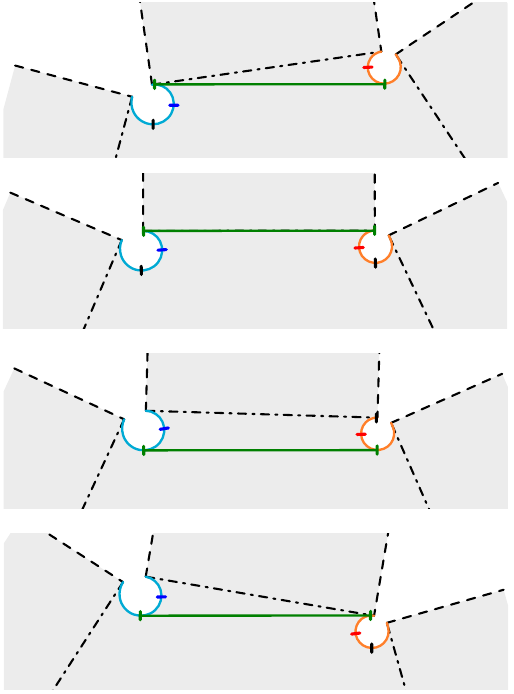}};
\end{tikzpicture}
\caption{
Even if the eyelets are consecutive, the arc from the tangency point to the point mapped to the closest attracting/repelling axis (marked with a blue/red tick) is still contained in a single eyelet.
}
\label{fig:asymp:dir:2}
\end{figure}

In the next statement, we compare the positions of certain features of the vector field with the ideal vector field, and do the same for their geometric models.
\begin{lemma}\label{lem:compare}
For all $\eta>0$, we can ensure by taking $r$ small, then $\rho$ small, that for all $\eps\in B(0,\rho)$:
\begin{enumerate}
\item the arguments of the endpoints of the arcs $C_{j+1/2}\subset \partial B(0,r)$ for $\omega_\eps$ and $\tilde C_{j+1/2} \subset \partial B(0,r)$ for $\tilde \omega_\eps$, differ by less than $\eta$;
\item the angle span $J'_{j+1/2}$ and $\tilde J'_{j+1/2}$ of the outward pointing normal vector to $C'_{j+1/2}$ and to $\tilde C'_{j+1/2}$ are intervals whose endpoints differ by less than $\eta$.
\end{enumerate} 
\end{lemma}
\begin{proof}
 The first claim follows from \Cref{lem:ujvszj} since the endpoints of $C_{j+1/2}$ are $u_{j}$ and $u_{j+1}$ while those of $\tilde C_{j+1/2}$ are $\tilde u_{j}$ and $\tilde u_{j+1}$ and $\arg \tilde u_j = \arg \tilde z_j$.

 The second claim from the fact that the eyelets are images of $C_{j+1/2}$ by the antiderivative of $\omega_\eps$ and $\omega_\eps$ is close to $\omega_0$ which is close to $z^{k+1}$.
\end{proof}

The interval $\tilde J'_{j+1/2} \subset \R/2\pi \Z$ introduced in \Cref{lem:compare} has length $2\pi \frac{k}{k+1}$ and its midpoint is $\pi + \arg v_{j+1/2}(\delta)$, i.e.\ it is antipodal to the corresponding vertex of the ideal period-gon.

\FloatBarrier

\begin{lemma}\label{lem:ihi}
For all $\eta>0$ and $r>0$, we can ensure by taking $\rho$ small that the following holds.
For all $\eps\in B(0,\rho)$ such that $\arg \eps$ is a bifurcation direction, for all horizontal trajectory joining two vertices of the ideal period-gon,
consider the corresponding trajectory $\gamma$ of $\tilde \omega_\eps$.
We recall that $\gamma$ is a homoclinic connection, and that it has an incoming direction $\phi_a$ and an outgoing one $\phi_r$ which are respectively an attracting and a repelling direction of $\tilde\omega_0$.
Then $\gamma$ intersects $\partial B(0,r)$ in two points, whose arguments are $\eta$-close to $\phi_a$ and $\phi_r$.
\end{lemma}
\begin{proof}
We will use the scaling invariance of the vector field (see \cref{rescaling}): let $\eps' = \eps/|\eps| =  e^{i\arg\eps}$.
Let $m(z) = z/|\eps|^{1/(k+1)}$.
Then $m^* \tilde \omega_\eps = |\eps|^{k/(k+1)} \tilde\omega_{\eps'}$ and $B(0,r)$ is transformed into $B(0,r/|\eps|^{1/(k+1)})$.
So, up to a real linear reparametrization of time, $m\circ \gamma$ is a trajectory of $\tilde \omega_{\eps'}$.

For $\tilde \omega_{\eps'}$, a straightening coordinate near $\infty$ is given by an antiderivative of the convergent power series $1/(z^{k+1}-\eps') = \sum_{n\geq 0} \frac{{\eps'}^{n}}{z^{(n+1)(k+1)}}$.
By choosing the integration constant appropriately, one of these antiderivatives $F(z)$ satisfies $F(z) = \frac{-1}{kz^k} + \cal O(1/z^{2k+1})$.
Near $\infty$, separatrices are preimages by $F$ of the segment $(0,x)$ or $(-x,0)$, so are of the form $t\mapsto e^{i\phi} \times (kt)^{-1/k}\times(1+\cal O(t^{(k+1)/k})) = e^{i\phi}k^{-1/k}(t^{-1/k} + \cal O(t))$, where $\phi$ is an attracting or repelling direction of $\tilde\omega_0$, and $t$ is small.

It follows that a separatrix of $\tilde\omega_\eps$ has the repelling/attracting axis as an asymptote.
Coming back from $\tilde\omega_{\eps'}$ to $\tilde\omega_\eps$ via the scaling function $m$, this implies the claim.
\end{proof}

\begin{proof}[Proof of \Cref{item:cx:asymp} of \Cref{thm:connection}]
This proof is illustrated by \Cref{fig:asymp:dir}.
Let $\eps \in \cal{C}_{\phi_a,\phi_r}$ and consider a boundary connection with entry arc $A_{\phi_a}(\eps)$ and exit arc $R_{\phi_r}(\eps)$.
We saw that in the geometric model, this trajectory is mapped to a horizontal segment $S$ between two distinct eyelets $C'_{j_0+1/2}$ and $C'_{j_1+1/2}$.
We also saw that $\arg \eps$ is close to a value $\theta_b$ for which the ideal period-gon has its vertices of indices $j_0+1/2$ and $j_1+1/2$ of the same imaginary part, in particular the imaginary axis is a symmetry axis of this ideal period-gon.

The eyelets $C'_{j+1/2}$ are arcs of nearly circular curves: indeed they are obtained as the image of arcs $C_{j+1/2}$ of $\partial B(0,r)$, of angular span close to $2\pi/(k+1)$ (see \Cref{lem:ujvszj}), under the antiderivative of $z\mapsto 1/\omega_\eps(z)$, which can be assumed close to $1/z^{k+1}$ since we can choose $r$ and $\rho$ small. This antiderivative is close to $z\mapsto -1/kz^k$, up to a translation.
So the eyelet is close to an arc of circle of angular span $2\pi k /( k+1 )= 2\pi - 2\pi/(k+1)$.
It starts at a point where its tangent is parallel to $-\mu_j$ and ends at one where its tangent is parallel to $-\mu_{j+1}$.

It follows that, provided $\eps$ is small enough, the left (resp.\ right) eyelet contains an arc where the outward pointing vector spans a set of argument containing $[-\pi/2+\eta,\pi/2-\eta]$ (resp.\ the antipodal interval).
There is a unique point in the part of $C_{j_0+1/2}$ (resp.\ $C_{j_1+1/2}$) mapped to this arc and whose argument is $\phi_a$ (resp.\ $\phi_r$).
Since the argument span of every $C_{j+1/2}$ is close to $2\pi/(k+1)$ while distinct attracting directions differ by $2\pi/k$, we get that:
\emph{$\phi_a$ ($\phi_r$) is the unique attracting (repelling) direction in the argument span of $C_{j_0+1/2}$ ($C_{j_1+1/2}$)}.

Consider now the limit ideal period-gon, i.e.\ the geometric model of the ideal vector field for $\arg\eps = \theta_b$.
By \Cref{lem:ihi}, the horizontal line joining the vertices of indices $j_0+1/2$ and $j_1+1/2$ in this ideal period-gon corresponds in the $z$-coordinate to a homoclinic connection $\gamma$ for the limit ideal vector field, which intercepts $\partial B(0,r)$ at two points of arguments $\eta$-close to the incoming  and outgoing directions $\hat\phi_a$ and $\hat\phi_r$ of $\gamma$.
Note that these two intercept points belong to the arcs $\tilde C_{j_0+1/2}$ and $\tilde C_{j_1+1/2}$ that get mapped to the eyelets of $\tilde \omega_\eps$ of respective indices $j_0+1/2$ and $j_1+1/2$.
It follows that $\hat\phi_a \in \tilde C_{j_0+1/2}$ and
$\hat\phi_r\in\tilde C_{j_1+1/2}$.
Moreover for $\eps$ small enough there is some definite margin (of $\pi/4k$ for instance) between $\hat\phi_a$ (resp.\ $\hat\phi_r$) and the arguments of the extremities of $\tilde C_{j_0+1/2}$ (resp.\ of $\tilde C_{j_1+1/2}$).
By \Cref{lem:compare}, which compares the eyelets for $\omega_\eps$ and $\tilde \omega_\eps$, we get that $\hat\phi_a\in C_{j_0+1/2}$ and $\hat\phi_r\in C_{j_1+1/2}$.
By the aforementioned uniqueness, $\hat\phi_a=\phi_a$ and $\hat\phi_r = \phi_r$.
We saw in \Cref{prop:hl:pairs:dirs} that $\theta_b = \phi_{\mathrm{mid}} + \frac{\pi}{2}$.
\end{proof}

\begin{proof}[Proof of \Cref{prop:CDT}]
We recall that the $2k$ tangency points $T_\phi(\eps)$, are indexed by the intermediate directions $\phi$, are close to $re^{i\phi}$ by \Cref{lem:tgcy:pts}, and cut $\partial B(0,r)$ into $k$ entry arcs $A_{\phi_a}(\eps)$, with $\phi_a$ close to the midpoint of the arc, and $k$ exit arcs
$R_{\phi_r}(\eps)$, with $\phi_r$ close to its midpoint.

So the tangency point $T_\phi(\eps)$ belongs to a unique entry arc $A_{\phi_a}(\eps)$ and $\phi_a$ is the closest attracting direction to $\phi$.
Similarly the tangency point $T_{\phi'}(\eps)$ belongs to a unique exit arc $R_{\phi_r}(\eps)$ and $\phi_r$ is the closest attracting direction to $\phi$.
\end{proof}

\Cref{item:tc2} of \Cref{thm:bifurcation} is an immediate consequence of \Cref{prop:CDT} and \Cref{item:cx:asymp} of \Cref{thm:connection}.

\begin{remark}\label{rem:meso}
Though not needed for the rest of the article, we find interesting to give the following interpretation of the figures, illustrated in \Cref{fig:meso}.
When $\eps$ is small, the initial part $A$ of a bi-tangent trajectory is close to a tangent trajectory of $\omega_0$, which is itself not too far from the very simple vector field $dz/dt = z^{k+1}$ on $B(0,r)$ if $r$ has initially been chosen small.
A straightening (time) coordinate for this vector field is $t = -1/kz^k$ and the trajectory is well approximated by the image of a straight line (not passing through $0$) by $t\mapsto (-kt)^{1/k}$.
The trajectory plunges towards $0$ and at a mesoscopic scale $B$, it is indistiguishable from an attracting axis.
In the microscopic phase $C$, the trajectory reaches a distance to $0$ comparable to the absolute value $|\eps|^{1/(k+1)}$ of the singularities and their influence deviates the trajectory.
Being double tangent, the trajectory then goes through a similar regime $B$, this time along a repelling axis, and finally a macroscopic phase $A$ to a tangency point.
\end{remark}

\begin{figure}
\begin{tikzpicture}
\node at (0,-1) {\includegraphics[scale = 0.4]{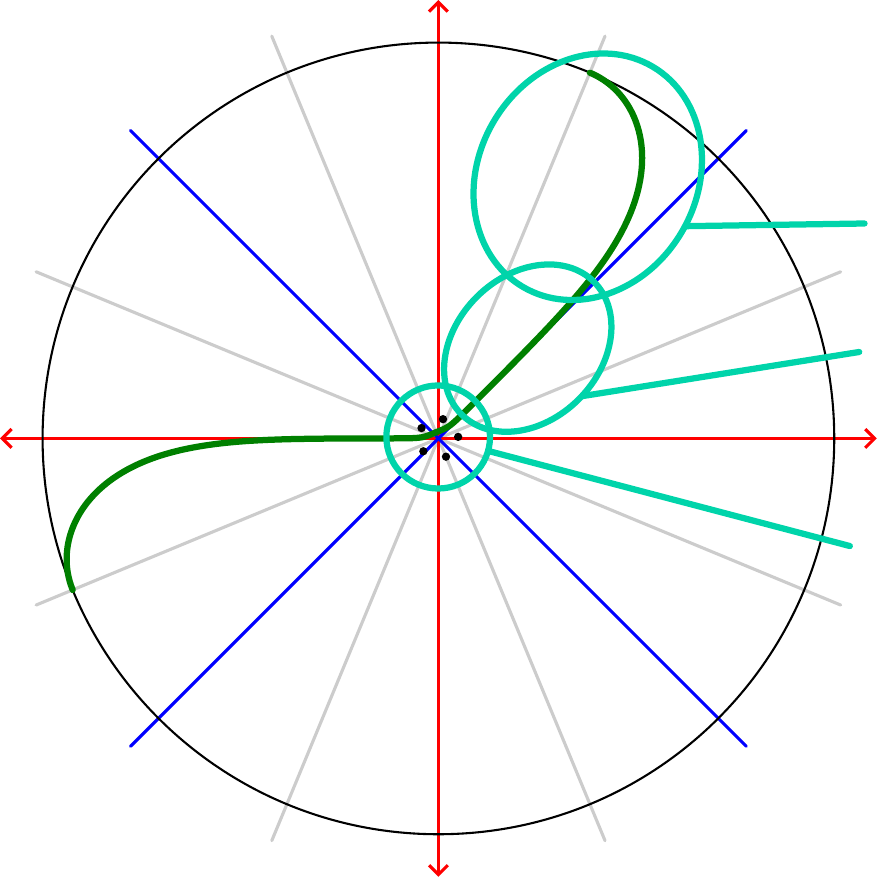}};
\node at (3.2,0.5) {A};
\node at (3.2,-0.4) {B};
\node at (3.1,-1.8) {C};

\node at (-4,-6) {\includegraphics[width = 3cm]{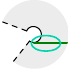}};
\node at (-4,-8) {A};
\node at (-0.5,-6) {\includegraphics[width = 3cm]{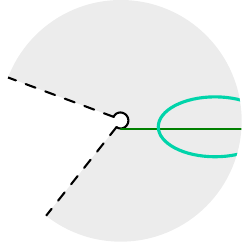}};
\node at (-0.5,-8) {B};
\node at (4,-6) {\includegraphics[width = 4.5cm]{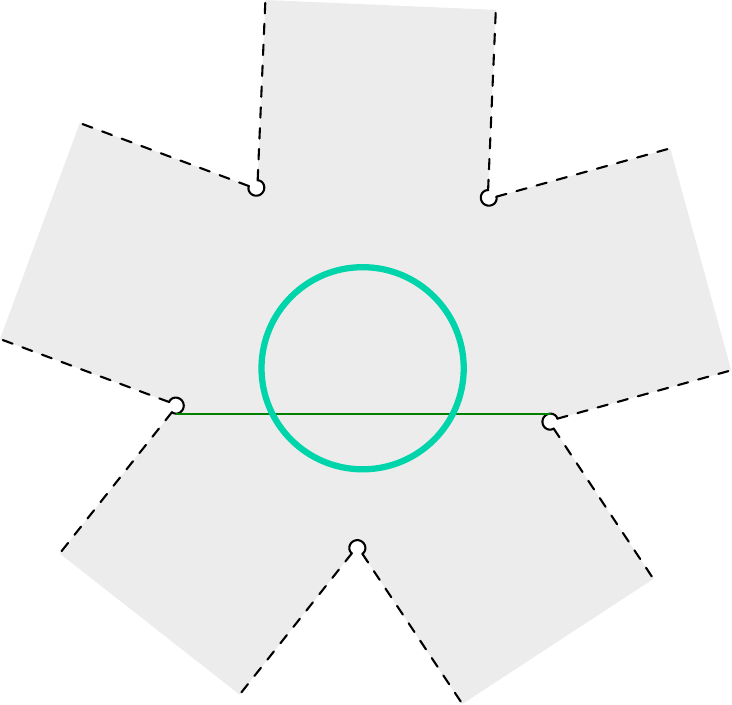}};
\node at (4,-8) {C};
\end{tikzpicture}
\caption{Illustration of \Cref{rem:meso}.
A double tangent trajectory, when $\eps$ is small, can be decomposed into 3 scales. A: macroscopic, B: mesoscopic, C: microscopic.
The corresponding parts of the trajectory are circled, both in the original coordinate and in the geometric model.}
\label{fig:meso}
\end{figure}

\subsubsection{Asymptotics in parameter space for a trajectory with a double tangency of given type}\label{subsub:d}

We now turn to the analytic aspect of the bifurcation sets.
Here we introduce analytic equations and prove that the bifurcation sets are contained in them.
The proof of the converse inclusion is given in \Cref{subsub:0}.

We recall here an important point: the singularities $z_j$ of $\omega_\eps$ and $\tilde z_j$ of $\tilde \omega_\eps$ are circularly indexed by $j\in\Z/(k+1)\Z$ but there is no way to choose the indexing in a continuous way as a function of $\eps\in B(0,\rho)$ because when $\eps$ winds once around $0$, then the indexing has shifted by $-1$.
A way to solve this is to take a $(k+1)$-fold covering of the set of values of $\eps$: we set $\eps = \delta^{k+1}$ and take $\delta\in B(0,\rho^{1/(k+1)})$ as a new parameter.
Then we can set
\begin{equation}\label{eq:delta}
\tilde z_j = \xi^{j} \delta
\end{equation}
with $\xi = \exp(2\pi i/(k+1))$.
Similarly, we can take $z_j=g(\tilde z_j)$ for some holomorphic function $g$ independent of $j$ satisfying $g(0)=0$ and $g'(0)=1$, see \Cref{prop:zeroes}.
As usual, we let
\[\mu_j = \frac{2\pi i}{\omega_\eps'(z_j)},\ 
\tilde \mu_j = \frac{2\pi i}{\tilde \omega_\eps'(\tilde z_j)} = \frac{2\pi i}{(k+1)\tilde z_j^k} = \frac{2\pi i}{(k+1)\eps} \tilde z_j.\]
The vertices $v$ of the ideal period-gon are separated by the vectors $-\tilde \mu_j$.
We label them using $j+1/2$, so that
\[ v_{j+1/2} - v_{j-1/2} = -\tilde \mu_j
.\] 
Then, since we center the period-gon at the origin:
\[ v_{j-1/2} = c' \xi^j \delta^{-k}
\]
with
\[c' = \frac{2\pi i}{ (k+1) (1-\xi)}
.\]
Note that the eyelet labeled $C'_{j+1/2}$ is the one closest to $v_{j+1/2}$.

\begin{convention}\label{conv:num}
In the sequel, we take the parameter $\delta$, we let $\eps=\delta^{k+1}$ and we take $z_j$, $\tilde z_j$, $\mu_j$, $\tilde \mu_j$, $v_{j+1/2}$ and $C'_{j+1/2}$ as above.
\end{convention}

Let $\phi$ and $\phi'$ be intermediate axes directions.
Consider $\eps \in \cal {DT}_{\phi,\phi'}$ and a trajectory $\tau$ of $\omega_\eps$ which has a double tangency to $\partial B(0,r)$ from $T_\phi(\eps)$ to $T_{\phi'}(\eps)$.
We saw that if $\eps$ is small then its argument is close to a bifurcation direction $\theta_0$ which we identified, and that the image of $\tau$ in the geometric model is a horizontal segment $H$ linking an eyelet $C'_{j_0+1/2}$ to another $C'_{j_1+1/2}$ and tangent to each, as in \Cref{essential_tangencies}.

\begin{figure}
\begin{center}
\subfigure{\includegraphics[width=5.5cm,angle=180]{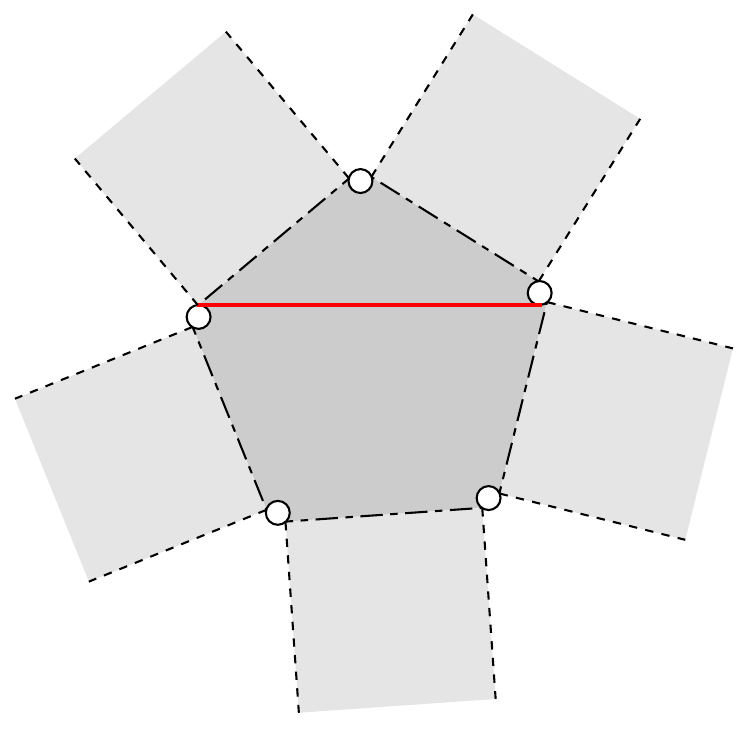}}
\qquad
\subfigure{\includegraphics[width=5.5cm,angle=180]{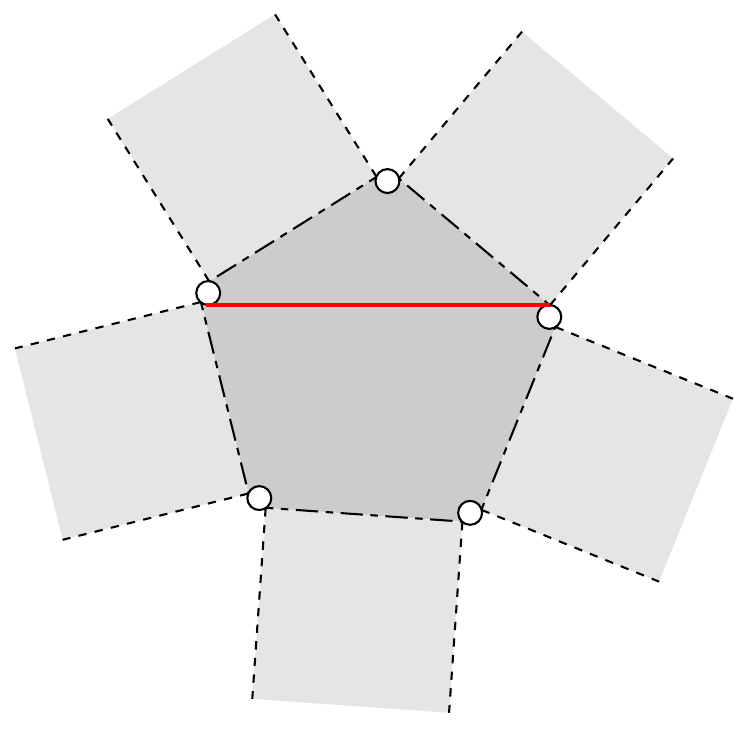}}
\caption{The two essential double tangency bifurcations for parameters on the boundary of a region $\mathcal{C}_{2j,2\ell+1}$.}
\label{essential_tangencies}
\end{center}
\end{figure}

\begin{convention}\label{conv:delta:p}
We choose, among the $k+1$ roots $\delta$ of $\eps$, the one such that
\[j_0=-1,\]
so the horizontal segment $H$ goes from $C'_{-1/2}$ to $C'_{p-1/2}$ for some integer $p$ with
\[0<p<k+1\]
(using the numbering \Cref{conv:num}). 
Note that $\arg \delta$ is close to a value $\alpha_0 \in\R/2\pi \Z$ such that $(k+1)\alpha_0 \equiv \theta_0 \bmod 2\pi \Z$.
This $\alpha_0$ is characterized as the one so that the ideal period-gon vertex $v_{-1/2}$ is the left end of the homoclinic connection of the associated limit ideal situation (see \Cref{def:limit:ideal}).
\end{convention}

If the two tangency points of $\tau$ are equal, then we saw that $\tau$ is the boundary of a periodic region for a singularity which is a center, i.e.\ whose eigenvalue is purely imaginary.
This condition is analytic and takes the form
\[\Re(\mu_0) = 0\]
if the ideal vertex $v_{-1/2}$ is at the bottom of the period-gon (then the tangency point is of bottom type), or
\[\Re(\mu_{-1}) = 0\]
if it is at the top (then the tangency point is of top type).
Conversely if this condition is satisfied and if the argument of $\delta$ is close to $\alpha_0$, then the boundary of the periodic region of the center is a double tangent trajectory of type $\cal {DT}_{\phi,\phi}$.

If the two tangency points of $\tau$ are disjoint, let us close the path $\tau$ into a simple loop by following the arc $C_\eps$ of $\partial B(0,r)$ that goes \emph{clockwise} from the endpoint $T_{\phi'}(\eps)$ to the starting point $T_{\phi}(\eps)$ of $\tau$.
The integral of $dz/\omega_\eps(z)$ along this closed path is equal to the sum of negated periods $\sum -\mu_j$ on the set $J\subset \Z/(k+1)\Z$ of indices $j$ of the singuarities $z_j$ enclosed by the simple loop.
Hence (since the integral of the holomorphic $1$-form $dz/\omega_\eps(z)$ along any trajectory is a real number)
\begin{equation}\label{eq:1}
\im \left(\int_{C_\eps} \frac{dz}{\omega_\eps(z)}\right) + \im \left (\sum_{j\in J} \mu_j \right)=0.
\end{equation}
Recall that $r$ is fixed to a small value and that both $C_\eps$ and $\mu_j$ depend on $\delta$.

\medskip

By \Cref{cor:d0}, the singularities enclosed by the loop we defined above are consecutive and since following $C_\eps$ clockwise implies, in the model, that each curve $C'_{j+1/2}$ is followed anticlockwise, we deduce that the indices of the enclosed singularities are of the form
\[J = \{0,1,\ldots,p-1\}.\]
See \Cref{fig:0p}.

\begin{figure}[htbp]
\begin{tikzpicture}
\node at (0,0) {\includegraphics*[scale=0.6]{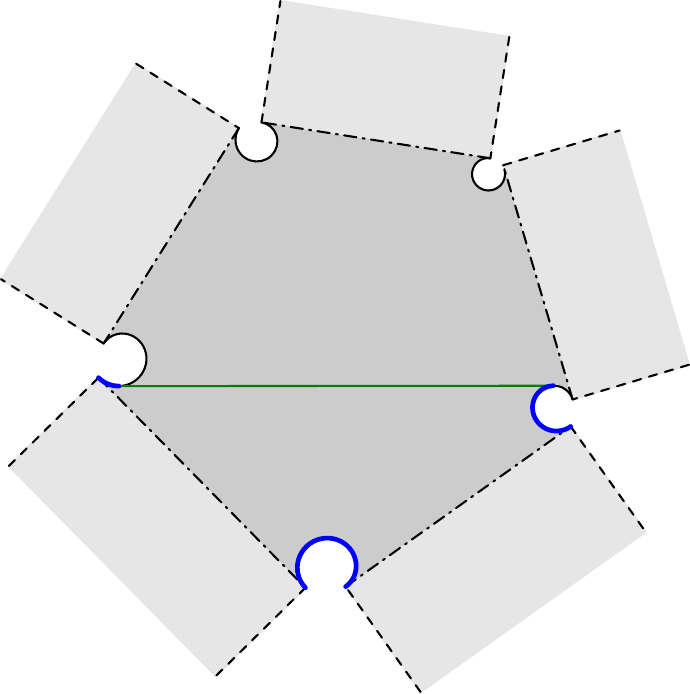}};
\node at (-2.6,-2.6) {$T_0$};
\node at (2.2,-3) {$T_1$};
\node at (3.45,1.05) {$T_2$};
\node at (0.6,3.6) {$T_3$};
\node at (-3.05,2) {$T_4$};
\node at (6.5,0) {\includegraphics*[scale=0.55]{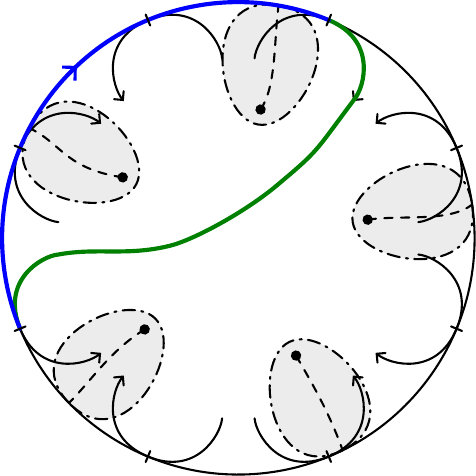}};
\node at (4.8,2.1) {$C_\eps$};
\node at (6.7,-0.1) {$\tau$};
\node at (6.5,0.9) {$z_0$};
\node(A) at (7,3) {$U(z_0)$};
\draw (A) -- (6.6,1.9);
\node at (5.8,0.45) {$z_1$};
\end{tikzpicture}
\caption{Enclosed singularities, $k+1=5$, $p=2$. Left: in the geometric model ($t$-space). Right: sketch in the $z$-space; we indicated germs of trajectories through the $8$ tangency points.}
\label{fig:0p}
\end{figure}

There are only finitely many possible equations of type \cref{eq:1}, according to:
\begin{enumerate}
\item the value of $p$ in $\{1,\ldots,k\}$,
\item which of the $2k$ potential tangency points on $\partial B(0,r)$ the trajectory $\tau$ joins.
\end{enumerate}
(All ordered pairs of tangency points on $\partial B(0,r)$ are possible and this determines uniquely the trajectory $\tau $ and the value of $p$ in (1), but we do not need to discuss these details at this point of the proof).
For each choice, the first term in \cref{eq:1} is a function of $\eps$ that is defined and real analytic in a domain that contains $\eps = 0$. In particular it is bounded and differentiable.

Recall that $\tilde z_{j} = \xi^j \delta$ with $\xi = \exp(2\pi i/(k+1))$, see \eqref{eq:delta}.
Each $\mu_j$ is close to $\tilde \mu_j =  \frac{2\pi i}{(k+1) \tilde z_j^k} = \frac{2\pi i}{(k+1)} \xi^j \delta^{-k} $.

\begin{lemma}\label{lem:d2}
\[\sum_{j=0}^{p-1} \mu_j = \frac{h_p(\delta)}{\delta^k}\]
where $h_p$ is holomorphic and $h_p(0)\neq 0$.
\end{lemma}
\begin{proof}
By \Cref{prop:eig},
\[ \lambda_j := \omega_\eps'(z_j) = (k+1) \tilde z_j^k + \tilde z_j^{k+1} s(\tilde z_j)\]
for some holomorphic function $s$ independent of $j$ (using $\eps = \tilde z_j^{k+1}$), and since
$\mu_j = 2\pi i / \lambda_j$ we get
\[\mu_j = \frac{2\pi i}{(k+1)\tilde z_j^k} g(\tilde z_j)\]
for some holomorphic function $g$ independent of $j$ that satisfies
\[g(0)=1.\] 
Now 
$\tilde z_j = \xi^j \delta$ so
\[\sum_{j=0}^{p-1} \mu_j = \frac{2\pi i}{(k+1) \delta^k}\sum_{j=0}^{p-1} \frac{g(\xi^j \delta)}{\xi^{jk}}.\]
and using $\xi^{jk} = 1/\xi^j$:
\[\sum_{j=0}^{p-1} \mu_j = \frac{2\pi i}{(k+1) \delta^k}\sum_{j=0}^{p-1} \xi^j g(\xi^j \delta).\]
So the claim holds with
\[h_p(\delta) := \frac{2\pi i}{(k+1)}\sum_{j=0}^{p-1}\xi^j g(\xi^j \delta),\]
for which
\[h_p(0) = \frac{2\pi i}{(k+1)}\sum_{j=0}^{p-1}\xi^j.\]
We have $h_p(0)\neq 0$ since $\sum_{j=0}^{p-1} \xi^j = \frac{1-\xi^{p}}{1-\xi}$, and $0<p<k+1$.
\end{proof}

The following statement will be used in the proof of \Cref{prop:ut} in \Cref{subsub:0}.

\begin{lemma}\label{lem:h0}
For the ideal period-gon $\tilde P$, we have
\[\sum_{j=0}^{p-1} \tilde\mu_j = \frac{h_p(0)}{\delta^k}\]
for the same function $h_p$ as in the previous lemma.
\end{lemma}
\begin{proof} In the previous proof we introduced a function $g$.
The corresponding function $g$ for the ideal vector field is constant, hence the corresponding $h_p$ is constant.
The value of $h_p$ at the origin was computed in the previous proof, so coincides for the families $\omega_\eps$ and $\tilde \omega_\eps$.
\end{proof}

\begin{lemma}
The set $\cal {DT}_{\phi,\phi'}$ is contained in a curve satisfying \Cref{item:tc4} of \Cref{thm:bifurcation}.
\end{lemma}
\begin{proof}
By \cref{eq:1} and \Cref{lem:d2}, points $\eps\in \cal {DT}_{\phi,\phi'}$ satisfy
\[ \im \left ( \frac{h_p(\delta)}{\delta^k} \right) = - \im \left(\int_{C_\eps} \frac{dz}{\omega_\eps(z)}\right) \]
where $\delta$ is an appropriate $(k+1)$-th root of $\eps$, see \Cref{conv:delta:p}.
The term on the right depends analytically on $\eps$ and remains bounded in $B(0,\rho)$.
According to \Cref{conv:delta:p}, the argument of $\delta$ is close to some $\alpha_0$
such that the term $\frac{h_p(\delta)}{\delta^k} = \sum_{j=0}^{p-1} \mu_j$ is close to horizontal, in particular $h_p(0)/e^{ik\alpha_0}$ is real and positive.
Let $\delta = s e^{i(\alpha_0+a)}$. Multiplying the equation by $s^{k}$ we get
\[ \im \left ( f(\delta) e^{-ika} \right) + s^k g(\eps) = 0\]
with $f,g$ analytic and $f(0)$ real positive.
For $s=0$, the value $a=0$ is a solution, and
we conclude using the implicit function theorem, since the differential of the equation w.r.t.\ $a$ at $(s,a)=(0,0)$ does not vanish: it is equal to $-k f(0)$.
\end{proof}

For fixed intermediate axes directions, $\phi$ and $\phi'$, the integration arc $C_\eps$ along $\partial B(0,r)$, which  goes from $T_{\phi'}(\eps)$ to $T_\phi(\eps)$ \emph{clockwise}, depends continuously on $\eps$.
Let
\[T_0 = T_{\phi}(0)\text{ and }T_1 = T_{\phi'}(0).\]

Let us write that $\omega_\eps(z) \approx z^{k+1}$ on $\partial B(0,r)$, where $\approx$ means that by reducing $r$ (then $\rho$) we can ensure that the quotient is arbitrarily close to $1$. 
Then on the one hand, by the tangency of the vector field with $\partial B(0,r)$ at $T_0$ and at $T_1$, we get that
$\omega_0(T_0) \in i T_0\R$ and $\omega_0(T_1) \in i T_1\R$,
whence
\[T_0^{k} \approx \pm i r^k\ \text{ and } T_1^{k} \approx \pm i r^k.\]
On the other hand $\forall \eps\in B(0,\rho)$:
\[\int_{C_\eps} \frac{dz}{\omega_\eps(z)} \approx \frac{-1}{k}\left(\frac{1}{T_1^{k}}-\frac{1}{T_0^{k}}\right).\]
By the above two estimates, it follows that 
\[\Im \int_{C_\eps} \frac{dz}{\omega_\eps(z)} \approx \frac{(\pm 1) +(\pm 1) }{kr^k}.\]

Let us use the labelling $\phi_m$ of the intermediate axes of $\omega_0$ introduced in \cref{eq:def:phi:m}.
In \Cref{def:top:bottom} we introduced the notion of topmost and bottomost types for the tangency points $T_{\phi_m}(\eps)$ when $\phi_m$ is an intermediate axis of $\omega_0$.
In only depends on the parity of $m\in\{1,\ldots 2k\}$, not on $\eps$.
In \Cref{sub:descr:bif}, we called \emph{essential} the case of a double tangent trajectory that reaches two tangency points $T_{\phi_{m_0}}(\eps)$ and $T_{\phi_{m_1}}(\eps)$ having different tangency sides.
In this case:
\[
\im \int_{C_\eps} \frac{dz}{\omega_\eps(z)} \approx \frac{\pm 2}{kr^k}
.\]
(This is actually valid for all $\eps$ provided $m_0$ and $m_1$ have different parity, even in the absence of a double tangency.)

If we are it the other case, where $m_0$ and $m_1$ have the same parity, then $\frac{1}{T_1^{k}}-\frac{1}{T_0^{k}}$ may compensate and depending on the situation, it may be hard to tell the order, or the disjointness, of the bifurcation arcs in this case. All we can do is to ensure that for any fixed $\eta>0$ in advance, then provided $r$ and $\rho$ are small enough we have
\[
\left|\int_{C_\eps} \frac{dz}{\omega_\eps(z)}\right| \leq \frac{\eta}{kr^k}
.\]

In any case, let
\begin{equation} \label{eq:def:I0}
I_0 = I_0[m_0,m_1] = \im \int_{C_{0}} \frac{dz}{\omega_0(z)}
\end{equation}
where $C_0$ is the limit as $\eps\to 0$ of the arc $C_\eps$, i.e.\ it is the arc that goes clockwise from $T_{\phi'}(0)$ to $T_\phi(0)$.

By \cref{eq:1} and \Cref{lem:d2}:

\begin{lemma}
For any parameter $\eps$ with a double tangent trajectory $\tau$,
\begin{equation} \label{eq:b2e}
\im\frac{h_p(\delta)}{\delta^k} = -I_0[m_0,m_1] + O(\eps)
\end{equation}
where $\delta$ is the $(k+1)$-th root of $\eps$ chosen in \Cref{conv:delta:p} (where we also defined the value of $p$). Moreover $h_p$ is the function of \Cref{lem:d2} and $\phi_{m_0},\phi_{m_1}$ are the entry and exit indices of $\tau$.
\end{lemma}

Consider the map
\begin{equation}\label{eq:def:p}
q(\delta) = h(\delta)/\delta^k
\end{equation}
as taking values in the Riemann sphere.
It sends $0$ to $\infty$ with a critical point of local degree $k$.
Let us denote the functions $h$ and $q$ as $h_{p}$ and $q_{p}$, to stress the dependence of $h$ on $p$, and
\[\cal {DT}'[m_0,m_1]\]
the set of $\delta\in B(0,\rho^{1/(k+1)})$ such that $|\arg(\delta)-\alpha_0|<\pi/k$ ($\alpha_0$ is defined just after \Cref{conv:delta:p}) and such that for $\eps = \delta^{k+1}$, $\omega_\eps$ has a double tangent trajectory from $T_{\phi_{m_0}}$ to $T_{\phi_{m_1}}$ (i.e.\ $\eps \in \cal{DT}_{\phi_{m_0},\phi_{m_1}}$).
By \cref{eq:b2e}:

\begin{corollary}
  The image of $\cal {DT}'[m_0,m_1]$
  under $q_{p}$ is contained in a curve asymptotic to the horizontal line of imaginary part $-I_0[m_0,m_1]$.
\end{corollary}
\noindent By asymptotic, we mean that $\forall \eta>0$, $\exists M>0$ such that for every point $w$ in this image such that $|\Re(w)|>M$ then $|\Im(w)+I_0[m_0,m_1]|<\eta$.
So the above does not imply that the image is a curve or even non-empty: this is the task of \Cref{subsub:0}.

\subsubsection{Finding trajectories with a double tangency of a given type}\label{subsub:0}

Consider two intermediate directions $\phi=\phi_{m_0}$ and $\phi'=\phi_{m_1}$.
As in the previous section, we use the variable $\delta$ with $\eps = \delta^{k+1}$, the indexing of \Cref{conv:num} and let $\alpha = \arg \delta$.
We let $\alpha_0$ and $p$ be defined as in \Cref{conv:delta:p}.
Recall $C_\eps$ denotes the arc of circle from $T_{\phi'}(\eps)$ to $T_{\phi}(\eps)$ in clockwise order.
The bifurcation directions split the circle of all possible arguments for $\eps$ into $K=2k$ equal intervals.
Let $K'=(k+1)K$.

\begin{theorem}\label{thm:erdt}
Assume $s$ small enough, $\alpha\in [\alpha_0-1/10K',\alpha_0+1/10K']$ and $s'\in(0,s)$. Let $\delta = s' e^{2\pi i\alpha}$ and $\eps=\delta^{k+1}$.
Then $\eps \in \cal{DT}_{\phi,\phi'}$ if and only if 
\begin{equation}\label{eq:2}
\im \left(\int_{C_\eps} \frac{dz}{\omega_\eps(z)}\right) + \im \left (\sum_{j=0}^{p-1} \mu_j(\delta) \right)= 0
.
\end{equation}
\end{theorem}

\begin{proof}
The direct implication follows from \Cref{eq:1}.

For the converse, assume that $|\alpha- \alpha_0|<1/10K'$.
We saw that a double tangent trajectory from $T_{\phi}$ to $T_{\phi'}$ maps in the geometric model to a horizontal segment that tangentially links $C'_{0-1/2}$ to $C'_{p-1/2}$ with our numbering conventions.

As in the proof of \Cref{eq:1}, we use that the straightening map from the phase space $B(0,r)$, with $k+1$ (not necessarily straight) slits, to the geometric model, splits $\partial B(0,r)$ into $k+1$ \emph{closed} arcs $C_{j-1/2}$ running anticlockwise from $u_{j-1}$ to $u_j$ where $u_j\in\partial B(0,r)$ are points satisfying, by \Cref{lem:ujvszj}, $|\arg(u_j/\tilde z_j)| < \eta$ and, by \Cref{lem:tgcy:pts}, $|\arg(T_{\phi_m}(\eps))-\phi_m|<\eta$, with $\eta$ as small as we like.

The argument span of $C_{j-1/2}$ is $\eta$-close to the interval \[I_j(\alpha) = (\arg(\tilde z_{j-1}(\delta)),\arg(\tilde z_{j}(\delta))) \subset \R/2\pi \Z
\]
which differs from $I_j(\alpha_0)$ by a translation of $\alpha-\alpha_0$ in $\R/2\pi\Z$.
We denote by $\psi_j$ the angles $\psi_j$ bounding $I_j(\alpha_0)$, so that
\[I_j(\alpha_0) = (\psi_{j-1},\psi_j)
.\]
The curve $C'_{j-1/2}$ is arbitrarily close to an arc of circle. Its argument span is the translation by $-k\times(\alpha-\alpha_0)$ of
\[ J_j(\alpha_0) := \left[ \Big( \frac{j}{k+1}-\frac{1}{4} \Big) 2\pi, \Big( \frac{j}{k+1}-\frac{1}{4}+\frac{k}{k+1} \Big) 2\pi \right]
\]
in $\R/2\pi \Z$.
The interval $J_j(\alpha_0)$ has length $2\pi \frac{k}{k+1}$ and its center is antipodal to the argument of the vertex $v_{j-1/2}(\delta)$ of the ideal period-gon.

Denote $\phi = \phi_{m_0}$ and $\phi' = \phi_{m_1}$.
We may or may not have $\phi_{m_0}\in \partial I_0(\alpha_0)$. This happens if and only if $p=1$ or $p=k$, i.e.\ the eyelet $\tilde C'_{j_0+1/2}$ is respectively bottommost or topmost, and the initial tangency point of the trajectory in the model is the bottommos/topmost of $C'_{j_0+1/2}$.
A similar statement holds for the condition $\phi_{m_1}\in \partial I_0(\alpha_0)$.

We first deal with the case where both angles $\phi_{m_0}$ and $\phi_{m_1}$ respectively belong to the interior of $I_0(\alpha_0)$ and $I_p(\alpha_0)$. Their distance to $\partial I_0(\alpha_0)$ (resp.\ $\partial I_p(\alpha_0)$) is then at least $\frac{2\pi}{2k(k+1)}>\frac{2}{10K'}$.
Hence for $|\alpha-\alpha_0|<1/10K'$ and $\eta$ initially chosen small enough, $\arg T_{\phi_{m_0}}(\eps)\in I_0(\alpha)$ and $\arg T_{\phi_{m_1}}(\eps)\in I_p(\alpha)$.
The images $c(\eps)$ and $c'(\eps)$  of $T_{\phi_{m_0}}(\eps)$ and $T_{\phi_{m_1}}(\eps)$ in the geometric model belong respectively to $C'_{0-1/2}$ and $C'_{p-1/2}$.
Choose any $\delta_0$ such that $\arg\delta_0 = \alpha_0$ (its modulus does not matter) and let $v'_{j-1/2}=v_{j-1/2}(\delta_0)$.
Since $\alpha_0$ is a distinguished bifurcation direction we have that the arguments of $v'_{0-1/2}$ and $v'_{p-1/2}$ both make an angle of at least $2\pi/2(k+1)$ with the imaginary axis.
Since $|\alpha-\alpha_0|<1/10K'$ then, provided $s$ was initially chosen small enough, the point $c(\eps)$ has negative real part while $c'(\eps)$ has positive real part.
In particular, for the $\eps$ corresponding to such $\alpha$, the arc $C_\eps$ crosses the contact points $u_j(\delta)$ with the same indices $j$ as the angles $\psi_j$ crossed by the arc going clockwise from $\phi_{m_1}$ to $\phi_{m_0}$.
From this and the geometric model, it follows that the left hand side of \Cref{eq:2} represents the difference in imaginary part between $c(\eps)$ and $c'(\eps)$.

If this difference vanishes, then the interior of the straight horizontal segment between these points is contained in the geometric model: these points are topmost or bottommost points on $C'_{0-1/2}$ and $C'_{p-1/2}$, where these nearly circular curves have a horizontal tangent, and 
the other $C'_{j-1/2}$ are too far away (fix $r$ and take $s$ small).
This provides a double tangent trajectory from $T_{\phi_{m_0}}$ to $T_{\phi_{m_1}}$ and proves the theorem.

We now deal with the case where $\phi_{m_0}$ belongs to the boundary of $I_0(\alpha_0)$ while $\phi_{m_1}$ belongs to the interior of $I_p(\alpha_0)$.
The side type of $T_{\phi_{m_0}}(\eps)$ only depends on $m_0$.
If it is of bottom type then $p=1$ otherwise $p=k$. We explain how to treat the bottom type case, the other case being similar.
In this case, $m_1$ is of top type (otherwise, given that $p=1$, we get $m_1=m_0$, so $\phi_{m_1}=\phi_{m_0}$ belongs to the boundary of $I_p(\alpha_0)$, in contradiction with the hypotheses).
The argument of the previous case still applies to $T_{\phi_{m_1}}(\eps)$ and proves that it belongs to $I_p(\alpha)$.
To prove that $T_{\phi_{m_0}}(\eps)$ belongs to $I_0(\alpha)$ we proceed by contradiction: if not, then it is in the closure of $I_p(\alpha)$ ($p=1$); but then $\Im(-\mu_0(\delta))>0$ while $\Im(\int_{C_\eps} dz/\omega_\eps(z)) < 0$ hence \Cref{eq:2} cannot hold. These last two inequalities can be proved by direct estimates or by considering the geometric situation illustrated on \Cref{fig:geo:ill}.
The rest of the proof is then the same.

\begin{figure}[htbp]
\begin{tikzpicture}
\node at (0,0) {\includegraphics[scale=0.5]{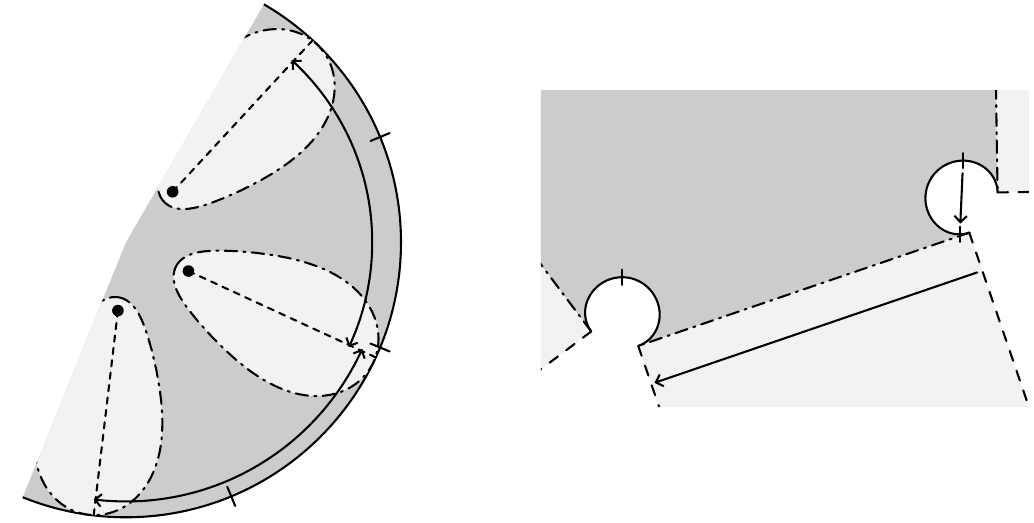}};
\node at (-2.3,-2.4) {\footnotesize $T_{\phi_{m}}(\eps)$};
\node at (-3.05,0.05) {$z_0$};
\node at (-0.8,-1.1) {$u_0(\delta)$};
\node at (-0.48,-0.67) {\footnotesize $T_{\phi_{m_0}}(\eps)$};
\node at (-0.5,1.2) {\footnotesize $T_{\phi_{m_1}}(\eps)$};
\node at (-2.45,-1.55) {$C_{-1/2}$};
\node at (-1.55,0.4) {$C_{1/2}$};

\node at (2.5,-1) {$\mu_0$};
\node at (3.2,0.9) {$C'_{1/2}$};
\node at (1.4,0.1) {$C'_{-1/2}$};
\end{tikzpicture}
\caption{Case where $\phi_{m_0}=\psi_0$ i.e.\ $\phi_{m_0}=\arg(u_0(\delta_0))$.
Then $T_{\phi_{m_0}}(\eps)$ is close to $u_0(\delta)$. We illustrate here what happens if $T_{\phi_{m_0}}(\eps)$ does not belong to $C_{-1/2}$.
Since $\mu_0$ is tangent to $\partial C'_{1/2}$ at the low endpoint of $C'_{1/2}$, its imaginary part is negative. The small arrow circled by $C'_{1/2}$ is a vector representing the integral of $dz/\omega_\eps(z)$ along the clockwise arc from $T_{\phi_{m_1}}(\eps)$ to $T_{\phi_{m_0}}(\eps)$.
The imaginary part of the sum of $\mu_0$ and of this vector is negative, so \cref{eq:2} cannot hold.}
\label{fig:geo:ill}
\end{figure}

The case where $\phi_{m_0}$ belongs to the interior  of $I_0(\alpha_0)$ while $\phi_{m_1}$ belongs to the boundary of $I_p(\alpha_0)$ is treated the same way.

Finally, we assume that $\phi_{m_0}$ belongs to the boundary of $I_0(\alpha_0)$ while $\phi_{m_1}$ belongs to the boundary of $I_p(\alpha_0)$. In this case $m_1=m_0$.
Again there are two cases: either these are bottom-type points and $p=1$, or these are top-type points and $p=k$.
Again we only treat the case $p=1$, the other being similar.
Then the integral part of \cref{eq:2} is null, while the other term is just $\Im(\mu_0)$. Its vanishing means that the periodic region of $z_0$ is bounded by a closed trajectory of $\omega_\eps$, and this loop joins $T_{\phi_{m_0}}(\eps)$ to $T_{\phi_{m_1}}(\eps)$.
\end{proof}

This proves \Cref{item:tc1} of \Cref{thm:bifurcation}.
\Cref{item:tc4} of \Cref{thm:bifurcation} is proved by the second point of the following statement:

\begin{proposition}\label{prop:ut}
Provided $r$ has been chosen sufficiently small, there exists $\rho>0$ such that
for all real $s>0$ with $s^k < \rho)$:
\begin{enumerate}
\item\label{item:ut1} There is a unique $\alpha = \alpha(s) \in [\alpha_0-1/10K',\alpha_0+1/10K']$ such that for $\delta=s e^{2\pi i\alpha}$, \cref{eq:2} holds.
\item\label{ut:2} The solution $\alpha(s)$ depends real analytically on $s$, and this function has a real analytic extension through $0$.
\end{enumerate}
\end{proposition}
\begin{proof}
The term $\int_{C_\eps} \frac{dz}{\omega_\eps(z)}$ is bounded independently of $\alpha$ and $s$, so its imaginary part too.

\Cref{eq:2} depends real-analytically on $\delta$, hence on $(s,\alpha)$ for $s>0$.
\Cref{lem:d2} also applies here, so we can write
$\sum_{j=0}^{p-1} \mu_j(\delta) = h_p(\delta) / \delta^k$ for some holomorphic function $h$ with $h(0)\neq 0$.

The directions $\arg \delta=\alpha_0$ was chosen so as to satisfy the following condition: $\arg(v_{p-1/2}-v_{-1/2}) \equiv 0 \bmod 2\pi$.
This means exactly that $\arg \sum_{j=0}^{p-1} \tilde \mu_j \equiv 0\bmod 2\pi$.
By \Cref{lem:h0}, we have thus, for $\arg \delta = \alpha_0$:
$ \arg \frac{h_p(0)}{\delta^k} \equiv 0 \bmod 2\pi
$.
So $\arg h_p(0) \equiv k\arg \alpha_0 \bmod 2\pi$.
Hence
\[ \Im \frac{h_p(\delta)}{\delta^k} = c_p |\delta|^{-k} \sin (k(\alpha_0-\alpha)) + \Im( g_p(\delta) / \delta^{k-1}) \] 
with $c_p = |h_p(0)|$ and $h_p(\delta) = h_p(0) + \delta g_p(\delta)$.
 
Existence, uniqueness and analytic dependence on $s$, of a solution of \Cref{eq:2}, follow from the implicit function theorem applied to \Cref{eq:2} multiplied by $s^k$ with $s=|\delta|$. 
We check that it has an analytic extension through $s=0$, vanishing at $(\alpha,s)=(\alpha_0,0)$ and that the differential with respect to $\alpha$ at $s=0$ does not vanish.
One computes
\[
\frac{\partial}{\partial\alpha}s^k \Im \frac{h_p(\delta)}{\delta^k}
 \to -k c_p \cos(k(\alpha_0-\alpha)) 
\]
Note that $|k(\alpha_0-\alpha)|< 1/10$ so the $\cos$ term is bounded away from $0$.
The term $\int_{C_\eps} \frac{dz}{\omega_\eps(z)}$ is analytic in $\eps$ in a neighborhood of $\eps=0$. Multiplying by $s^k$, its derivative w.r.t.\ $\alpha$ uniformly vanishes at $s=0$.
\end{proof}

\begin{proof}[Proof of \Cref{item:cx:1,item:cx:2} of \Cref{thm:connection}]
We saw in \Cref{rem:inv} that for a given 
choice of bifurcation direction $\theta_0$ for $\eps$ and a choice of pair of vertices with the same imaginary part in the corresponding limit ideal period-gon, there are $4$ different pairs $(m_0, m_1)$ corresponding to these choices,
one for every possible pair of tangency types.
The $4$ corresponding values of $I_0=I_0[m_0,m_1] = \im \int_{C_{\eps=0}} \frac{dz}{\omega_0(z)}$ (see \Cref{eq:def:I0}) are generically different, and we can ensure at least $3$ different values by taking $r$ small enough: indeed we saw in \Cref{subsub:d} that for a $(\text{bottom},\text{top})$ tangency, $I_0\approx -2/kr^k$, for a $(\text{top},\text{bottom})$ $I_0\approx 2/kr^k$ and for a $(\text{bottom},\text{bottom})$ or $(\text{top},\text{top})$, $|I_0|\leq \eta/r^k$ where $\eta$ can be made as small as we wish by taking $r$ small.
When $\eps$ is small, $\Im\int_{C_\eps} \frac{dz}{\omega_\eps(z)}$ is close to $I_0$.
The analysis done in the proof of \Cref{prop:ut} carries out here and shows that as $\alpha=\arg\delta$ increases in $[\alpha_0-1/10K',\alpha_0+1/10K']$ while $s=|\delta|$ remains constant, the quantity $\im \left (\sum_{j=0}^{p-1} \mu_j(\delta) \right)$ decreases. It will thus be equal to the value of $\Im\int_{C_\eps} \frac{dz}{\omega_\eps(z)}$ for the $(\text{top},\text{bottom})$ type first and for the $(\text{bottom},\text{top})$ last, while the other two types happen in between.

From the geometric model, for $\arg\delta \in [\alpha_0-1/10K',\alpha_0+1/10K']$ and $\eps=\delta^{k+1}$, then $\eps \in \cal C_{\phi_a,\phi_r}$ if and only if the right facing part of $C'_{-1/2}$ and the left facing part of $C'_{p-1/2}$ have points with the same imaginary part.
This happens if and only if one eyelet has a topmost or/and a bottomost tangency point between the topmost and bottommost tangency point of the other one.
We know from the previous statements that for each type pair and $s$ fixed small enough there is only one value of $\alpha$ close to $\alpha_0$ for which \cref{eq:2} holds, and from the proof of \Cref{thm:erdt}, that the quantity in this equation represents the difference of imaginary parts of the images in the model of the tangency points $T_\phi(\eps)$ and $T_{\phi'}(\eps)$.
It follows that $\cal C_{\phi_a,\phi_r}$ is the space between the $(\text{bottom},\text{top})$ type and the $(\text{top},\text{bottom})$ one.
\end{proof}

\begin{proof}[Proof of \Cref{item:cx:width} of \Cref{thm:connection}]
Let $(m_0,m_1)$ be the tangency point indices for one of the curves bounding $\cal C_{\phi_a,\phi_r}$ and $(m'_0,m'_1)$ the ones for the other curve.
Going back to the equation $\im\frac{h_p(\delta)}{\delta^k} = -I_0[m_0,m_1] + O(\eps)$
with $h_p(\delta)$ tending to $h_p(0) = \frac{2\pi i}{(k+1)}\sum_{j=0}^{p-1}\xi^j \neq 0$ as $\delta\to 0$, with $\xi = e^{2\pi i/(k+1)}$, and $\eps = \delta^{k+1}$,
we get that the angle span of $\cal C_{\phi_a,\phi_r}$ is equivalent as $\eps\to 0$ to
\[\left|\frac{I[m_0,m_1]-I[m'_0,m'_1]}{h_p'(0)}\right| |\delta|^k\]
(we saw previously that $I[m_0,m_1]-I[m'_0,m'_1]$ does not vanish).
\end{proof}

\begin{remark}
Though not needed, we find interesting to provide an alternate approach to \Cref{thm:connection}, which we will not detail completely.
As we mentioned, the sum of all periods $A=\sum_{j=0}^{k}\mu_j$ may fail to be equal to $0$.
One idea is to introduce \emph{pseudoperiods} $\hat\mu_j =\mu_j - \frac{A}{k+1}$, so that $\sum_{j=0}^{k}\hat\mu_j = 0$, and a \emph{pseudoperiod-gon} of vertices $\hat v_{j+1/2}$ such that $\hat v_{j+1/2}-\hat v_{j-1/2} = \hat \mu_j$ and $\sum_{j=0}\hat v_j = 0$.
Consider the Laurent series expansion for $|\eps|<|z|<r$
\[\frac{1}{\omega_\eps(z)} = \sum_{n\in\Z} a_n(\eps) z^n.\]
It is possible to place the eyelets using on $C_{j+1/2}$ the antiderivative
\[\hat v_{j+1/2} + a_{-1} \log_p \frac{z}{e^{i\theta_{j+1/2}(\eps)}} + \sum_{n\in\Z^*} a_n(\eps) \frac{z^{n+1}}{n+1}\]
where $\log_p$ is the principal branch of the logarithm and $\theta_{j+1/2}(\eps)$ is the argument of the ideal period-gon vertex $\tilde v_{j+1/2}$.
This is why we may call $\tilde v_{j+1/2}$ an \emph{eyelet pseudocenter}.
Let us show how for instance this allows for an interesting proof of \Cref{item:cx:1} of \Cref{thm:connection}: As $\arg \eps$ varies near $\theta_0$ while $|\eps|$ is kept constant, the position of the pseudocenter vertices $\hat v_{j+1/2}$ vary fast, compared to the position of the topmost/bottommost points of the eyelets relative to the center of the eyelet.
It follows that for parameters close to a bifurcation parameter, increasing $\arg \eps$ while keeping $|\eps|$ fixed, we will see the right eyelet $C'_{j_1+1/2}$ plunge downwards (in a direction where the imaginary part decreases) while the left eyelet $C'_{j_0+1/2}$ ascends, at a much higher speed than the topmost and bottomost tangency points move relative to $\hat v_{j+1/2}$.
It follows that, fixing $|\eps|$ to a small value, in an interval of angles of width $\pi/2k$ and centered on a bifurcation direction, there is a unique parameter for which a horizontal segment can link the bottommost horizontal point of the left eyelet to the topmost of the right eyelet.
Moreover, this must happen before the value of $\arg\eps$ for which topmost and bottommost are permuted.
\end{remark}

\subsubsection{Proof of \Cref{prop:bif2:6}}

Consider
\[ \omega_\eps = z^{k+1}+a z^{k+2} - \eps + a' \eps z + O(z^{k+3},\eps z^2,\eps^2)
.\]
We know that the singularity $z_0(\delta)$ expands as $z_0(\delta)  = \delta + b_1 \delta^2+O(\delta^3)$ and substituting in $\tilde \omega_\eps(z_0)=0$ we get
that
\[ b_1 = -\frac{a'+a}{k+1}
.\]
Recall $\lambda_0(\delta) = \omega_\eps'(z_0)$. 
Since $\omega_\eps'(z) = (k+1)z^k +a(k+2)z^{k+1} +a'\eps + O(z^{k+2},\eps z,\eps^2)$
we get
\[ \lambda_0(\delta) = (k+1)\delta^k + \Big( k(k+1)b_1 +a(k+2) +a' \Big) \delta^{k+1} + O(\delta^{k+2})
,\]
hence
\[\mu_0(\delta) = \frac{2\pi i}{k+1} \delta^{-k} \Big( 1 + c\delta + O(\delta^2) \Big)
,\]
with
\[c = \frac{(k-1)a'-2a}{k+1}.\]
We recall that
\[ \mu_j(\delta) = \mu_0(\xi^j\delta) \text{ where } \xi = e^{2\pi i/(k+1)}
.\]
As usual, we denote $\alpha = \arg \delta$.
For a given bifurcation direction $\theta_0$ for $\eps=\delta^{k+1}$ (one of the $2k$ solutions $\theta$ of $\arg ((ie^{i\theta})^k) \equiv 0 \bmod \pi$) there are $k+1$ solutions $\alpha$ of $(k+1)\alpha = \theta_0$.
For each choice of $\alpha$, we chose in \Cref{conv:num} to label the ideal period-gon vertices as $v_{j-1/2}$ so that, in particular
\begin{equation}\label{eq:argvl}
\arg v_{-1/2} = \pi -\frac{\pi}{k+1}-k\alpha.
\end{equation}
Among $k+1$ possible choices of $\alpha$ for a given $\theta_0$, half of them, or slightly less (precisely $k/2$ if $k+1$ is odd and either $(k+1)/2$ or $(k-1)/2$ if $k+1$ is even), are such that, $v_{-1/2}$ is the leftmost point of a horizontal segment linking it to $v_{p-1/2}$ for some $0<p<k+1$.
(This is coherent with \Cref{conv:delta:p}.) 
We call such $\alpha$ the \emph{distinguished lifted bifurcation direction}.
Given $\theta_0$, denote by $v_{bl}$ the bottommost ideal period-gon vertex strictly to the left of the imaginary axis. Then $v_{bl}$ may make an angle of either $\pi/(k+1)$ or $2\pi/(k+1)$ with the imaginary axis.
We get
\begin{dseries*}
\begin{math}
\arg v_{bl} = -\frac{\pi}{2} -\frac{\pi}{k+1}
\end{math}
(Case 1), or
\begin{math}
\arg v_{bl} = -\frac{\pi}{2} -\frac{2\pi}{k+1} 
\end{math}
(Case 2),
\end{dseries*}
and this means that, denoting by $\alpha_0$ the unique value of $\alpha\in\R/2\pi\Z$ satisfying \Cref{eq:argvl} for $v_{-1/2} = v_{bl}$ and $(k+1)\alpha \equiv \theta_0 \bmod 2\pi$, we have
\begin{dseries*}
\begin{math}
k\alpha_0 \equiv -\frac{\pi}{2} \bmod 2\pi
\end{math}
(Case 1) or
\begin{math}
k\alpha_0 \equiv -\frac{\pi}{2} + \frac{\pi}{k+1} \bmod 2\pi
\end{math}
(Case 2).
\end{dseries*}
Hence
\begin{dseries*}
\begin{math}
\alpha_0 \equiv \theta_0 + \frac{\pi}{2} \bmod 2\pi
\end{math}
(Case 1) or
\begin{math}
\alpha_0 \equiv \theta_0 + \frac{\pi}{2} - \frac{\pi}{k+1} \bmod 2\pi
\end{math}
(Case 2).
\end{dseries*}
The $2k$ values of $\alpha_0$ satisfying these equations are in bijection by $\alpha\mapsto (k+1)\alpha$ with the $2k$ bifurcation directions $\theta_0$: let us temporarily call $\alpha_0$ the \emph{preferred lift} of $\theta_0$.

We now consider a bifurcation direction $\theta_0$ and its preferred lift $\alpha_0$.
The distinguished lifted bifurcations directions associated to $\theta_0$ are of the form
$\alpha_q := \alpha_0 - q \frac{2\pi}{k+1}$
with $q\in\N$, $0\leq q$ and
\begin{dseries*}
\begin{math}
\frac{\pi}{k+1} + q\frac{2\pi}{k+1}<\pi
\end{math}
(Case 1), or
\begin{math}
\frac{2\pi}{k+1} + q\frac{2\pi}{k+1}<\pi
\end{math}
(Case 2).
\end{dseries*}
Let $\theta$ close to $\theta_0$ be the argument of $\eps$ in the sense that $|\theta-\theta_0|<1/10k$, let $\delta$ be its unique lift by  $\delta^{k+1}=\eps$ such that $\alpha=\arg \delta$ is close to $\alpha_0$ in the sense that $|\alpha-\alpha_0|<1/10k(k+1)$.
The other lifts of $\eps$ are $\xi^{-q}\delta$ for $q$ satisfying the inequalities above.

The corresponding index $p$ (such that $v_{-1/2}$ and $v_{p-1/2}$ have the same imaginary part) is
\begin{dseries*}
\begin{math}
p=2q+1
\end{math}
(Case 1) or
\begin{math}
p=2q+2
\end{math}
(Case 2).
\end{dseries*}
Note that the inequalities above are equivalent to $p < k+1$.

The equation of the bifurcation curve is then 
\begin{equation}\label{eq:3}
\Im F_p(\delta) = -I[m_0,m_1](\eps),
\end{equation}
with
\[ 
F_q(\delta) :=  \sum_{j=0}^{p-1} {\mu_j(\xi^{-q}\delta)}
\text{ and } I[m_0,m_1](\eps) = \Im \int \frac{dz}{\omega_\eps(z)},
\]
where the integral is taken along an arc of $\partial B(0,r)$ followed \emph{clockwise} from $T_{\phi_{m_1}}(\eps)$ to $T_{\phi_{m_0}}(\eps)$, where the $T_{\phi_m}(\eps)$ denote the tangency points and where $m_0$ and $m_1$ each are one of the two possibilities for the given bifurcation direction and ideal period-gon vertex pair, see \Cref{rem:inv}.

Now, using $\xi^{-k}=\xi$ and the expansion of $\mu(\delta)$, a computation gives
\[ \sum_{j=0}^{p-1} {\mu_j(\xi^{-q}\delta)} =
\frac{2\pi i}{k+1} \delta^{-k} \Big(\sum_{j=0}^{p-1} \xi^{j-q} + c\delta\sum_{j=0}^{p-1} \xi^{2(j-q)} + O(\delta^2) \Big)
.\]
In Case~1 we can transform the expression using classical trigonometric identities into
\[ F_q(\delta) =
\frac{2\pi i}{k+1} \delta^{-k}
\frac{\sin\frac{p\pi}{k+1}}{\sin\frac{\pi}{k+1}}
\left( 1+ c\delta \frac{\cos\frac{p\pi}{k+1}}{\cos\frac{\pi}{k+1}}
+ O(\delta^2) \right)
\]
and in Case~2, denoting $\zeta = e^{\frac{\pi i}{k+1}}$:
\[ F_q(\delta) =
\frac{2\pi i}{k+1} \delta^{-k} \zeta
\frac{\sin\frac{p\pi}{k+1}}{\sin\frac{\pi}{k+1}}
\left( 1+ c\delta \zeta \frac{\cos\frac{p\pi}{k+1}}{\cos\frac{\pi}{k+1}}
+ O(\delta^2) \right)
.\]

We claim that if $\Im (ce^{i\alpha_0}) \neq 0$ in Case~1 or $\Im (ce^{i\alpha_0}\zeta) \neq 0$ in Case~2, then the bifurcation arcs for different values of $q$ are disjoint sufficiently close to $0$.
If $k=1$ there is nothing to prove as there can only be one value of $p$, so we assume $k>1$.
Assume $q'\neq q$ and consider the associated $p'=2q'+1$.
Let $S_p = 2\pi\sin\frac{p\pi}{k+1} / (k+1)\sin\frac{\pi}{k+1}$ and $\Gamma_p = \cos\frac{p\pi}{k+1} / \cos\frac{\pi}{k+1}$, which are both real constants.
The estimate above gives that 
\begin{align*}
\frac{\Im F_{q'}(\delta)}{S_{p'}} - \frac{\Im F_{q}(\delta)}{S_p} & = (\Gamma_{p'}-\Gamma_{p}) \Im\left( i c\delta^{1-k}  \right) + O(|\delta|^{2-k}) \\
& = (\Gamma_{p'}-\Gamma_{p}) \Im\left( i e^{-ik\alpha} c e^{i\alpha}  \right) |\delta|^{1-k} + O(|\delta|^{2-k}).
\end{align*}
As $\delta \to 0$, we have $\alpha\tend \alpha_0$.
We also know that $ie^{-ik\alpha}\tend \pm 1$; this can be checked by a direct computation but also follows from the following facts: $F_q(\delta)\sim S_p i \delta^{-k}$ while $S_p$ is real, $\Im F_q(\delta)$ remains bounded (it is equal to $-I[m_0,m_1](\delta)$) and $|\delta|^{1-k} \to +\infty$).
So $\Im\left( i e^{-ik\alpha} c e^{i\alpha}  \right)$ tends to a non-zero constant because of the assumption about $c$.
Since moreover $\Gamma_{p'}\neq \Gamma_p$ ($p\pi/(k+1)$ and $p'\pi/(k+1)$ are different and both belong to $[0,\pi]$ so their cosines are different), it follows that
\[ \left|\frac{\Im F_{q'}(\delta)}{S_{p'}} - \frac{\Im F_{q}(\delta)}{S_p} \right| \to +\infty
.\]
If $\delta$ were a bifurcation parameter both for $q$ and $q'$, then the difference $\Im F_q(\delta)$ and $\Im F_{q'}(\delta)$ would stay bounded (they equal respectively $I[m_0,m_1](\delta)$ and $I[m'_0,m'_1](\delta)$), which leads to a contradiction.
The proof in Case~2 is similar, with the condition on $c$ being replaced by $\Im(c \zeta e^{i\alpha_0}) \neq 0$.
Note that the union of all these conditions on $c$ for Case~1 and Case~2 and all the possible values of $\alpha_0$, is the complement of a union of straight lines in $\C$, so it is generic.

There remains to treat arcs coincidence for a given value of $q$, for which we recall that there are $4$ possible arcs, given by $4$ possible choices of the pair $(m_0,m_1)$.
We already saw that the only curves that could be identical are the $(\text{top},\text{top})$ and (bottom, bottom) types.
Let us check that these types generically give disjoint curves.
Note that $m_0$ and $m_1$ have the same parity for a $(\text{top},\text{top})$ type and the same other parity for a $(\text{bottom},\text{bottom})$ type. Let us first develop
\[\hat{\phi}_m :=\arg T_{\phi_m}(0)\]
as a function of $r$: since $\phi_m = \frac{\pi}{2k} + m\frac{\pi}{k} = \lim_{r\to 0} \hat{\phi}_m$, we get from the equation $\Re(\omega_0(re^{i\hat{\phi}_m})/re^{i\hat{\phi}_m}) = 0$
that
\[ \hat{\phi}_m = \phi_m + r\frac{(-1)^m\Im(a e^{i\phi_m})}{k} + O(r^2).\]

Integrating clockwise around $\partial B(0,r)$ and using that $e^{-ik\phi_{m_0}} = e^{-ik\phi_{m_1}} = (-i)(-1)^{m_0} =  (-i)(-1)^{m_1}$: we get for $k>1$
\begin{multline*}
\int_{re^{i\hat{\phi}_{m_1}}}^{re^{i\hat{\phi}_{m_0}}} \frac{dz}{\omega_0(z)} = \int_{\cdot}^{\cdot} (z^{-k-1}-az^{-k}+O(z^{-k+1}))dz =
\\
\left[ -\frac{z^{-k}}{k}+a\frac{z^{1-k}}{k-1}+O(z^{2-k})
\right]_{re^{i\hat{\phi}_{m_1}}}^{re^{i\hat{\phi}_{m_0}}}
= \frac{r^{-k}}{k}\left( e^{-ik\hat{\phi}_{m_0}} - e^{-ik\hat{\phi}_{m_1}}\right)
\\
 +\frac{a}{k-1} r^{1-k} \left( e^{i(1-k)\hat{\phi}_{m_0}} - e^{i(1-k)\hat{\phi}_{m_1}} \right) + O(r^{2-k}) = 
O(r^{2-k}) + 
\\
\left( \Im(a e^{i\phi_{m_1}}) - \Im(a e^{i\phi_{m_0}})
 +\frac{a}{k-1} (-i)(-1)^m\left( e^{i\phi_{m_0}} - e^{i\phi_{m_1}} \right)\right)r^{1-k}
.
\end{multline*}
Taking the imaginary part of this expression and using that $\Im(\Im(z))=0$ we get:
\[ I[m_0,m_1](0) = \frac{(-1)^{m_0}}{k-1} \Re \left(a( e^{i\phi_{m_1}} - e^{i\phi_{m_0}})\right) r^{1-k}+ O(r^{2-k}) 
.\]
By imposing the supplementary (real-linear) condition that
\[(-1)^{m_0}\Re \left(a( e^{i\phi_{m_1}} - e^{i\phi_{m_0}}) \right) \neq (-1)^{m'_0} \Re\left( a( e^{i\phi_{m'_1}} - e^{i\phi_{m'_0}} )\right)
\]
for every $m_0,m_1,m'_0,m'_1$ with $m_0$ and $m_1$ of one parity and $m'_0$ and $m'_1$ of the other parity, $m_0$ and $m'_0$ adjacent across an attracting axis and $m_1$ and $m'_1$ adjacent across a repelling axis, we know that $I[m_0,m_1](0) \neq I[m'_0,m'_1](0)$ for all $r$ small enough, which implies the two bifurcation arcs are disjoint.
This is equivalent to
$\Re \left(a( e^{i\phi_{m_1}} - e^{i\phi_{m_0}} + e^{i\phi_{m'_1}} - e^{i\phi_{m'_0}} )\right) \neq 0$
and the condition is generic since
$e^{i\phi_{m_1}} - e^{i\phi_{m_0}} + e^{i\phi_{m'_1}} - e^{i\phi_{m'_0}}$ cannot vanish: indeed it is equal to 
$(e^{i\frac{\phi_{m_1}+\phi_{m'_1}}{2}} - e^{i\frac{\phi_{m_0}+\phi_{m'_0}}{2}}) \cos \frac{\pi}{2k}$
so it can vanish only if $\frac{\phi_{m_1}+\phi_{m'_1}}{2} \equiv \frac{\phi_{m_0}+\phi_{m'_0}}{2} \bmod 2\pi $.
But these are the directions of a repelling axis and an attracting axis so they cannot coincide.

In the case $k=1$ the two curves in question correspond to the simpler equations $\mu_0(\delta)\in\R$ and $\mu_1(\delta)\in\R$.
We have $c=-a$, $\Im(\mu_0(\delta)) = \pi \Re(\delta^{-1}-a)+O(|\delta|)$ and  $\Im(\mu_1(\delta)) = \pi \Re(-\delta^{-1}-a)+O(|\delta|)$, and it is enough to require $\Re(a)\neq 0$ to ensure the disjointness of the two curves near $0$.
\hfill$\Box$

\section*{Acknowledgements}

We would like to thank Wolf Jung for stimulating discussions and for suggesting a simplification in \Cref{sub:ideal:case}.
We also thank an early reader for pointing us to the work of \cite{Ri}, and for many useful remarks.


\appendix


\section{The path method}\label{app:path}

Here we give an alternative proof of the converse implication in \Cref{thm:lambda_inv}, based on the \emph{path method}, attributed by Roussarie in \cite{Rou} to Mather (in another context) though it may date from before.
The form we use here is a re-interpretation of the form in \cite{Ma}, proof of Proposition~1.1.
The latter form is used in \cite{Ri}, proof of Proposition~5.10, to give a very general classification theorem of unfoldings of vector fields.\footnote{The path method allows Rib\'on to reduce the classification modulo change of variable and parameter to a simpler non-dynamical problem: consider the multi-dimensional space consisting of the phase space times the parameter space; there the eigenvalues (Rib\'on rather speaks of the residues) are defined on the variety of singular points and, in some circumstances, it is enough that these varieties be sent one to another so as to preserve the eigenvalues. In full generality the statement has some complications.}
This result is then used in the same article to give a \emph{formal} classification theorem of unfoldings of \emph{diffeomorphisms} in the most general case.
In our context, the use of the path method is easier.

\begin{proof}
Choose 
\[ \omega_{\eps,s} = (z^{k+1}-\eps) v_{\eps,s}(z)
\]
for $s\in[0,1]$ to be any interpolation\footnote{A linear interpolation works and will be not only $C^2$ but real-analytic.} between $\omega_{\eps} = \omega_{\eps,0}$ and $\tilde\omega_{\eps} = \omega_{\eps,1}$ such that $v_{\eps,s}$ is holomorphic in $(z,\eps)$, depends in a $C^2$ way on $s$, stays bounded away from $0$ and such that the eigenvalues of $\omega_{\eps,s}$ at its singularities do not depend on $s$.
The last condition is possible since $\omega$ and $\tilde \omega$ have the same eigenvalues (and this amounts to forcing $v_{\eps,s}(\delta)$ to remain independent of $s$ when $\delta^{k+1}=\eps$).
The method will construct a conjugacy $g_{\eps,s}$ from $\omega_{\eps}=\omega_{\eps,0}$ to $\omega_{\eps,s}$ (uniformly in a neighborhood $B(0,\rho)\times B(0,r)$ of the origin in $(z,\eps)$-space) such that $g_{\eps,s}(0)=0$, by solving an O.D.E.\ involving $\frac{\partial}{\partial s}g_{\eps,s}$.

We recall that straightening coordinates of a vector field $dz/dt = \omega(z)$ are locally given by antiderivatives: $F=\int dt = \int dz/\omega(z)$.
More globally, such a function is likely to be multivalued near a singularity $a$. More precisely, winding once around $a$ in the positive direction adds $2\pi i \res$ where $\res$ is
the residue at $a$ of the meromorphic $1$-form $dz/\omega(z)$. 
If we introduce $F_{\eps,s}$ as an anti-derivative of $1/\omega_{\eps,s}$, it is a multivalued function on $B(0,r)\setminus \setof{z\in B(0,r)}{z^{k+1}=\eps}$.
There is a choice of additive constant, and we choose it so that a branch sends $0$ to $0$ for $\eps\neq 0$
(note that we choose not to normalize $F_{0,s}$).
For $\eps\neq 0$, $\partial F_{\eps,s}/\partial s$ is well-defined (takes only one value at a given $z$): indeed, the residues do not depend on $s$. It is a holomorphic function of $z$ and has erasable singularities at the singularities of the vector field: indeed $1/\omega_{\eps,s}$ is written as
\[ \frac{1}{\omega_{\eps,s}(z)} = h_{\eps,s}(z)+\sum_j \frac{\res_j}{z-z_j}
,\]
where $h_{\eps,s}$ is holomorphic in $z$ throughout $B(0,r)$, so
$\partial F_{\eps,s}/\partial s$ coincides for $\eps\neq 0$ with the antiderivative $Q_{\eps,s}$, w.r.t.\ $z$, of $\partial h_{\eps,s}/\partial s$ that vanishes at $z=0$ (because of the normalization).
Note that $Q_{0,s}$ is also defined and that $Q_{\eps,s}(z)$ depends $C^1$ on $s\in[0,1]$ and analytically on $(\eps,z)$.

For $\eps\neq 0$, a function $g_{\eps,s}$ conjugates $\omega_{\eps,0}$ to $\omega_{\eps,s}$ if and only if it is, locally, a translation in the respective straightening coordinate, i.e.\ $F_{\eps,s}\circ g_{\eps,s} - F_{\eps,0}$ is locally constant (w.r.t.\ $z$).
By our choice of normalization $F_{\eps,s}(0)=0$, the condition $g_{\eps,s}(0)=0$ imposes that this constant can only be $0$.
Differentiating with respect to $s$, this gives that
\begin{equation}\label{eq:edo1}
\Big(\frac{\partial }{\partial s} F_{\eps,s}\Big) \circ g_{\eps,s} + \frac{1}{\omega_{\eps,s}\circ g_{\eps,s}} \frac{\partial }{\partial s} g_{\eps,s} = 0
.
\end{equation}
\Cref{eq:edo1} is equivalent to 
\begin{equation}\label{eq:edo2}
\frac{\partial }{\partial s} g_{\eps,s} = -\Big(\omega_{\eps,s}\times Q_{\eps,s}\Big) \circ g_{\eps,s}
,
\end{equation}
which has the advantage of also making sense for $\eps=0$.
Conversely, if $g_{\eps,0}$ is the identity and \cref{eq:edo2} holds for all $s$ and $\eps$, then $g_{\eps,s}(0)=0$ and for all $\eps\neq 0$, $F_{\eps,s}\circ g_{\eps,s} - F_{\eps,0}=0$.
So $g_{\eps,s}$ sends $\omega_{\eps,0}$ to $\omega_{\eps,s}$ and, by continuity, this holds for $\eps=0$ too.
This O.D.E.\ can be considered on the space of $(\eps,z,s)\in B(0,\rho)\times B(0,r)\times [0,1]$ and encodes the flow with time parameter $s$ of a non-autonomous vector field in the $(\eps,z)$ space.
As such, it has a solution by the Cauchy-Lipschitz theorem.\footnote{A.k.a.\ Picard–Lindelöf theorem.}
\end{proof}

In this version of the proof, we also obtain a $0$-fixing change of variable and parameter.


\printbibliography

@misc{CR0,
  title={Generic 1-parameter pertubations of a vector field with a singular point of codimension k}, 
  author={Arnaud Chéritat and Christiane Rousseau},
  year={2017},
  eprint={1701.03276v1},
  archivePrefix={arXiv},
  primaryClass={math.DS},
  url={https://arxiv.org/abs/1701.03276v1}, 
}

@article{CMR,
 author = {Christopher, C. and Marde{\v{s}}i{\'c}, P. and Rousseau, C.},
 title = {Normalizability, synchronicity, and relative exactness for vector fields in {\(\mathbb {C}^2\)}},
 fjournal = {Journal of Dynamical and Control Systems},
 journal = {J. Dyn. Control Syst.},
 issn = {1079-2724},
 volume = {10},
 number = {4},
 pages = {501--525},
 year = {2004},
 language = {English},
 doi = {10.1023/B:JODS.0000045362.42314.60},
 zbMATH = {2153916},
 Zbl = {1068.37030}
}

@unpublished{DES,
  author = {Douady, A. and Estrada, F. and Sentenac, P.},
  title = {Champs de vecteurs polynomiaux sur {{\(\mathbb{C}\)}}},
  year = {2005},
  note = {Preprint, Orsay},
}

@article{Duff,
 author = {Duff, G. F. D.},
 title = {Limit-cycles and rotated vector fields},
 fjournal = {Annals of Mathematics. Second Series},
 journal = {Ann. Math. (2)},
 issn = {0003-486X},
 volume = {57},
 pages = {15--31},
 year = {1953},
 language = {English},
 doi = {10.2307/1969724},
 zbMATH = {3078792},
 Zbl = {0050.09103}
}

@article{KR,
 author = {Klime{\v{s}}, Martin and Rousseau, Christiane},
 title = {Generic {{\(2\)}}-parameter perturbations of parabolic singular points of vector fields in {{\(\mathbb {C}\)}}},
 fjournal = {Conformal Geometry and Dynamics},
 journal = {Conform. Geom. Dyn.},
 issn = {1088-4173},
 volume = {22},
 pages = {141--184},
 year = {2018},
 language = {English},
 doi = {10.1090/ecgd/325},
 zbMATH = {6952345},
 Zbl = {1403.37057}
}

@article{K,
 author = {Kostov, V. P.},
 title = {Versal deformations of differential forms of degree {{\(\alpha\)}} on a line},
 fjournal = {Functional Analysis and its Applications},
 journal = {Funct. Anal. Appl.},
 issn = {0016-2663},
 volume = {18},
 pages = {335--337},
 year = {1984},
 doi = {10.1007/BF01083698},
 zbMATH = {3915260},
 Zbl = {0573.58002}
}

@misc{Ma,
 author = {Martinet, Jean},
 title = {Remarks about saddle-node bifurcation in the complex domain},
 year = {1987},
 language = {French},
 howpublished = {Singularit{\'e}s d'{\'e}quations diff{\'e}rentielles, {Journ}. {Dijon}/{France} 1985, {Ast{\'e}risque} 150-151, 131-149 (1987).},
 zbMATH = {4069883},
 Zbl = {0655.58025}
}

@article{Ri,
 author = {Rib{\'o}n, Javier},
 title = {Formal classification of unfoldings of parabolic diffeomorphisms},
 fjournal = {Ergodic Theory and Dynamical Systems},
 journal = {Ergodic Theory Dyn. Syst.},
 issn = {0143-3857},
 volume = {28},
 number = {4},
 pages = {1323--1365},
 year = {2008},
 doi = {10.1017/S0143385707000867},
 zbMATH = {5353247},
 Zbl = {1153.37026}
}

@book{Rou,
 author = {Roussarie, Robert},
 title = {Modèles locaux de champs et de formes},
 fseries = {Ast{\'e}risque},
 series = {Ast{\'e}risque},
 issn = {0303-1179},
 volume = {30},
 year = {1976},
 publisher = {Soci{\'e}t{\'e} Math{\'e}matique de France (SMF), Paris},
 zbMATH = {3513028},
 Zbl = {0327.57017}
}

@article{RT,
 author = {Rousseau, Christiane and Teyssier, Loïc},
 title = {Analytical moduli for unfoldings of saddle-node vector fields},
 fjournal = {Moscow Mathematical Journal},
 journal = {Mosc. Math. J.},
 issn = {1609-3321},
 volume = {8},
 number = {3},
 pages = {547--614},
 year = {2008},
 url = {www.ams.org/distribution/mmj/vol8-3-2008/abst8-3-2008.html},
 zbMATH = {5503154},
 Zbl = {1165.37016}
}

@phdthesis{Teyssier,
 author = {Teyssier, Lo{\"i}c},
 title = {{\'E}quation homologique et classification analytique des germes de champs de vecteurs holomorphes de type n{\oe }ud-col},
 school = {N°2869, Universit{\'e} de Rennes},
 year = {2003},
}


\end{document}